\documentclass[11pt]{amsart}
\usepackage{times}
\usepackage{amsfonts}
\usepackage{enumerate}
\usepackage{amsthm}
\usepackage{amsmath}
\usepackage{amssymb}
\usepackage{enumerate}
\usepackage{tikz-cd}
\advance \topmargin by -\headheight
\advance \topmargin by -\headsep
\evensidemargin \oddsidemargin
\bibstyle{et}

\newtheorem{theorem}{Theorem}[section]
\newtheorem{lemma}[theorem]{Lemma}
\newtheorem{corollary}[theorem]{Corollary}
\newtheorem{proposition}[theorem]{Proposition}

\newtheorem*{kaplansky}{Kaplansky's Conjecture}
\newtheorem*{WIFP}{Weyl's Inequality for Positive Operators}
\newtheorem*{linnell}{Linnell's $L^2$ Conjecture}
\newtheorem{definition}[theorem]{Definition}

\theoremstyle{definition}
\newtheorem{example}{Example}[section]
\newtheorem{remark}[theorem]{Remark}
\begin{document}

\title[The Rank Theorem and $L^2$-invariants in Free Entropy: Global Upper Bounds]{The Rank Theorem and $L^2$-invariants in Free Entropy: \\ Global Upper Bounds}

\author{Kenley Jung}
\begin{abstract} Using an analogy with the rank theorem in differential geometry, it is shown that for a finite $n$-tuple $X$ in a tracial von Neumann algebra and any finite $m$-tuple $F$ of $*$-polynomials in $n$ noncommuting indeterminates, 
\begin{eqnarray*} 
\delta_0(X) & \leq & \text{Nullity}(D^sF(X)) + \delta_0(F(X):X)
\end{eqnarray*}
where $\delta_0$ is the (modified) microstates free entropy dimension and $D^sF(X)$ is a kind of derivative of $F$ evaluated at $X$.   When $F(X) =0$ and $|D^sF(X)|$ has nonzero Fuglede-Kadison-L{\"u}ck determinant, then $X$ is $\alpha$-bounded in the sense of \cite{j3} where $\alpha = \text{Nullity}(D^sF(X))$.  Using Linnell's $L^2$ integral domain results in \cite{l} as well as Elek and Szab{\'o}'s work on L{\"u}ck's determinant conjecture for sofic groups in \cite{es} the following result is proven.  Suppose $\Gamma$ is a sofic, left-orderable, discrete group with 2 generators and $\Gamma \neq \{0\}$.  The following conditions are equivalent:
\vspace{.01in}
\begin{enumerate}[(1)]
\item $\Gamma \not\simeq \mathbb F_2$. 
\item $L(\Gamma) \not\simeq L(\mathbb F_2)$.
\item $L(\Gamma)$ is strongly $1$-bounded.
\item $\delta_0(X) = 1$ for any finite set of generators $X$ for $L(\Gamma)$.
\end{enumerate}
\vspace{.1in}
From Brodski{\u i} and Howie's results on local indicability (\cite{b}, \cite{h}), it follows that a sofic, torsion-free, one-relator group von Neumann algebra on two generators with nontrivial relator is strongly $1$-bounded.   It also follows from the residual solvability of the positive one relator groups (\cite{baum}) that a one-relator group von Neumann algebra on two generators whose relator is a nontrivial, positive, non-proper word in the generators is strongly $1$-bounded.
\end{abstract}
\maketitle
\section{Introduction}

Consider the space of solutions, $M$, to the equation $F(x)=0$ where $F:\mathbb R^n \rightarrow \mathbb R^m$ is a smooth function.  To understand the geometry of $M$ one could start by computing the derivative $DF(x)$.  Under suitable regularity conditions (e.g., $DF(x)$ is surjective for all solutions) $M$ is a smooth manifold of dimension $\alpha=\text{Nullity}(DF(x))$.  This result is one of several closely related observations in differential geometry that I'll collectively refer to as the rank theorem (the standard version states that $F$ is equivalent to a projection of rank equal to $\text{Rank}(DF(x))$).  If one further assumes that $M$ is compact, then the smoothness of $F$ implies that the $\alpha$-Hausdorff measure of $M$ is finite.

These observations have a direct bearing on von Neumann algebras and free probability.  I'll spell out this analogy by translating each of the quantities above into an appropriate operator algebra counterpart.  $F$ will now be an $m$-tuple of $*$-polynomials in $n$ indeterminates and $X$ will denote an $n$-tuple of generators for a tracial von Neumann algebra $(M,\varphi)$ such that $F(X)=0$.  Associated to $X$ is a free probability quantity defined in \cite{v2} called the (modified) microstates free entropy dimension $\delta_0(X)$.  $\delta_0(X)$ is an analogue of Minkowski dimension and is defined through an asymptotic process involving $\epsilon$-entropy, i.e., the minimum number $\epsilon$-balls required to cover a metric space.  It  will replace topological dimension in this analogy.  Differentiation in operator algebras can be expressed through derivations and tensor products and these devices can used to construct a derivative $D^sF(X)$ which will be a matrix with entries in $M \otimes M^{op}$ - another tracial von Neumann algebra.  One can speak of the nullity and rank of $D^sF(X)$ by taking (unnormalized) traces of the projections onto the kernel or cokernel of $D^sF(X)$.  Thus, setting $\alpha = \text{Nullity}(D^sF(X))$ one might expect 
\begin{eqnarray*}
\delta_0(X) \leq \alpha.
\end{eqnarray*}
Notice here that this is an inequality instead of an equality, as $X$ is not defined solely through the relation, but only required to satisfy it.  To complete the analogy, just as there is a suitable replacement for dimension, so too is there a kind of Hausdorff entropy $\mathbb H^{\alpha}$ for $X$ and if one believes that $X$ is 'compact' in some suitable way, then 
\begin{eqnarray*}
\mathbb H^{\alpha}(X) < \infty.
\end{eqnarray*}

This paper is concerned with proving the above von Neumann algebra counterparts.  The analogy is simple.  The proofs are not.  Before outlining some of the difficulties involved, I'll discuss the free probability results, their applications to group von Neumann algebras, and the connections to $L^2$-invariants and combinatorial group theory.

\subsection{Group von Neumann Algebras Results, $L^2$-invariants, and Combinatorial Group Theory}

The dimension inequality in terms of the nullity has a slightly more general expression than discussed above.  Suppose $F$, $X$ are as before, but that the condition $F(X)=0$ is removed.  It will be shown (Section 4) that for a suitable notion of the derivative, denoted by $D^s$,
\begin{eqnarray*}
\delta_0(X) & \leq & \text{Nullity}(D^sF(X)) + \delta_0(F(X):X)
\end{eqnarray*}
where here $\delta_0(:)$ is a kind of relative free entropy dimension (\cite{v2}).  If $F(X)=0$, then $\delta_0(F(X):X)=0$ and the above reduces to the free entropy dimension inequality stated above.  The inequality is connected to a number of von Neumann algebra applications of free entropy dimension involving normalizers/commutators (\cite{v2}, \cite{gs}), eigenvalue results for maximal free entropy dimension tuples (\cite{msm}, \cite{ss}), and $L^2$-Betti number computations (\cite{cs}, \cite{dl}).  The connections are discussed at length in the examples of Section 4.

The argument for the Hausdorff entropy bound requires additional assumptions on $F$ and $X$ including the condition $F(X)=0$ (see Remark 6.10 about dropping this condition). It will also require that $D^sF(X) \in M_{2n}(M \otimes M^{op})$ has\textbf{ geometric decay}, a property equivalent to $|D^sF(X)|$ having nonzero Fuglede-Kadison-L{\"u}ck determinant.  \cite{j2} introduced a notion of Hausdorff $\alpha$-entropy $\mathbb H^{\alpha}$ and covering $\alpha$-entropy $\mathbb K^{\alpha}$.  In general $\mathbb H^{\alpha} \leq \mathbb K^{\alpha}$, although it is not known whether the inequality is strict.  Using the covering entropy I will prove a slightly stronger form of the entropy inequality (Section 6), namely that 
\begin{eqnarray*}
\mathbb H^{\alpha}(X) \leq \mathbb K^{\alpha}(X) <\infty,
\end{eqnarray*}
for $\alpha = \text{Nullity}(D^sF(X))$.   When $X$ satisfies such an inequality $X$ is said to be $\alpha$-bounded.

When $\alpha=1$, $\mathbb K^{\alpha}$ has particular significance.  Motivated by free probability applications to von Neumann algebras in \cite{v1}, \cite{g}, and \cite{gs}, and the fundamental work of Besicovitch (\cite{besicovitch1}, \cite{besicovitch2}, and reference \cite{falconer}) on regular fractal sets of finite Hausdorff $1$-measure, I showed in \cite{j3} that if $X$ is a finite tuple of self-adjoint elements such that there exists an element $x \in X$ with finite free entropy and such that $\mathbb K^1(X) < \infty$, then $\delta_0$ is an invariant for the von Neumann algebra generated by $X$.  To be clear, this means that any other finite tuple $Y$ of self-adjoint generators for the von Neumann algebra generated by $X$ satisfies the condition that $\mathbb K^1(Y) < \infty$ (and thus $\delta_0(Y) \leq 1$).  A von Neumann algebra with such a generating set was called strongly $1$-bounded.  Thus, by computations made in \cite{v1}, the free group factors cannot be generated by a strongly $1$-bounded von Neumann algebra.  Applications of this included showing that a union of two strongly $1$-bounded von Neumann algebras with a diffuse intersection generates a strongly $1$-bounded von Neumann algebra; also any finite generating set for the von Neumann algebra generated by a strongly $1$-bounded von Neumann algebra along with any subset of its normalizers generates a strongly $1$-bounded von Neumann algebra.  This work also provided a way to unify and generalize the already established results of \cite{v2} and \cite{g}, which demonstrated that the free group factors have no Cartan subalgebra, and are prime, respectively.  Note that the results on normalizers and commutants can be obtained and significantly strengthened under much more general conditions through a variety of solidity/rigidity techniques introduced in \cite{o1}, \cite{o2}, \cite{op}, \cite{ipp}, among others.  \cite{va} is one place where the interested reader can read a more detailed account of this area.

The geometric decay condition and nullity computation are unwieldy to verify/compute in a general von Neumann algebra context without further assumptions (aside from commutators, skew-commutators,  normalizers, and skew-normalizers).  In the discrete group case, however, one can access $L^2$-theory (L{\"u}ck's determinant property, Linnell's $L^2$-property) as well as combinatorial group theory to guarantee these spectral derivative conditions.  Indeed, if $X$ is an $n$-tuple consisting of the canonical unitaries associated to a generating set of elements for a discrete group $\Gamma$, then $D^sF(X)$ will have geometric decay at $0$ whenever $\Gamma$ is sofic by a result of \cite{es}.  Moreover, if $\Gamma$ is also left orderable (i.e., has a left invariant linear ordering) and $F$ has a nontrivial (nonidentity) $*$-monomial, then by a result of \cite{l}, $D^sF(X)$ will have nullity no greater than $n-1$.  In particular when $n=2$ one has:

\begin{proposition} Suppose $\Gamma$ is a left orderable, sofic group with $2$ generators and $\Gamma \neq \{0\}$.  The following conditions are equivalent:
\begin{enumerate}[(1)]
\item $\Gamma \not\simeq \mathbb F_2$. 
\item $L(\Gamma) \not\simeq L(\mathbb F_2)$.
\item $L(\Gamma)$ is strongly $1$-bounded.
\item $\delta_0(X) = 1$ for any finite set of generators $X$ for $L(\Gamma)$.
\end{enumerate}
\end{proposition}
\noindent One can actually show this if left orderability in the above is replaced with a weaker condition (Corollary 7.8).

In essence, for such sofic group von Neumann algebras on two generators, the existence of a nontrivial relation collapses the microstate spaces of the generators into a kind of rectifiable curve of amenable von Neumann algebras (i.e., it's strongly $1$-bounded).  Putting this together with results of \cite{h} (\cite{b}) on the local indicability (residual solvability) of torsion-free (positive) one-relator groups yields:

\begin{corollary} Suppose $\Gamma$ is a torsion free, discrete group with a one relator presentation on $2$ generators.   If $\Gamma$ is sofic and the relator is nontrivial, then $L(\Gamma)$ is strongly $1$-bounded.  In particular, if the relator is nontrivial and positive, then $L(\Gamma)$ is strongly $1$-bounded.
\end{corollary}
The torsion-free condition can be replaced with an equivalent condition on the form of the relator, as stated in the abstract (see also Corollary 7.10).  Note that the torsion-free condition is necessary as a free product of two finite cyclic groups demonstrates.  Further von Neumann algebra results are presented in Sections 6 and 7.

\subsection{Technical Overview}
Recalling the proof of the Rank Theorem will give an idea of some of the challenges in establishing the analogy.  

From the outset there is a slight problem with comparing a microstates dimension with nullities in a von Neumann algebra context.  Dimension and the derivative in the geometric context (and thus the microstates dimension) take place in a real setting, but when one takes traces in von Neumann algebras, these occur in a complex setting.  Thus, a little care must be taken to extract the appropriate real dimension from the complex traces.   One can break the derivative up into its self-adjoint and skew-adjoint parts, keep track of its action on these spaces, and define the corresponding derivation.   A simple example will clarify the issues.

Consider the tracial von Neumann algebra of $k \times k$ complex matrices $M=M_k(\mathbb C)$, the map $F:M \rightarrow M$ defined by $F(x) = x^*x$, and the level set $L = \{ x \in M: F(x) = I\}$.  $L$ is just the Lie group of unitaries and its dimension is $k^2$, which is the same dimension as the space of the self-adjoint complex matrices.  Normalizing by $k^2$, the dimension of $L$ 'is' $1$.   This can easily be recovered in terms of the derivative and the von Neumann algebra.  The derivative of $F$ at $I$ is $DF(I) = I_M + J$ where $I_M$ is the identity operator on $M$ and $J$ is the conjugation map on $M$.  The kernel of $DF(I)$ should have dimension agreeing with the dimension of $L$, i.e., $1$.  Unfortunately the appearance of the conjugation operator $J$ (which is real linear and not complex linear) makes this situation slightly annoying if one wants to use it in a complex setting.  However, if one breaks up the action of $I_M$ and $J$ onto the self-adjoint and skew-adjoint portions of $M$, $M^{sa}$ and $M^{sk}$ (regarded as real subspaces), then in a matricially loose sense:
\begin{eqnarray*} DF(I) & = &   I_M +J \\
                                      & = &  \begin{bmatrix} I & 0 \\ 
                                                                          0 & I \\ 
                                                 \end{bmatrix} + 
                                                 \begin{bmatrix} I & 0 \\ 
                                                                          0 & -I \\ 
                                                 \end{bmatrix} \\
                                       & = & \begin{bmatrix} 2I & 0 \\ 
                                                                          0 & 0 \\ 
                                                 \end{bmatrix}.\\                               
\end{eqnarray*}
The nullity of $DF(I)$ in its form above as a $2 \times 2$ matrix with entries in $M_k(\mathbb C)$ should be the trace of the projection onto the orthogonal complement of the range, namely
\begin{eqnarray*}                     \begin{bmatrix} 0 & 0 \\ 
                                                                          0 & I \\
                                                \end{bmatrix}. 
\end{eqnarray*}
The trace of the matrix above 'is' $1$ as expected (by the normalization $1$ is equivalent to $k^2$ which is exactly the dimension of $L$) and its range consists of the set of skew-adjoints i.e., the tangent space of the unitaries.

The decomposition described above can be generalized to the derivative of any $*$-polynomial.  Notice that one could avoid these algebraic preliminaries by working exclusively in the self-adjoint context (as done initially in \cite{v1}, \cite{v2}).  However, I will need to work in the unitary case as well and creating one differential calculus which encompasses both cases seemed preferable.  The differential calculus rules for the self-adjoint and unitary cases follow from the general setup.  In particular, after a change of bases, the unitary calculus for $*$-monomials can be explicitly derived and this will have computational consequences in the group von Neumann algebra setting.   These arguments take place in Section 3. 

Assuming the linear algebra notions of rank and nullity for derivatives are in order, the manifolds argument proceeds to show that the solution space $F(x)=0$ is locally diffeomorphic to $\ker DF(x)$.  Here, one claims (again under some regularity conditions), that after a change of bases, $DF(x)$ is a projection.  Although microstates are finite dimensional matrices like $DF(x)$, they are asymptotic approximants to operators with singular values that can be more complex than a finite list of eigenvalues.  For example the spectrum may the entire unit interval.  This spectral complexity makes the differential geometry argument difficult to carry over.   It lead to the idea of \textbf{splitting the spectrum} by a continuous parameter $\alpha$.

Spectral splitting occurs in the setting of microstates and since the microstates are simply elements in Euclidean space, for simplicity I will describe spectral splitting in that context.  Suppose $0 \in E \subset B \subset \mathbb R^d$ with $B$ a ball of small radius and $f: B \rightarrow \mathbb R^m$ is a $C^1$-function.  For small $\beta >0$ define $Q_{\beta} = 1_{[0,\beta]}(|Df(0)|)$; this is the analogue of the kernel of $DF(x)$ in the differential geometry argument.  $Q_{\beta}^{\bot}$ is the analogue of the nondegenerate projection which represents $DF(x)$ after a change of bases.  Very roughly, by projecting onto $Q_{\beta}$ and $Q_{\beta}^{\bot}$, for any $\epsilon >0$
\begin{eqnarray*}
K_{\epsilon}(E) & \leq & K_{\epsilon}(Q_{\beta}(E)) \cdot K_{\epsilon}(f(E))
\end{eqnarray*}
where $K_{\epsilon}$ is the minimum number of elements in an $\epsilon$-cover of $E$ (the inequality is not true as stated; Lemma 4.2 is the exact statement).  If one takes an appropriate limit as $\epsilon \rightarrow 0$, then 
\begin{eqnarray*}
\dim(E) & \leq & \dim(Q_{\beta}(E)) + \dim(f(E))
\end{eqnarray*}
which is very close to the initial free entropy dimension inequality ($\dim$ is a heuristic term that isn't rigorously defined here).   The one slight problem is that $Q_{\beta}$ is a projection onto a subspace which might strictly contain the kernel.  But since the above holds for any $\beta$ one can take a limit as $\beta \rightarrow 0$.  Then $Q_{\beta}$ converges to $Q_0$, the projection onto the kernel of $Df(0)$ and one arrives at the free entropy dimension inequality stated in the previous section.  Notice that these estimates will only work in a small enough neighborhood $B$ because one needs uniform control of the derivative.  Using these inequalities, some general covering estimates, and putting this into the microstates machinery yields the stated free entropy dimension inequality.  The rigorous arguments take place in Section 4.

While a single spectral split yields a free entropy dimension bound in terms of the nullity of the derivative, it alone cannot establish the entropy estimate because of the lingering $\beta$ dimension.  To see why, it's helpful to think of Hausdorff/covering entropy as a kind of natural measure associated to the dimension of the set.  What happens when one takes the natural measure of a set and uses it to evaluate a set with a slightly larger dimension?  For example, the Hausdorff 2-measure of a 3-dimensional solid is infinite regardless of the size of the 3-dimensional solid.  This mismatch will occur for any $n$-dimensional manifold and $(n-1)$-dimensional Hausdorff measure.  So too will it occur with the lingering $\beta$-dimension - trying to show that $\mathbb K^{\alpha}(X)$ is finite for $\alpha = \text{Nullity}(D^sF(X))$ on a space which has dimension $\alpha$ plus a subspace of dimension dependent on $\beta$ and $X$ looks hopeless.   

The solution I'll present quantifies the relationship between the $\epsilon$-entropy and the excess dimension (Theorem 6.8).  Basically, any $\epsilon$-neighborhood of the microstate space looks like a subspace of dimension slightly larger than the right dimension.  The difference is quantified by $\beta$ and $D^sF(X)$.   The smaller one makes the $\epsilon$-neighborhood, the closer its entropy is to that of an $\epsilon$-ball in a subspace of normalized dimension $\alpha =\text{Nullity}(D^sF(X))$.   Thus, as $\epsilon \rightarrow 0$ the exponential growth rate of the local entropy decreases to $\alpha$, the 'right' rate.  On the other hand as $\epsilon$ shrinks, the number of balls required to cover the microstate space increases; thus, as $\epsilon \rightarrow 0$ the global entropy increases.  Quantifying the rate of the local entropy decrease and global entropy increase in terms of $\epsilon$ is crucial.  It is here that the geometric decay in the assumption, or determinant class (in the sense of L{\"u}ck), becomes indispensable.  Once this relationship is established, one can perform \textbf{iterative spectral splits} on a geometric scale of $\epsilon$ where $\epsilon$ remains fixed.  The approach is different from the one-time split of the spectrum in that for each split, one accounts for the entropy of both the degenerate and nondegenerate portions, pairing the gain with the loss.  Iterating the spectral splitting procedure and keeping tally of the cumulative entropy resulting from each iteration yields a finite upper bound.  This kind of argument will require significantly more complicated covering estimates than those presented in the dimension argument. It will include Chebyshev estimates, Rogers's asymptotic bounds for covers of balls by balls \cite{rogers}, coverings by sets of negligible dimension (bindings and fringes), and estimates using the quasi-norm spaces $L^p(M)$ for $0 < p <1$.  A more detailed description of the difficulties and intuition can be found in Section 5.  There, additional Euclidean estimates, von Neumann projection lemmas, and their relation to matricial calculus are presented.  Section 6 puts this all together to arrive at the entropy estimate.

There are three other sections aside from these core sections (3-6) and the introduction.  Section 2 will set up notation, review background material, and prove some technical lemmas, some of which are interesting in their own right (e.g., Lemma 2.7 which proves continuity of the covering number as a function of $\epsilon$).  Section 7 will present the group von Neumann algebra corollaries.  The appendix will review St. Raymond's volume estimates \cite{sr} and establish the necessary volume/metric entropy inequalities for the quasinorm balls in the $k\times k$ complex matrices.

\section{Preliminaries}

\subsection{Notation} A tracial von Neumann algebra consists of a tuple $(M, \varphi)$ where $M$ is a von Neumann algebra and $\varphi$ is a tracial state on $M$.  $\varphi$ will always be normal and faithful.  For any $0 <p < \infty$ $\|\cdot\|_p$ denotes the $p-$ (quasi)norm on $M$ given by $\|x\|_p = \varphi(|x|^p)^{1/p}$, $x \in M$.  When $p =\infty$ I will sometimes write $\|x\|_{\infty}$ for the operator norm of $x$.  Unless otherwise stated, $(M,\varphi)$ will denote a tracial von Neumann algebra.  $L^2(M)$ is the completion of $M$ under the $\|\cdot\|_2$-norm.  

Suppose $x$ is a normal element in $(M, \varphi)$. $\varphi$ induces a Borel measure $\mu$ on the spectrum of $x$, $sp(x)$, obtained by restricting $\varphi$ to $C(sp(x))$ via the spectral theorem and using the Riesz representation theorem to extend $\varphi$'s restriction to a Borel measure supported on $sp(x)$.  This measure will be called the spectral distribution of $x$.   If $f$ is a bounded, complex-valued Borel function on $sp(x)$, then $f(x)$ denotes the unique element obtained by extending the continuous functional calculus to the bounded, Borel functional calculus.  For a subset $S \subset \Omega$ I will write $1_S$ for the indicator function on $S$, i.e., $1_S(\omega) = 1$ if $\omega \in S$ and $1_S(\omega)=0$ if $\omega \in \Omega - S$.  With this notation, if $E \subset sp(x)$ is a Borel set, then $1_E(x)$ is the spectral projection associated to $E$ and $x$.  

Notice that these notions make sense for a symmetric, real linear operator $T$ from on a finite dimensional vector space (since it's diagonalizable).  The spectral calculus notation in the complex case will also be used in the real case.  For example, $|T|$ denotes $(T^*T)^{1/2}$ where $T^*$ is the transpose of $T$.  As in the complex case, in this situation one has a polar decomposition $T = U|T|$ where $U$ is an orthogonal matrix and $|T|$ is a symmetric, positive semidefinite operator.  These also make sense (as in the complex case) when the domain and range of $T$ are distinct real vector spaces (i.e., of different dimension).

For any $k \in \mathbb N$, $M_k(\mathbb C)$ denotes the space of $k \times k$ complex matrices and $tr_k$ denotes the unique tracial state on $M_k(\mathbb C)$.  $M^{sa}_k(\mathbb C)$ denotes the space of $k \times k$ self-adjoint complex matrices.   If $R>0$ then $(M_k(\mathbb C))_R$ and $(M^{sa}_k(\mathbb C))_R$ denote the sets of elements $x$ in $M_k(\mathbb C)$ or $M^{sa}(\mathbb C)$, respectively, whose operator norm is no greater than $R$.  $(M_k(\mathbb C))^n$ will denote the space of $n$-tuples of elements in $M_k(\mathbb C)$ and the $\|\cdot\|_2$-norm on this space is defined by $\|(\xi_1,\ldots, \xi_n)\|_2 = (\sum_{i=1}^n tr_k(\xi_i^*\xi_i))^{1/2}$.  This notation agrees with the 2-norm for a von Neumann algebra when $n=1$.  The $\|\cdot\|_{\infty}$-norm (sometimes written simply as $\|\cdot\|$) is defined by $\|(\xi_1,\ldots, \xi_n)\|_{\infty} = \max_{1\leq i \leq n} \|\xi\|_{\infty}$.  $(M^{sa}_k(\mathbb C))^n$, $(M_k(\mathbb C))_R)^n$, and $(M^{sa}_k(\mathbb C)_R)^n$ have analogous meanings - they are direct sums of $n$ copies of the inner terms.  If $\xi, \eta \in (M_k(\mathbb C))^n$, then $\xi \cdot \eta, \xi+\eta \in (M_k(\mathbb C))^n$ have the obvious coordinatewise meaning.  Given $\xi \in (M_k(\mathbb C))^n$ and $\epsilon, p \in (0,\infty)$, $B_p(\xi, \epsilon) \subset (M_k(\mathbb C))^n$ denotes the ball of $\|\cdot\|_p$-radius $\epsilon$ with center $\xi$.

Given a discrete group $\Gamma$, $L(\Gamma) \subset B(\ell^2(\Gamma))$ denotes the group von Neumann algebra generated by the left regular representation of $\Gamma$.  $e_{\Gamma}$ denotes the identity element of the group.  $\Gamma^{op}$ denotes the opposite group of $\Gamma$. 

$L(\Gamma)$ has a canonical, faithful, tracial state given by $\varphi(x) = \langle x 1_{e_{\Gamma}}, 1_{e_\Gamma} \rangle$.  Unless otherwise stated, $L(\Gamma)$ will be regarded as a tracial von Neumann algebra w.r.t. this canonical trace.  Recall that there exists a unique tracial state on $L(\Gamma) \iff L(\Gamma) \text{ is a factor} \iff \Gamma \text{ is i.c.c.}$.

For $n \in \mathbb N$, $\mathbb F_n$ denotes the free group on $n$ elements.  Suppose $a_1,\ldots, a_n$ are the canonical generators for $\mathbb F_n$.  Recall that there is a bijective correspondence between $\mathbb F_n$ and the reduced words $w$ in $a_1,\dots, a_n$.  If $g_1,\ldots, g_n$ are elements in a group $\Gamma$, then $w(g_1,\ldots,g_n)$ denotes the obvious element obtained by associating to $w$ the corresponding unique element $g \in \mathbb F_n$ and then setting $w(g_1,\ldots, g_n) = \pi(g)$ where $\pi:\mathbb F_n \rightarrow \Gamma$ is the unique group homomorphism such that $\pi(a_i) = g_i$, $1\leq i\leq n$.  I will blur the distinction between reduced words $w$ in $a_1,\ldots, a_n$, and the unique group element they represent in $\mathbb F_n$. 

Suppose $X=\{x_1,\ldots, x_n\}$ is a finite $n$-tuple of elements in a complex $*$-algebra and $F$ is a $p$-tuple of noncommutative $*$-polynomials in $n$ indeterminates.  Write $F = \{f_1,\ldots,f_p\}$.  $F(X)$ denotes the $p$-tuple of elements $\{f_1(X),\ldots, f_p(X)\}$.  Given another finite  $Y=\{y_1,\ldots, y_p\}$ in the same complex $*$-algebra, $X \cup Y$ denotes the concatenated finite tuple $\{x_1,\ldots, x_n, y_1,\ldots, y_p\}$.   When the elements of a tuple $X$ are identically $0$ or the identity this will written as $X=0$ or $X=I$, respectively.

Given a complex algebra $A$, $A^{op}$ denotes the opposite algebra of $A$.  Recall that $A^{op}=A$ as a vector space is $A$.  Given an element $a \in A$, I will also write $a \in A^{op}$.  Sometimes to reinforce that this element is regarded in $A^{op}$ I will write $a^{op} \in A^{op}$, however at times (only when multiplication is not an issue and in order to reduce notation) I'll simply use the shorthand $a \in A^{op}$.   $I$ will denote the identity of a given complex algebra $A$.

\subsection{Metric Notation, Packing/Covering, Sumsets}

Suppose $(M,d)$ is a metric space.  Given $\epsilon >0$ and $x \in M$, $B(x,\epsilon)$ denotes the open ball with radius $\epsilon$ and center $x$.  Given $X \subset M$, an $\epsilon$-cover for $X$ is a subset $\Lambda \subset M$ such that $\cup_{i \in \Lambda} B(x_i, \epsilon) \supset X$.   For any $\epsilon >0$ $K_{\epsilon}(X)$ denotes the minimum number of elements in an $\epsilon$-cover for $X$.  A subset $F$ of $M$ is $\epsilon$-separated if for any distinct elements $x, y \in F$, $d(x,y) \geq \epsilon$.  $S_{\epsilon}(X)$ denotes the maximum numbers of elements in an $\epsilon$-separated subset $F \subset X$.  Denote by $\mathcal N_{\epsilon}(X)$ the $\epsilon$-neighborhood of $X$ in $M$.  These $\epsilon$-parametrized metric concepts are closely related to one another.  Below are a number of their simple but very useful properties which are stated for convenience and without proof (they are all easy to verify):

\begin{proposition} Suppose $X$ is a subset of the metric space $(M,d)$.  The following hold:
\begin{enumerate} [(i)]
\item If $Y \subset X$, then for any $\epsilon >0$, $S_{\epsilon}(Y) \leq S_{\epsilon}(X)$, $K_{\epsilon}(Y) \leq K_{\epsilon}(X)$, and $\mathcal N_{\epsilon}(Y) \subset \mathcal N_{\epsilon}(X)$.
\item $S_{\epsilon}(X \cup Y) \leq S_{\epsilon}(X) + S_{\epsilon}(Y)$.
\item If $r > \epsilon >0$, then  $S_r(X) \leq S_{\epsilon}(X)$, $K_r(X) \leq K_{\epsilon}(X)$, and $\mathcal N_{\epsilon}(X) \subset \mathcal N_r(X)$.
\item If $\epsilon > 2\delta >0$, then $S_{\epsilon}(\mathcal N_{\delta}(X)) \leq S_{\epsilon - 2\delta}(X)$ and if $\epsilon > \delta >0$, then $K_{\epsilon}(\mathcal N_{\delta}(X)) \leq K_{\epsilon -\delta}(X)$.
\item $K_{\epsilon}(X) \leq S_{\epsilon}(X) \leq K_{\frac{\epsilon}{2}}(X)$.
\item If $(M, d)$ is Euclidean space of dimension $k$, $\mu$ is Lebesgue measure, and $c_k$ is the volume of the unit ball, then
\begin{eqnarray*}
\frac{\mu(\mathcal N_{\epsilon}(X))}{c_k \epsilon^k} & \leq & K_{\epsilon}(X) \\
                                                                                   & \leq & S_{\epsilon}(X) \\
                                                                                   & \leq & \frac{\mu(\mathcal N_{\epsilon/2}(X))}{c_k (\epsilon/2)^k}.
\end{eqnarray*}
\end{enumerate}
\end{proposition}

The (upper) covering dimension (also known as the box-counting, entropy, or Minkowski dimension) of $X$ is defined by 
\begin{eqnarray*}
\dim(X) = \limsup_{\epsilon \rightarrow 0} \frac{K_{\epsilon}(X)}{|\log \epsilon|}.
\end{eqnarray*}  
It is straightforward to check from the properties listed above that $\dim(X) = \limsup_{\epsilon \rightarrow 0} \frac{S_{\epsilon}(X)}{|\log \epsilon|}$ and when $X$ is a subset of Euclidean space of dimension $k$ and $\mu$ is Lebesgue measure, then
\begin{eqnarray*} 
\dim(X) = k + \limsup_{\epsilon \rightarrow 0} \frac{\log(\mu(\mathcal N_{\epsilon}(X)))}{|\log \epsilon|}.
\end{eqnarray*}

Suppose now that $E$ is a normed linear space and $A$ and $B$ are subsets of $E$.  The sumset of $A$ and $B$ is the subset of $E$ whose elements are of the form $a+b$ where $a \in A$ and $b \in B$.  This construction is often denoted by $A+B$, but because of the prolific use of the addition sign as an algebraic operation throughout this paper, I will write the sumset of $A$ and $B$ as $A \boxplus B = \{a+b: a \in A, b \in B\}$ to emphasize its set theoretic meaning.  If $A_1, \ldots, A_n$ is a sequence of subsets then their sumset is $\boxplus_{i=1}^n A_i = \{a_1 + \cdots + a_n : a_i \in A_i, 1\leq i \leq n\}$.  The sumset operations will be used primarily in the context of metric properties and covering estimates.  Note that the usage of $\boxplus$ conflicts with its standard meaning in free probability for the free additive convolution.  Since I will not be using the free additive convolution in this paper there should be no confusion.

By the triangle inequality if $\Lambda_i$ are $\epsilon$-covers for subsets $A_i$, then $\boxplus_{i=1}^n \Lambda_i$ is an $n\epsilon$-cover for $\boxplus_{i=1}^n A_i$.  Thus, there is the following simple and very coarse estimate:

\begin{lemma} If $E$ is a Banach space and $\langle A_i \rangle_{i=1}^n$ is a sequence of subsets of $E$, then for any $\epsilon >0$, 
\begin{eqnarray*} S_{2n\epsilon}(\boxplus_{i=1}^n A_i) & \leq & K_{n\epsilon}(\boxplus_{i=1}^n A_i) \\
                                                                                          & \leq & \Pi_{i=1}^n K_{\epsilon}(A_i) \\
                                                                                           & \leq & \Pi_{i=1}^n S_{\epsilon}(A_i).\\
\end{eqnarray*}
\end{lemma}

Finally, when $B = B(x,r) \subset E$, then for $s>0$, $sB = B(x, rs)$, the balls obtained from dilating $B$ by $s$ w.r.t. the center $x$.

\subsection{Microstates} 
Suppose $(M,\varphi)$ is tracial von Neumann algebra and $X =\{x_1,\ldots, x_n\}$ is an $n$-tuple of elements in $M$.  Given $m,k \in \mathbb C$ and $\gamma >0$ the $(m,k,\gamma)$ $*$-microstates $\Gamma(X;m,k,\gamma)$ consists of all elements $\xi = (\xi_1,\ldots, \xi_n) \in (M_k(\mathbb C))^n$ such that for any $1 \leq p \leq m$, $1 \leq i_1,\ldots, i_p \leq n$, and $j_1,\ldots, j_p \in \{1, *\}$,
\begin{eqnarray*}
|\varphi(x_{i_1}^{j_1} \cdots x_{i_p}^{j_p}) - tr_k(\xi_{i_1}^{j_1} \cdots \xi_{i_p}^{j_p})| < \gamma.
\end{eqnarray*}
There are several useful variants of the above.  For $R>0$, $\Gamma_R(X;m,k,\gamma)$ is the set consisting of all elements $\xi \in \Gamma(X;m,k,\gamma)$ such that $\| \xi \|_{\infty} \leq R$.  When $X$ consists of self-adjoint or unitary elements, one can impose the condition that each microstate's entries are self-adjoint or unitary complex matrices.  These sets will be denoted by $\Gamma^{sa}_{\cdot}(\cdot)$ and $\Gamma^u_{\cdot}(\cdot)$, respectively.  The original definition of matricial microstates was made in \cite{v1}.  While originally defined for finite tuples of self-adjoint elements, the unitary or general definition above is straightforward.

One can also consider the microstates of $X$ in the presence of $Y$ (introduced in \cite{v2}).  For this suppose $Y = \{y_1,\ldots, y_d\} \subset M$.  $\Gamma(X:Y;m,k,\gamma)$ consists all $\xi \in \Gamma(X;m,k,\gamma)$ for which there exists an $\eta \in (M_k(\mathbb C))^d$ such that $(\xi, \eta) \in \Gamma(X \cup Y:m,k,\gamma)$.  Another way of putting it is that $\Gamma(X:Y;m,k,\gamma)$ is the projection of $\Gamma(X \cup Y;m,k,\gamma)$ onto the first $n$ coordinates.

Denote by $\text{vol}$ Lebesgue measure w.r.t. the real inner product metric which $\|\cdot\|_2$ induces on $(M_k(\mathbb C))^n$ ($\|\xi \|_2^2 = \sum_{i=1}^n tr_k(\xi_i^*\xi)^{1/2}$).  Observe that under this identification, $(M_k(\mathbb C))^n$ is isomorphic as a real vector space to $\mathbb R^{2nk^2}$.  Define successively
\[ \chi(X;m,\gamma) = \limsup_{k \rightarrow \infty} k^{-2} \cdot \log(\text{vol}(\Gamma(X;m,k,\gamma))) + 2n \cdot \log k,
\]
\[ \chi(X) = \inf \{\chi(X;m,\gamma): m \in \mathbb N, \gamma >0\}.
\]
Note the scaling factor $2n \cdot \log k$ is different from that in \cite{v1} due to both the nonself-adjoint context as well as the normalization of the traces used here.  $\chi(X)$ is called the free entropy of $X$.  Replacing the microstates of $X$ with the microstates in the presence of $Y$ in the above yields a quantity $\chi(X:Y)$ called the free entropy of $X$ in the presence of $Y$.  One can impose operator norm cutoff conditions and obtain free entropy quantities $\chi_R(X)$ as well as $\chi_R(X:Y)$.   When $X$ consists of self-adjoint quantities one can consider the self-adjoint microstates and using Lebesgue measure on $(M^{sa}_k(\mathbb C))^n$ inherited from $\|\cdot\|_2$.   Replacing the scaling factor $2n \cdot \log k$ in the definition above with $n \log k$, one arrives at the free entropy of the self-adjoint tuple, denoted by $\chi^{sa}(X)$.

Other functions can be applied to the microstate spaces with a more geometric measure theoretic bent.  Using covering numbers w.r.t. the metric $\|\cdot\|_2$, define successively for any $\epsilon>0$
\[ \mathbb K_{\epsilon}(X;m,\gamma) = \limsup_{k \rightarrow \infty} k^{-2} \cdot \log(K_{\epsilon}(\Gamma(X;m,k,\gamma))),
\]
\[ \mathbb K_{\epsilon}(X) = \inf \{\mathbb K_{\epsilon}(X;m,\gamma): m \in \mathbb N, \gamma >0\}.
\]
One similarly defines $\mathbb S_{\epsilon}$ by replacing $K_{\epsilon}$ above with $S_{\epsilon}$.  Again, an operator norm cutoff can be introduced giving rise to notation such as $\mathbb K_{\epsilon, R}(\cdot)$ and $\mathbb S_{\epsilon, R}(\cdot)$.   One can also define corresponding covering and separated $\epsilon$-quantities for self-adjoint and unitary microstates spaces when $X$ consists of self-adjoint or unitary elements.  They will be notated by $\mathbb K_{\epsilon}^{sa}(X)$, $\mathbb S_{\epsilon}^{sa}(X)$, $\mathbb K_{\epsilon}^u(X)$, $\mathbb S_{\epsilon}^u(X)$.

The (modified) free entropy dimension of $X$, $\delta_0(X)$ is the common quantity
\begin{eqnarray*}
\delta_0(X) & = & \limsup_{\epsilon \rightarrow 0} \frac{\mathbb K_{\epsilon}(X)}{|\log \epsilon|} \\
                  & = & \limsup_{\epsilon \rightarrow 0} \frac{\mathbb S_{\epsilon}(X)}{|\log \epsilon|}. \\
\end{eqnarray*}
The equation above was not the original formulation of $\delta_0$ introduced in \cite{v2} involving semicircular (or what would in this context be circular) free perturbations, but it was show in \cite{j1} to be equivalent.  Again one can consider self-adjoint or unitary quantities, when $X$ consists of self-adjoint or unitary elements and they will be denoted by $\delta_0^{sa}$ and $\delta_0^u$, respectively.

The following says that the use of operator norm cutoffs or self-adjoint/unitary restrictions have no effect on the entropy quantities.  The free entropy claim of (i) is essentially contained in \cite{bb}.  The covering equalities (ii)-(iv) follow from \cite{j4}, Rogers's asymptotic estimates \cite{rogers}, and Remark 2.6.

\begin{proposition} If $R > \max_{x \in X} \|x\|$, and $\epsilon >0$ then the following are true:
\begin{enumerate}[(i)]
\item $\chi_R(X) = \chi(X)$.  If $X$ consists of self-adjoint elements, then $\chi_R^{sa}(X) = \chi^{sa}(X)$.
\item  $\mathbb K_{\epsilon}(X) = \mathbb K_{\epsilon, R}(X)$.   
\item If $X$ consists of self-adjoint elements, then $\mathbb K_{\epsilon}(X) = \mathbb K^{sa}_{\epsilon}(X) = \mathbb K^{sa}_{\epsilon, R}(X)$.
\item If $X$ consists of unitary elements, then $\mathbb K_{\epsilon}(X) = \mathbb K^u_{\epsilon}(X)$.
\end{enumerate}
\end{proposition} 

It was asked in \cite{v1} whether $\delta_0$ or some variant of it is a von Neumann algebra invariant.  More precisely, the question is this: if $X$ and $Y$ are finite tuples in a tracial von Neumann algebra which generate the same von Neumann algebra, then is $\delta_0(X)=\delta(Y)$?  An affirmative answer to this question would show (again by \cite{v1}) the nonisomorphism of the free group factors and settle a longstanding problem in operator algebras.  

It is known from \cite{v3} that if $X$ and $Y$ generate the same $*$-algebra, then $\delta_0(X)=\delta_0(Y)$ (Proposition 2.5 below will basically demonstrate this as well).  Thus, one can make sense of the free entropy dimension of a finitely generated complex $*$-algebra as the free entropy dimension of any of its finite generating sets.  In particular, for a finitely generated discrete group $G$ one can rigorously define $\delta_0(G)$ to be the free entropy dimension of the tuple consisting of the unitaries associated to any finite tuple of group generators.  

\cite{v2} applied the microstates theory to show that the free group factors have no Cartan subalgebras, and \cite{g} used it to show that they were prime, answering several old operator algebra questions.   These results were subsequently strengthened and generalized using different methods which were discussed in the introduction.

Motivated by geometric measure theoretic considerations, \cite{j3} effectively showed that $\delta_0$ is an invariant when a tuple has finite covering $1$-entropy.  Recall there that for $\alpha >0$, a finite tuple of self-adjoints $X$ in a tracial von Neumann algebra is said to be $\alpha$-bounded if there exist $K, \epsilon_0 >0$ such that for all $0 < \epsilon < \epsilon_0$,
\begin{eqnarray*}
\mathbb K^{sa}_{\epsilon}(X) \leq \alpha \cdot |\log \epsilon| + K.
\end{eqnarray*} 
If $X$ is $1$-bounded and contains an entry $x$ such that $\chi^{sa}(x) > -\infty$, then $X$ is said to be \textbf{strongly 1-bounded}.  It turns out that if $X$ is strongly $1$-bounded, then any other finite generating tuple $Y$ for the von Neumann algebra generated by $X$ satisfies the same inequality above, possibly with a different $\epsilon_0$ and $K$.  In particular, $\delta_0(Y) \leq 1$.   This allows one to distinguish strongly $1$-bounded von Neumann algebras (which are closed under diffuse intersections, normalizers, and pairwise commutation relations) from von Neumann algebras with (microstates) free entropy dimension strictly greater than $1$ (e.g. the free group factors as demonstrated in \cite{v1}).  

The key feature in the definition of strongly $1$-bounded is that the defining inequality propagates to any representative of the associated von Neumann algebra.  It reduces a von Neumann algebra nonexistence result to computing a numeric of one propitious generating set.  As a result of this, a tracial von Neumann algebra is said to be strongly $1$-bounded if it has a finite set of self-adjoint elements which is strongly $1$-bounded.

While $\alpha$-boundedness was phrased for finite tuples of self-adjoints, it makes sense for a finite tuple of general (possibly non-self-adjoint) elements.   Formally,

\begin{definition} A finite tuple $X$ in $(M,\varphi)$ is $\alpha$-bounded if there exist $K, \epsilon_0 >0$ such that for all $0 < \epsilon < \epsilon_0$,
\begin{eqnarray*}
\mathbb K_{\epsilon}(X) & \leq & \alpha \cdot |\log \epsilon| +K.
\end{eqnarray*}
\end{definition}

This definition coincides with the original one made when the finite tuple consists of self-adjoint elements (by Proposition 2.3).  Notice that if $X$ is $\alpha$-bounded, then $\delta_0(X) \leq \alpha$.  

Whether one works with self-adjoint, general, or unitary tuples is immaterial.  Indeed, one can move from one type of generating set to another with $*$-algebraic operations and invoke the following:

\begin{proposition} If $X$ and $Y$ are finite tuples of elements in $(M,\varphi)$ which generate the same complex $*$-algebra, then for any $R>0$ there exist $L, R_1 >0$ such that for any $\epsilon >0$,
\begin{eqnarray*}
\mathbb K_{\epsilon, R}(X) \leq \mathbb K_{\frac{\epsilon}{L}, R_1}(Y).
\end{eqnarray*}
It follows that there exists a $C >0$ such that for any $\epsilon >0$, $\mathbb K_{\epsilon}(X) \leq \mathbb K_{\epsilon/C}(Y)$.  
In particular, if $X$ is a general finite tuple and $X^{sa}$ is the tuple obtained by taking the real and imaginary portions of $X$, then $X$ is $\alpha$-bounded (as a finite tuple of general elements) iff $X^{sa}$ is $\alpha$-bounded (as a finite tuple of self-adjoint elements).  Also, $\delta_0(X) = \delta_0(X^{sa}) = \delta_0^{sa}(X^{sa})$.
\end{proposition}

\begin{proof}  By hypothesis there exist finite tuples of $*$-polynomials, $F$ and $G$, such that $F(X)=Y$ and $G(F(X))=Y$.  Given $R>0$, there exists and $L_1, R_1$ dependent on $R$ such that for any $k \in \mathbb N$ and $\xi, \eta \in ((M_k(\mathbb C))_R)^n$
\begin{enumerate}[(i)]
\item $L_1 \cdot \|F(\xi)-F(\eta)\|_2  \geq \|G(F(\xi)) - G(F(\eta))\|_2$.  
\item $\|F(\xi)\|_{\infty} < R_1$.
\end{enumerate}

Because $G(F(X))=X$ there exist $m_0 \in \mathbb N$, $\gamma_0>0$ such that for any $\xi \in \Gamma_R(X,m_0,k,\gamma_0)$, $\|G(F(\xi)) - \xi\|_2 < \epsilon/4$.  Using condition (i) above, for any  $\xi, \eta \in \Gamma_R(X,m_0,k,\gamma_0)$,
\begin{eqnarray*}
L_1 \cdot\|F(\xi) - F(\eta)\|_2 & \geq & \|G(F(\xi) - G(F(\eta))\|_2 \\
                                               & \geq & \|\xi - \eta\|_2 - \epsilon/2.\\
\end{eqnarray*}
It follows that for any $m_1 > m_0$ and $0 < \gamma_1 < \gamma_0$
\begin{eqnarray*}
S_{\epsilon}(\Gamma_R(X;m_1,k,\gamma_1)) & \leq & S_{\frac{\epsilon}{2L_1}}(F(\Gamma_R(X;m_1,k,\gamma_1))).\\
\end{eqnarray*}  

Given $m \in \mathbb N$ and $\gamma >0$, then there exist by condition (ii) corresponding $m_1 > m_0$ and $0 < \gamma_1 < \gamma_0$ such that $F(\Gamma_R(X;m_1,k,\gamma_1)) \subset \Gamma_{R_1}(Y;m,k,\gamma)$.  Combining this with the above and Proposition 2.1,
\begin{eqnarray*}
K_{\epsilon}(\Gamma_R(X;m_1,k,\gamma_1)) & \leq & S_{\epsilon}(\Gamma_R(X;m_1,k,\gamma_1)) \\ & \leq & S_{\frac{\epsilon}{2L_1}}(F(\Gamma_R(X;m_1,k,\gamma_1))) \\
                                                                   & \leq & S_{\frac{\epsilon}{2L_1}}(\Gamma_{R_1}(Y;m,k,\gamma)) \\
                                                                   & \leq & K_{\frac{\epsilon}{4L_1}}(\Gamma_{R_1}(Y;m,k,\gamma)).\\
\end{eqnarray*}
This being true for any $m \in \mathbb N$, $\gamma >0$, it follows that with $L = (4L_1)^{-1}$, 
\begin{eqnarray*}
\mathbb K_{\epsilon,R}(X) \leq \mathbb K_{\frac{\epsilon}{L}, R_1}(Y).  
\end{eqnarray*}
The first claim is established.  

The second claim follows from the first claim and Proposition 2.3.

For the third part, notice that the second claim implies that if $X$ and $Y$ generate the same $*$-algebra, then $X$ is $\alpha$-bounded iff $Y$ is $\alpha$-bounded.  Thus, if $X$ is a general finite tuple and $X^{sa}$ denotes the tuple consisting of the real and imaginary parts of $X$, then $X$ is $\alpha$-bounded iff $X^{sa}$ is $\alpha$-bounded (as a general tuple).  Proposition 2.3 shows that the $\epsilon$-covering numbers of a tuple of self-adjoints computed w.r.t. self-adjoint or general microstates coincide.  So $X^{sa}$ is $\alpha$-bounded as a general tuple iff $X^{sa}$ is $\alpha$-bounded as a tuple of selfadjoint elements.  This completes the third claim.

The fourth and final claim concerning $\delta_0$ is trivial from the first claim and Proposition 2.3.
\end{proof}

\subsection{Rogers's asymptotic bound}
 \cite{rogers} investigated $\epsilon$-covering estimates for the unit ball $B_d$ in $\mathbb R^d$ for large $d$.  One might expect that $K_{\epsilon}(B_d)$ should be the ratio of $1$ over $\epsilon$ raised to the ambient dimension $d$, i.e., $K_{\epsilon}(B_d) \sim (\frac{1}{\epsilon})^d$.  From simple volume comparison arguments one has a coarser estimate involving an additional exponential constant of $2$:

\begin{eqnarray*} K_{\epsilon}(B_d) \leq \left(\frac{2}{\epsilon}\right)^d = \frac{2^d}{\epsilon^d}.
\end{eqnarray*}

\noindent While this estimate (or a better one with a numerator of $1+\epsilon$) often suffices to get appropriate dimension bounds in the microstate setting, I'll need a sharper estimate where the numerator $2^d$ is replaced with a term with polynomial growth. Rogers proved in \cite{rogers} that there exists a universal constant $C_r$ such that for $d \geq \max\{1/\epsilon, 9\}$,

\begin{eqnarray*}  K_{\epsilon}(B_d) \leq \frac{C_r \cdot d^{5/2}}{\epsilon^d}.
\end{eqnarray*}

\noindent By dilating, it follows that if $B_d(\alpha)$ denotes the ball of radius $\alpha$ in $\mathbb R^n$, then for $d \geq \max\{\frac{\alpha}{\epsilon}, 9\}$,

\begin{eqnarray*}  K_{\epsilon}(B_d(\alpha)) \leq C_r \cdot d^{5/2} \cdot \left (\frac{\alpha}{\epsilon}\right)^d.
\end{eqnarray*}

\begin{remark} Suppose $\Omega \subset \mathbb R^d$ and $s, t >0$.  Observe that for $d > \max\{\frac{s+t}{s}, 9\}$ the result above implies 
\begin{eqnarray*}
K_{s}(\Omega) & \leq & C_r \cdot d^{5/2} \cdot \left (\frac{s+t}{s}\right)^d \cdot K_{s+t}(\Omega).
\end{eqnarray*}
To see this pick an $(s+t)$-cover $\langle x_i \rangle_{i \in I}$ for $\Omega$ such that $\#I = K_{s+t}(\Omega)$.  From the discussion above, for each $i$ the ball $B(x_i, s+t)$ has an $s$-cover $\langle y_{(i,j)} \rangle_{j \in J}$ where $J$ is an indexing set such that 
\begin{eqnarray*}
\#J & \leq & C_r \cdot d^{5/2} \cdot \left (\frac{s+t}{s}\right)^d.\\
\end{eqnarray*}
Clearly $\langle y_{(i,j)} \rangle_{(i,j) \in I \times J}$ is an $s$-cover for $\Omega$ and it has cardinality no greater than
\begin{eqnarray*}
 \#I \cdot \#J & \leq & C_r \cdot d^{5/2} \cdot \left (\frac{s+t}{s}\right)^d \cdot K_{s+t}(\Omega). \\
\end{eqnarray*}
\end{remark}

When Remark 2.6 is used for the $\epsilon$-coverings in the microstate setting, the polynomial term vanishes under the asymptotic logarithmic process and one recovers the kind of nested scaling property enjoyed by dyadic cubes.  More specifically one has the following which are interesting to compare with the corresponding properties proved in \cite{v4} for the free Fisher information of a semicircular perturbation:

\begin{lemma} Suppose $X$ is an $n$-tuple of operators in a tracial von Neumann algebra $M$.  Define $f:(0,\infty) \rightarrow \mathbb [0,\infty)$ by $f(t) = \mathbb K_t(X)$.  The following hold for $f$:
\begin{itemize}
\item $f$ is monotonically decreasing.
\item For any $s,t \in (0,\infty)$, $f(s) \leq f(s+t) + 2n \log \left(\frac{s+t}{s} \right)$.  
\item $f$ is continuous.
\end{itemize}
The same results hold if $X$ consists of self-adjoint elements and $f(t) = \mathbb K^{sa}_t(X)$.
\end{lemma}

\begin{proof} For the first property for $s, t \in (0,\infty)$ such that $s<t$, and for any $m,k \in \mathbb N$ and $\gamma >0$ 
\begin{eqnarray*}
K_s(\Gamma(X;m,k,\gamma)) \geq K_t(\Gamma(X;m,k,\gamma)).
\end{eqnarray*}  
Passing this through the limiting process, one has $f(s) = \mathbb K_s(X) \geq \mathbb K_t(X) = f(t)$.

For the second property, suppose $s,t \in (0,\infty)$ and again, $m,k \in \mathbb N$ and $\gamma >0$.  By Remark 2.6,

\begin{eqnarray*} \mathbb K_s(X;m,\gamma) & = & \limsup_{k \rightarrow \infty} k^{-2} \cdot \log \left(K_s(\Gamma(X;m,k,\gamma)) \right) \\& \leq & \limsup_{k \rightarrow \infty} k^{-2} \cdot \log \left(\frac{C_r (2nk^2)^{5/2}(s+t)^{2nk^2}}{s^{2nk^2}} \cdot K_{s+t}(\Gamma(X;m,k,\gamma))\right) \\ & \leq & \mathbb K_{s+t}(X;m,\gamma) + 2n \log \left(\frac{s+t}{s} \right).\\
\end{eqnarray*}

\noindent As this holds for any $m, \gamma$, $f(s) = \mathbb K_s(X) \leq \mathbb K_{s+t}(X) + 2n \log \left ( \frac{s+t}{s} \right) = f(s+t) + 2n \log \left(\frac{s+t}{s} \right)$.

The third property follows from the first and second ones.

The self-adjoint situation is completely analogous to the nonself-adjoint arguments presented above.
\end{proof}

\begin{remark}  Rogers's result can be used to prove statements (ii)-(iv) of Proposition 2.3.  I'll show this for the Proposition 2.3(iii) and leave the other two to the reader, as they are completely analogous.  The third statement says that for $R > \max_{x \in X} \|x\|$ with $X$ a finite tuple consisting of self-adjoint elements, that for any $\epsilon >0$, $\mathbb K_{\epsilon}(X) = \mathbb K^{sa}_{\epsilon}(X) = \mathbb K^{sa}_{\epsilon, R}(X)$.   From Proposition 2.1(i) and the inclusions, $\Gamma(X;m,k,\gamma) \supset \Gamma^{sa}(X;m,k,\gamma) \supset \Gamma_R^{sa}(X;m,k,\gamma)$ it follows that $\mathbb K_{\epsilon}(X) \geq \mathbb K^{sa}_{\epsilon}(X) \geq \mathbb K^{sa}_{\epsilon, R}(X)$.  It remains then to prove the reverse inequalities of this chain.

Fix $t >0$.  For sufficiently large $m \in \mathbb N$ and small $\gamma >0$, if $\xi \in \Gamma(X:m,k,\gamma)$, then $\|\xi - (\xi+\xi^*)/2\|_2 < t$.  Thus, $\Gamma(X;m,k,\gamma) \subset \mathcal N_t(\Gamma^{sa}(X;m,k,\gamma))$.  By Proposition 2.1(iv), 
\begin{eqnarray*}
K_{\epsilon}(\Gamma(X;m,k,\gamma)) & \leq & K_{\epsilon}(\mathcal N_t(\Gamma^{sa}(X;m,k,\gamma))) \\
                                                              & \leq & K_{\epsilon - t}(\Gamma^{sa}(X;m,k,\gamma)).\\
\end{eqnarray*}
Passing this through the limiting process, $\mathbb K_{\epsilon}(X) \leq \mathbb K_{\epsilon-t}^{sa}(X)$.   The continuity property for the self-adjoint case in Lemma 2.6 shows that $\mathbb K_{\epsilon}(X) \leq \mathbb K_{\epsilon}^{sa}(X)$ which is the first of the two remaining inequalities.

Turning to the inequality $\mathbb K^{sa}_{\epsilon}(X) \leq \mathbb K^{sa}_{\epsilon, R}(X)$, when $R > \max_{x \in X} \|x\|$, Lemma 2.1 of \cite{j3} shows that for a given $t >0$ and any $m_0 \in \mathbb N$ and $\gamma_0>0$, there exist an $m \in \mathbb N$ and $\gamma >0$ such that 
\begin{eqnarray*}
\Gamma^{sa}(X;m,k,\gamma) & \subset & \mathcal N_t(\Gamma^{sa}_R(X;m_0, k, \gamma_0)).\\
\end{eqnarray*}
By Proposition 2.1, (i) and (iv), and Remark 2.6
\begin{eqnarray*}
K_{\epsilon}(\Gamma^{sa}(X;m,k,\gamma)) & \leq & K_{\epsilon}(\mathcal N_t(\Gamma^{sa}_R(X;m_0,k,\gamma_0))) \\
                                                                      & \leq & K_{\epsilon-t}(\Gamma^{sa}_R(X;m_0,k,\gamma_0))\\
                                                                      & \leq & C_r \cdot (2nk^2)^{5/2} \cdot \left (\frac{\epsilon}{\epsilon-t}\right)^{2nk^2} \cdot K_{\epsilon}(\Gamma^{sa}_R(X;m_0,k,\gamma_0)) \\
\end{eqnarray*}
Applying $\limsup_{k\rightarrow \infty} k^{-2} \log$ on both sides, 
\begin{eqnarray*} 
\mathbb K^{sa}_{\epsilon}(X) & \leq & \mathbb K^{sa}_{\epsilon}(X;m,\gamma) \\
                                               & \leq & \mathbb K^{sa}_{\epsilon, R}(X;m_0,\gamma_0) + 2n \cdot [\log(\epsilon) - \log(\epsilon-t)].
\end{eqnarray*}  
This is true for any $m_0 \in \mathbb N$, $\gamma_0 >0$.  Thus, $\mathbb K^{sa}_{\epsilon}(X) \leq \mathbb K^{sa}_{\epsilon, R}(X)+  2n \cdot [\log(\epsilon) - \log(\epsilon-t)]$.  $t >0$ was arbitrary, so it follows that $\mathbb K^{sa}_{\epsilon}(X) \leq \mathbb K^{sa}_{\epsilon, R}(X)$ as promised.
\end{remark}

\subsection{Derivatives} 

I'll set forth here some basic notation for derivatives and recall a few fundamental facts.  Details can be found in many places, e.g., \cite{lang}.  It will be convenient to speak about derivatives in the general context of real Banach spaces.   Suppose $A_1, \ldots, A_n, B_1,\ldots, B_p$ are real Banach spaces and $A = \oplus_{i=1}^n A_i$ and $B = \oplus_{j=1}^p B_j$ are the direct sum Banach spaces.  If $U \subset A$ is open and $F:U \rightarrow B$, one can write $F(a) = (F_1(a), \ldots, F_p(a))$ where the $F_j: A \rightarrow B_j$.  Assume $F$ is differentiable at $a$ with derivative denoted by $DF(a)$.  Just as in multivariable calculus, $DF(a)$ can be canonically represented as a matrix:

\begin{eqnarray*}                            
DF(a)  & = & \begin{bmatrix} \partial_1F_1(a) & \cdots & \partial_nF_1(a) \\ 
                                              \vdots &  & \vdots \\
                                              \partial_1F_p(a) & \cdots & \partial_nF_p(a) \\ 
                      \end{bmatrix}\\                               
\end{eqnarray*}
where the $\partial_iF_j(a): E_i \rightarrow F_j$ are defined exactly as in the Euclidean case.  To be clear the $i$th partial derivative of $F_j$ exists at $a$ if there exists a bounded (real) linear map $D_iF_j(a):E_j \rightarrow F_j$ such that

\begin{eqnarray*}
\lim_{h \rightarrow 0} \frac{\|F_j(a_1,\ldots, a_i+h, \ldots, a_n) - F_j(a) - (D_iF_j(a))(h)\|} {\|h\|} = 0
\end{eqnarray*}
As in the Euclidean case, $F$ is smooth iff the partial derivatives are smooth.

There is an integral version of the mean value theorem here as well, namely that if $a_1, a_2 \in U$ and the line segment joining $a_1$ and $a_2$ lies in $U$, then 
\begin{eqnarray*}
F(a_2) - F(a_1) & = & \left[ \int_0^1 DF(a_1 + t(a_2 -a_1)) \, dt \right] (a_2-a_1).
\end{eqnarray*}

Note that if the $A_i$ and $B_j$ are finite dimensional, then all norms are equivalent, and differentiability in one norm is equivalent to differentiability in any other norm and the derivatives (as linear maps) are one and the same.  In particular, if $A_i = B_j = M_k(\mathbb C)$ for some fixed $k$, and $F$ is differentiable at $a$ w.r.t. the operator norm, then $F$ is differentiable at $a$ w.r.t. 
any Schatten norm, and in particular the real Hilbert space direct sum of $L^2$-norms (normalized or unnormalized) induced on $A$ and $B$.  Thus, if $F$ is a $p$-tuple of (noncommutative) $*$-polynomials, then $F$ is differentiable w.r.t. the real Hilbert space direct sum norms induced on $A$ and $B$, with a derivative equal to the derivative computed w.r.t. the direct sum operator norms induced on $A$ and $B$.

\subsection{Moment Convergence and Spectral Projections}

Given a positive operator in a tracial von Neumann algebra and a microstate, I will need to know how the associated spectral projections are related.   It will be convenient to state the results in a context slightly more general than that of microstates and towards this end I'll introduce some convenient notation.  If $X=\{x_1,\ldots, x_n\}$ and $Y=\{y_1,\ldots, y_n\}$ are $n$-tuples in tracial von Neumann algebras $(M,\varphi)$ and $(N, \psi)$, respectively, and $R, \gamma >0$ and $m \in \mathbb N$, then I will write $X \approx^{R,m,\gamma} Y$ provided that the $x_i$ and $y_i$ have operator norms no greater than $R$ and that their $*$-moments are close by an order of $m$ and $\gamma$, i.e., for any $1 \leq p \leq m$ and $1 \leq i_1,\ldots, i_p \leq n$ and $j_1, \ldots, j_p \in \{1,*\}$,
\begin{eqnarray*}
|\varphi(x_{i_1}^{j_1} \cdots x_{i_p}^{j_p}) - \psi(y_{i_1}^{j_1} \cdots y_{i_p}^{j_p})| < \gamma.
\end{eqnarray*}
If $X=\{x\}$ and $Y=\{y\}$, I will simply write this as $x \approx^{R,m,\gamma} y$.  Note that this notion makes sense when $X=\{x\}$, $Y=\{y\}$, and $x$ happens to be a real linear operator on a finite dimensional (real) vector space.  In this case the $*$-moments are taken w.r.t. the real normalized trace on the operators acting on the finite dimensional real vector space and the adjoint is replaced with the transpose.

In the following lemma $(M,\varphi)$ and $(N, \psi)$ denote tracial von Neumann algebras.

\begin{lemma}  Suppose $a \in M$ is positive with $\|a\| \leq R$.  If $c, s >0$, then there exist an $m \in \mathbb N$, $\gamma >0$, dependent only on $c,s, R$ such that for any positive operator $b \in N$ satisfying $b \approx^{R,m,\gamma} a$ 

\begin{eqnarray*} \varphi(1_{[0,c)}(a)) -s & < & \tau(1_{[0,c)}(b)) \\
                                                                          & \leq & \tau(1_{[0,c]}(b)) \\
                                                                           & \leq & \varphi(1_{[0, c]}(a)) + s.\\
\end{eqnarray*}
The same conclusion holds if $b$ is a positive semidefinite real linear operator on a finite dimensional, real vector space.
\end{lemma}

\begin{proof}  There exists a monotonic decreasing sequence of uniformly bounded continuous functions which converges pointwise to $1_{[0,\alpha]}$ on $[0,R]$ such that for any $f$ in the sequence, $\inf_{t \in [0,R]} (f(t) - 1_{[0,\alpha]}(t)) > 0$.  Similarly there exists a monotonic increasing sequence of uniformly bounded continuous functions converging pointwise to $1_{[0,\alpha)}$ on $[0,R]$ such that each for any such function $f$ in the sequence, $\inf_{t \in [0,R]}1_{[0,\alpha]}(t) - f(t) > 0$.  By Stone-Weierstrass there exist uniformly bounded sequences of real polynomial functions $f_n$ and $g_n$ on $[0,R]$ such that $f_n \rightarrow 1_{[0,c]}$ and $g_n \rightarrow 1_{[0,c)}$ pointwise and  $f_n(t) \geq 1_{[0,c]}(t)$ and $1_{[0,c)}(t) \geq g_n(t)$ for all $t \in [0,R]$.  By Lebesgue's dominated convergence theorem and the Borel spectral theorem, $\lim_{n\rightarrow \infty} \varphi(f_n(a)) = \varphi(1_{[0,c]}(a))$ and $\lim_{n\rightarrow \infty} \varphi(g_n(a)) = \varphi(1_{[0,c)}(a))$.  Fix an $N$ such that $\varphi(f_N(a)) \leq \varphi(1_{[0,c]}(a)) +s/2$ and $\varphi(1_{[0,c)}(a)) - s/2 < \varphi(g_N(a))$.  Pick $m$ and $\gamma$ (dependent only on $f_N, s, c$) such that for any positive $b \in N$ satisfying $b\approx^{R, m, \gamma}a$, $|\tau(f_N(b)) - \varphi(f_N(a))| < s/2$ and $|\tau(g_N(b)) - \varphi(g_N(a))| < s/2$.  

\begin{eqnarray*} \varphi(1_{[0,c)}(a)) -s & \leq  & \varphi(g_N(a)) - s/2 \\ & \leq & \tau(g_N(b)) \\ & \leq & \tau(1_{[0,c)}(b)) \\  & \leq & \tau(1_{[0,c]}(b)) \\
                                        & \leq & \tau(f_N(b)) \\
                                       & \leq & \varphi(f_N(a)) + s/2) \\
                                       & \leq & \varphi(1_{[0,c]}(a)) + s.\\ 
\end{eqnarray*}
This completes the first claim.  The second claim for $b$ a positive semidefinite real linear operator follows by the same argument or alternatively, by realizing $b$ as a real matrix embedded in the space of complex matrices and applying the result above.
\end{proof}

\subsection{Fuglede-Kadison-L{\"u}ck Determinant, Spectral Projections, Finiteness Properties, Rank, Nullity}

Geometric decay is a condition on the traces of the spectral projections of a positive operator.  It turns out to be equivalent to the condition that the Fuglede-Kadison-L{\"u}ck determinant is nonzero (Section 6, Lemma 6.2), also known as being of \textbf{determinant class} \cite{luckbook} (see the references therein, including \cite{burg}, as well as \cite{pdh}).  

Suppose $x$ is an element in the tracial von Neumann algebra $(M,\varphi)$.  Denote by $\mu$ the spectral distribution of $|x|$ induced by $\varphi$ and by $E_t$ the spectral measure for $|x|$.  To be clear, $E_t$ is the projection-valued measure obtained from extending the continuous functional calculus on $|x|$.  Recall the Fuglede-Kadison-L{\"u}ck Determinant of $x$, referred to as the 'generalized Fuglede-Kadison determinant' in \cite{luckbook}.  Borrowing the terminology and implied name in the  exposition \cite{pdh} this is the common quality
\begin{eqnarray*} \text{det}_{FKL}(x) & = & \exp \left ( \lim_{\epsilon \rightarrow 0^+} \int_{\epsilon}^{\infty} \log t \, d\varphi(E_t) \right) \\ & = &  \exp \left (\int_{(0,\infty)} \log(\lambda) \, d\mu(\lambda) \right)\\
\end{eqnarray*}
when the integral is finite, and $0$ otherwise.  Notice that $\det_{FKL}(x) \in [0,\infty)$.  $x$ is said to be of determinant class (\cite{luckbook}) when $\det_{FKL}(x) >0$.  

I want to review here three properties of determinant class/geometric decay (expressed in terms of traces of spectral projections): monotonicity under operator ordering, invariance under row reduction operations, and upper-triangular formulas.  Choosing to phrase it in terms of determinant class or spectral projections is a matter of taste/convenience.  Both formulations are useful.

The operator ordering result will be expressed in terms of traces of spectral projections.  It relates the ordering of positive elements to their spectral distributions.  
\begin{WIFP} Suppose $0 \leq a \leq b$ are elements in a tracial von Neumann algebra $(M, \varphi)$.  For any $t >0$, 
\begin{eqnarray*}
\varphi(1_{[0,t]}(a)) & \geq & \varphi(1_{[0,t]}(b)).
\end{eqnarray*}
\end{WIFP}
In the matrix case (and thus by routine approximation for operators embeddable into an ultraproduct of the hyperfinite $\mathrm{II}_1$-factor), the above inequality follows from Weyl's inequality.  It holds in the general context of a tracial von Neumann algebra (e.g., Lemma 2.5 (iii) in \cite{fk}).  

The second property will show that when performing finitely many elementary row operations on derivatives one can retain control of spectral projections and rank.   This is obvious in the finite dimensional case and in the tracial case it's just a matter of writing out the analogous notions and drawing the natural connections.

Denote by $\pi:M \rightarrow B(L^2(M))$ the left regular representation of $M$ on $L^2(M)$.  For any $m,n \in \mathbb N$ denote by $M_{m \times n}(M)$ the set of bounded, complex linear operators $T: \oplus_{j=1}^n L^2(M) \rightarrow \oplus_{k=1}^m L^2(M)$ such that the canonical matrix representation of $T$ is of the form
\begin{eqnarray*} 
\begin{bmatrix}
T_{11} & \cdots & T_{1n} \\ \vdots &  & \vdots \\
T_{m1} & \cdots  & T_{mn} \\
\end{bmatrix}
\end{eqnarray*}
with $T_{ij} \in \pi(M)$.  $M_n(M)$ will be shorthand for $M_{n \times n}(M)$.   $T^*T \in M_n(A)$.  The \textbf{Nullity} and \textbf{Rank} of $T$ are $\text{Nullity}(T) = n \cdot (tr_n \otimes \varphi)(1_{\{0\}}(T^*T))$ and $\text{Rank}(T) = n \cdot (tr_n \otimes \varphi)(1_{(0,\infty)}(T^*T))$ where $tr_n$ is the normalized trace on the $n \times n$ complex matrices.   

Basic linear algebra facts carry over to the tracial von Neumann algebra context.  For example, $\text{Rank}(T) + \text{Nullity}(T) = n$ (rank-nullity equation), $\text{Rank}(T) = m \cdot (tr_m \otimes \varphi)(1_{(0,\infty)}(TT^*))$ (rank of an operator equals the rank of its adjoint), and if $S \in M_n(M)$, then $\text{Rank}(TS) \leq \text{Rank}(T)$.  Also, as $T^*T $ is an element in the tracial von Neumann algebra $M_n(M)$, $\det_{FKL}(T)$ is well-defined.  One can do all of this in a more algebraic, bimodular setting as in \cite{luckbook}) but I'll use the above approach to maintain the analogy with the rank theorem.

Given $T \in M_{m \times n}(M)$, one can perform elementary row column operations such as multiplying a row by an invertible element of $M$, permuting rows, and adding an $M$-left-multiple of one row to another.  As in linear algebra, each of these operations can be uniquely expressed by multiplying $T$ from the left by an invertible, "elementary" matrix $E \in M_{m \times m}(M)$.  Two matrices $S, T \in M_{m \times n}(M)$ are said to be \textbf{$M$-row equivalent} if there exists a finite product $M$ of elementary matrices $E$ in $M_{m \times m}(M)$ such that $S = ET$.  Note here that $E$ is invertible.  The following slightly more general terminology will be convenient.  If $S \in M_{p \times n}(M)$ and $T \in M_{m \times n}(M)$ with $p <m$, then $S$ and $T$ are $M$-row equivalent iff $S_0, T \in M_{m \times n}(M)$ are $M$-row equivalent where $S_0$ is the element of $M_{m \times n}(M)$ obtained by taking $S$ and turning it into an $M_{m\times n}(M)$ by stacking from below, $m-p$ rows of zeros of length $n$.  Notice that with this terminology, $|S| = |S_0|$.

Here is the elementary lemma which I'll need:

\begin{lemma} If $x,y,z \in M$, then for any $t >0$, $\varphi(1_{[0, t\|x\|\|z\|]})(|xyz|) \geq \varphi(1_{[0,t]}(|y|))$.
\end{lemma}

\begin{proof} Notice first that for any $a \in M$, if $a = u|a|$ is the polar decomposition, then $a^* = |a|u^* = u^*(u|a|u^*)$ is the polar decomposition of $a^*$.  Traciality of $\varphi$ then implies that the moments of $|a^*|$ equal the corresponding moments of $|a|$.  Thus, $|a|$ and $|a^*|$ have the same spectral distribution.  In particular for any $t >0$, $1_{[0,t]}(|a|) = 1_{[0,t]}(|a^*|)$.  Secondly, observe that

\begin{eqnarray*} |ab|^2 & = & b^*a^*ab \\
                                         & \leq & \|a^*a\| \cdot b^*b \\ 
                                         & = & \|a^*a\| \cdot |b|^2.\\
\end{eqnarray*}

\noindent Taking square roots, $|ab| \leq \|a\| \cdot |b|$ and applying Weyl's Inequality for positive operators with the observation above shows 

\begin{eqnarray*}   \varphi(1_{[0,t]}(|ab|)) & \geq & \varphi(1_{[0,t]}(\|a\| |b|) \\
                                                              & = & \varphi(1_{[0,t\|a\|^{-1}]}(|b|).\\
\end{eqnarray*}

Now, to prove the inequality, the second observation, followed by the first observation, and then recycled once more, yields

\begin{eqnarray*} \varphi(1_{[0,t]}(|xyz|) & \geq & \varphi(1_{[0,t\|x\|^{-1}]}(|yz|) \\
                                                                & = & \varphi(1_{[0,t\|x\|^{-1}]}(|(yz)^*|) \\
                                                                & = & \varphi(1_{[0,t\|x\|^{-1}]}(|z^*y^*|) \\
                                                                & \geq & \varphi(1_{[0,t\|x\|^{-1}\|z^*\|^{-1}]}(|y^*|) \\
                                                                & = & \varphi(1_{[0,t\|x\|^{-1}\|z\|^{-1}]}(|y|)) \\
                                                                \end{eqnarray*}
\noindent Rescaling $t$ finishes the proof.
\end{proof}

By Lemma 2.10, one has the following elementary observation:

\begin{corollary} If $S, T \in M_{m \times n}(M)$ are $M$-row equivalent, then $\text{Rank}(S) = \text{Rank}(T)$ and $\text{Nullity}(S) = \text{Nullity}(T)$.   Moreover, there exists a $r>0$ depending only on the matrix implementing the row equivalence of $S$ and $T$ such that $\varphi(1_{[0,rt]}(|S|)) \geq \varphi(1_{[0,t]}(|T|)$.
\end{corollary}

Alternatively and equivalently, one can use basic properties of $\det_{FKL}$ to show in the above that $S$ is of determinant class iff $T$ is of determinant class.

The third and final property concerning upper triangularity is succinctly phrased in terms of determinants.   The following is a property that one would expect, given that $\det_{FKL}$ is a natural extension/analogue of the usual determinant.  The proof can essentially be found in Theorem 3.14 (2) of \cite{luckbook}: 

\begin{proposition} Suppose $S \in M_{m \times n}(M)$ is upper triangular, i.e., $x_{ij} =0$ for all $1 \leq j < i \leq n$.  If the $x_{ii}$ are injective for $1 \leq i \leq p$ and $x_{ij}=0$ for $p < i \leq m$, then $\text{Rank}(S)= p$ and
\begin{eqnarray*}
\text{det}_{FKL}(x) = \Pi_{i=1}^{p} \text{det}_{FKL}(x_i).
\end{eqnarray*}
\end{proposition}

\section{Noncommutative $*$-polynomials, derivatives, rank and nullity}

As discussed in the introduction, in order to understand the geometry of tracial von Neumann algebra level sets I want to use the rank/nullity of the derivative as a bound on its microstates dimension.   On the von Neumann algebra level these quantities are expressed as traces of spectral projections of a certain derivation.  However, this operator algebraic context is predominantly a complex-valued environment, whereas microstates dimension, at least expressed in a manifold fashion, coincides with a real valued notion of dimension.   This section deals with formalizing an algebraic framework which allows passage from the real to the complex settings.
  
There are three parts to this section.  The first deals with the embedding the real linear bounded operators on $L^2(M)$ into the $2 \times 2$ matrices of complex linear bounded operators on $L^2(M)$; an analogous result on the real linear operators on the free complex $*$-algebra on $n$ unitaries is also discussed.  The second applies this to a generalized derivation definition to arrive at an appropriate notion of rank and nullity, and proves some technical results on microstate approximation with rank and nullity.  The last subsection applies this to the case where the domain spaces are the self-adjoint or unitary elements.

\subsection{$2 \times 2$ Real Representations} 
Fix a tracial von Neumann algebra $(M, \varphi)$.  The trace implements a real inner product on $M$ given by $\langle x,y \rangle_r = \text{Re } \varphi(y^*x)$ as well as the usual complex inner product on $M$ given by $\langle x, y \rangle = \varphi(y^*x)$.  Denote by $M_1$ and $M_2$ the real subspaces of self-adjoint and skew-adjoint elements of $M$ and by $H_j$, $j=1,2$ the closures of $M_j$ w.r.t. this real inner product norm ($\subset L^2(M)$).  Note that if $\xi, \eta \in H_j$, then $\langle \xi, \eta \rangle_r = \langle \xi, \eta \rangle$.  There are natural real projections $e_j: L^2(M)\rightarrow H_j$ given by $e_1 = (I+J)/2$ and $e_2 = (I-J)/2$ where $J$ is the extension of the conjugation map to all of $L^2(M)$.  Denote by $B_{\mathbb R}(L^2(M))$ the set of all real linear operators on $L^2(M)$ which are bounded w.r.t. the real norm generated by the real inner product.  Lastly, define $\rho$ to be the bijection on $\{1,2\}$ given by $\rho(1)=2$ and $\rho(2)=1$.

\begin{lemma} If $x: H_j \rightarrow H_k$ is a bounded, real linear map, then there exists a unique bounded, complex linear map $\tilde{x}:L^2(M) \rightarrow L^2(M)$ which extends $x$.  Moreover, $e_k \tilde{x} e_j = \tilde{x} e_j$,  $e_j \tilde{x}^* e_k = \tilde{x}^* e_k$, $e_{\rho(k)} \tilde{x} e_{\rho(j)} = \tilde{x} e_{\rho(j)}$, and $e_{\rho(j)} \tilde{x}^* e_{\rho(k)} = \tilde{x}^* e_{\rho(k)}$.
\end{lemma}
\begin{proof}  Uniqueness is obvious since the complex span of $H_1$ or $H_2$ is $L^2(M)$.  It remains to establish existence and verify that the complex linear extensions satisfies the relations.  By multiplying the domain and range by $i$ when necessary, this reduces to establishing existence and the relations for such an extension when $j=k=1$.  In this case $x:H_1 \rightarrow H_1$ is real linear.  Define $\tilde{x}: L^2(M) \rightarrow L^2(M)$ by $\tilde{x}(\xi + i \eta) = x(\xi) + i x(\eta)$ for $\xi, \eta \in H_1$.  It is easy to check that $\tilde{x}$ is complex linear.  Clearly $e_1 \tilde{x} e_1 = \tilde{x} e_1$ and $e_2 \tilde{x} e_2 = \tilde{x} e_2$.  It remains to check the relations for $\tilde{x}^*$.   Consider the complex linear map $y: L^2(M) \rightarrow L^2(M)$ defined by $y(\xi + i \eta) = x^*\xi + i x^*\eta$ where here, $x^*$ is the real adjoint of $x$.  For any $\xi, \eta \in H_1$,
\begin{eqnarray*}
\langle \tilde{x} \xi, \eta \rangle & = & \langle x \xi, \eta \rangle \\
                                                 & = & \text{Re} \langle x \xi, \eta \rangle \\
                                                 & = & \text{Re} \langle \xi, x^*\eta \rangle \\
                                                 & = & \langle \xi, y \eta \rangle \\
\end{eqnarray*}
Using the complex linearity of $\tilde{x}$ and $y$, it follows that the above holds for any $\xi, \eta \in L^2(M)$, i.e., $\tilde{x}^* = y$.  But now it is clear from the definition of $y$, that $e_1 y e_1 = y e_1$ and $e_2 y e_2 = ye_2$ and thus $\tilde{x}^*$ satisfies the same relations.  
\end{proof}

For $x \in B_{\mathbb R}(L^2(M))$ the matrix decomposition of $x$ w.r.t. $L^2(M) = H_1 \oplus H_2$ is of the form
\begin{eqnarray*} 
x =  \begin{bmatrix}
x_{11}  & x_{12}  \\ 
x_{21}  & x_{22}  \\
\end{bmatrix} \\
\end{eqnarray*}
where $x_{jk}: H_k\rightarrow H_j$ are bounded, real linear operators.  By Lemma 3.1, for each $1\leq j,k \leq 2$ there exists a unique complex linear operator $\tilde{x}_{jk}: L^2(M) \rightarrow L^2(M)$ extending $x_{jk}$.  Moreover, these extensions and their adjoints automatically satisfy the relations with the $e_j$ described in Lemma 3.1.   Define $\Phi: B_{\mathbb R}(L^2(M)) \rightarrow M_2(B(L^2(M)))$ by 
\begin{eqnarray*}
\Phi(x) =  \begin{bmatrix}
\tilde{x}_{11}  & \tilde{x}_{12}  \\ 
\tilde{x}_{21}  & \tilde{x}_{22}  \\
\end{bmatrix}. \\
\end{eqnarray*}
 
\begin{proposition} $\Phi: B_{\mathbb R}(L^2(M)) \rightarrow M_2(B(L^2(M)))$ is a real linear, $*$-preserving, multiplicative map which send the identity to the identity.
\end{proposition}

\begin{proof} Suppose $x, y \in B_{\mathbb R}(L^2(M))$ with matrix decompositions $\langle x_{ij} \rangle_{1\leq i,j \leq 2}, \langle y_{ij} \rangle_{1 \leq i,j \leq 2}$ w.r.t. $H_1 \oplus H_2$.  The matrix decomposition of $rx+y$ w.r.t. $H_1 \oplus H_2$ is clearly $\langle rx_{ij} + y_{ij} \rangle_{1 \leq i,j \leq 2}$ and $r \tilde{x}_{ij} + \tilde{y}_{ij}$ is the unique complex linear extension of $r x_{ij} + y_{ij}$.  By definition then, $\Phi(rx+y) = r \Phi(x) + \Phi(y)$.  To show that $\Phi$ preserves the $*$-operation, using the identities which $\tilde{x}_{ij}$ and $\tilde{x}_{ij}^*$ must satisfy by Lemma 3.1, $x^*$ has a matrix decomposition of the form
\begin{eqnarray*} 
 \begin{bmatrix}
x_{11}  & x_{12}  \\ 
x_{21}  & x_{22}  \end{bmatrix}^*  & = & 
\begin{bmatrix}
e_1 \tilde{x}_{11} e_1 & e_1 \tilde{x}_{12} e_2 \\ 
e_2 \tilde{x}_{21} e_1 & e_2 \tilde{x}_{22} e_2 \\
\end{bmatrix}^* \\ \\ & = &  \begin{bmatrix}
e_1 \tilde{x}_{11}^* e_1  & e_1 \tilde{x}_{21}^* e_2  \\ 
e_2 \tilde{x}_{12}^* e_1  & e_2 \tilde{x}_{22}^* e_2 \end{bmatrix} \\ \\
& = & \begin{bmatrix}
\tilde{x}_{11}^* e_1  & \tilde{x}_{21}^* e_2  \\ 
 \tilde{x}_{12}^* e_1  & \tilde{x}_{22}^* e_2 \end{bmatrix}. \\
\end{eqnarray*}
By the uniqueness of the complex linear extension (Lemma 3.1),
\begin{eqnarray*}
\Phi(x^*) & = & \begin{bmatrix}
\tilde{x}_{11}^*  & \tilde{x}_{21}^*  \\ 
 \tilde{x}_{12}^* & \tilde{x}_{22}^* \end{bmatrix} \\
 & = & \Phi(x)^*.\\
\end{eqnarray*}
For multiplicativity, the matrix decomposition of $xy$ is the product of the two matrix representations of $x$ and $y$:
\begin{eqnarray*} 
\begin{bmatrix}
x_{11}  & x_{12}  \\ 
x_{21}  & x_{22}  \end{bmatrix}  \begin{bmatrix}
y_{11}  & y_{12}  \\ 
y_{21}  & y_{22}  \end{bmatrix} & = & \begin{bmatrix}
\tilde{x}_{11} e_1 & \tilde{x}_{12}  e_2 \\ 
\tilde{x}_{21} e_1  & \tilde{x}_{22} e_2 \end{bmatrix}  \begin{bmatrix}
e_1 \tilde{y}_{11} e_1  & e_1 \tilde{y}_{12} e_2  \\ 
e_2 \tilde{y}_{21} e_1 & e_2 \tilde{y}_{22} e_2 \end{bmatrix} \\ 
& = & 
\begin{bmatrix}
\tilde{x}_{11} e_1 \tilde{y}_{11} e_1 + \tilde{x}_{12}e_2 \tilde{y}_{21} e_1 & \tilde{x}_{11}e_1 \tilde{y}_{12} e_2 + \tilde{x}_{12} e_2 y_{22} e_2 \\ 
\tilde{x}_{21} e_1 \tilde{y}_{11} e_1 + \tilde{x}_{22}e_2\tilde{y}_{21}e_1 & \tilde{x}_{21}e_1\tilde{y}_{12}e_2+ \tilde{x}_{22} e_2 y_{22} e_2 \\
\end{bmatrix} \\ 
& = & \begin{bmatrix}
(\tilde{x}_{11} \tilde{y}_{11} + \tilde{x}_{12}\tilde{y}_{21}) e_1 & (\tilde{x}_{11}\tilde{y}_{12} + \tilde{x}_{12} \tilde{y}_{22}) e_2 \\ 
(\tilde{x}_{21}  \tilde{y}_{11}  + \tilde{x}_{22}\tilde{y}_{21})e_1 & (\tilde{x}_{21}\tilde{y}_{12}+ \tilde{x}_{22}  \tilde{y}_{22}) e_2 \\
\end{bmatrix}. \\
\end{eqnarray*}
The parenthetical terms of the last matrix are complex linear maps and thus by the uniqueness of the complex linear extension,
\begin{eqnarray*}
\Phi(xy) & = & \begin{bmatrix}
(\tilde{x}_{11} \tilde{y_{11}} + \tilde{x}_{12}\tilde{y}_{21})  & (\tilde{x}_{11}\tilde{y}_{12} + \tilde{x}_{12} \tilde{y}_{22})  \\ 
(\tilde{x}_{21}  \tilde{y}_{11}  + \tilde{x}_{22}\tilde{y}_{21}) & (\tilde{x}_{21}\tilde{y}_{12}+ \tilde{x}_{22}  \tilde{y}_{22})  \\
\end{bmatrix} \\
& = & \begin{bmatrix}
\tilde{x}_{11}  & \tilde{x}_{12}  \\ 
\tilde{x}_{21}  & \tilde{x}_{22}  \end{bmatrix}  \begin{bmatrix}
\tilde{y}_{11}  & \tilde{y}_{12}  \\ 
\tilde{y}_{21}  & \tilde{y}_{22}  \end{bmatrix} \\
& =&  \Phi(x) \Phi(y).\\
\end{eqnarray*}
The claim concerning the identity is trivial.
\end{proof}

\begin{example}  Suppose $(M,\varphi)$ is a tracial von Neumann algebra and consider its left and right actions, $\pi_r$ and $\pi_l$ on $L^2(M)$.  For $x \in M$, $\pi_r(x) \in B(L^2(M)) \subset B_{\mathbb R}(L^2(M))$.  What is $\Phi(\pi_l(x))$?  Restricting $\pi_l(x)$ to the real linear subspaces $M^{sa}$ and $M^{sk}$, it is straightforward to show that
\begin{eqnarray*}
\Phi(\pi_l(x)) & = & \frac{1}{2} \cdot\begin{bmatrix}
\pi_l(x) + \pi_r(x^*) & \pi_l(x) - \pi_r(x^*)  \\ 
\pi_l(x) - \pi_r(x^*)  & \pi_l(x) + \pi_r(x^*)  \end{bmatrix} \\
& \in & M_2(B(L^2(M))).\\
\end{eqnarray*}
In particular, when $M=\mathbb C$ and $x = \lambda = a + ib \in \mathbb C$ with $a, b\in \mathbb R$, then the above equation becomes
\begin{eqnarray*}
\Phi(\pi_l(x)) & = & \begin{bmatrix}
a & ib  \\ 
ib  & a  \end{bmatrix} \\
& \in & M_2(\mathbb C).\\
\end{eqnarray*}
Compare this to the canonical real matrix representation of $x=\lambda$ \emph{relative to the real basis} $\{1, i\} \subset \mathbb C$:
\begin{eqnarray*}
\begin{bmatrix}
a  & -b  \\ 
b  & a  \\
\end{bmatrix} \\
\end{eqnarray*}
\end{example}

Proposition 3.2 can be applied to $M = M_k(\mathbb C)$ in which case one can also account for the action of the traces.  Denote by $Tr_{B_{\mathbb R}(M_k(\mathbb C))}$ the unnormalized real trace on $B_{\mathbb R}(M_k(\mathbb C))$ and by $Tr_{M_2(B(M_k(\mathbb C)))}$ the unnormalized, complex trace on $M_2(B(M_k(\mathbb C)))$, i.e., the trace on the $2\times 2$ matrices with entries regarded as complex linear operators on the complex vector space $M_k(\mathbb C)$.  It is transparent from the definition of $\Phi$ that $Tr_{B_{\mathbb R}(M_k(\mathbb C))} = Tr_{M_2(B(M_k(\mathbb C)))} \circ \Phi$.  Hence the following holds:

\begin{lemma} $\Phi : (B_{\mathbb R}(M_k(\mathbb C)), Tr_{B_{\mathbb R}(M_k(\mathbb C))}) \rightarrow (M_2(B(M_k(\mathbb C))), Tr_{M_2(B(M_k(\mathbb C)))})$ is a real linear, multiplicative, $*$-preserving, injection which preserves the unnormalized traces.
\end{lemma}

An analogous derivation can be performed for the universal complex, unital $*$-algebra $B_n$, on $n$ unitary generators.  Denote by $\mathcal A_1$ and $\mathcal A_2$ the real $*$-subalgebras of $B_n$ consisting of self-adjoint and skew-adjoint elements of $B_n$.  Notice that $\mathcal A_2 = i \mathcal A_1$.   Define $e_1: B_n \rightarrow \mathcal A_1$ by $e_1(x) = (x+x^*)/2$ and $e_2 = I-e_1$.  $e_1$ and $e_2$ are the real idempotent projection maps onto $\mathcal A_1$ and $\mathcal A_2$, respectively. Arguing as in Lemma 3.1 yields: 

\begin{lemma} If $x: \mathcal A_i \rightarrow \mathcal A_j$ is real linear, then there exists a unique complex linear map $\tilde{x}: B_n \rightarrow B_n$ which extends $x$.  Moreover, $e_j \tilde{x} e_i = \tilde{x} e_i$, $e_{\rho(j)} \tilde{x} e_{\rho(i)} = \tilde{x} e_{\rho(i)}$.
\end{lemma}

Denote by $L_{\mathbb R}(B_n)$, $L_{\mathbb C}(B_n)$ the space of real  and complex linear operators on $B_n$.  Exactly as before, for $x \in L_{\mathbb R}(B_n)$ the matrix decomposition of $x$ w.r.t. $B_n = \mathcal A_1 \oplus \mathcal A_2$ (here $\oplus$ denotes the algebraic direct sum) is of the form
\begin{eqnarray*} 
x =  \begin{bmatrix}
x_{11}  & x_{12}  \\ 
x_{21}  & x_{22}  \\
\end{bmatrix} \\
\end{eqnarray*}
where $x_{ij}: \mathcal A_j \rightarrow \mathcal A_i$ are real linear operators.  By Lemma 3.4, for each $1\leq i,j \leq 2$ there exists a unique complex linear operator $\tilde{x}_{ij}: \mathfrak{A}_n \rightarrow \mathfrak{A}_n$ extending $x_{ij}$.  Moreover, these extensions automatically satisfy the relations $e_i \tilde{x}_{ij} e_j=\tilde{x}_{ij} e_j$ and $e_{r(i)} \tilde{x}_{ij} e_{r(j)}=\tilde{x}_{ij} e_{r(j)}$.   Define again $\Phi: L_{\mathbb R}(B_n) \rightarrow M_2(L_{\mathbb C}(B_n))$ by 
\begin{eqnarray*}
\Phi(x) =  \begin{bmatrix}
\tilde{x}_{11}  & \tilde{x}_{12}  \\ 
\tilde{x}_{21}  & \tilde{x}_{22}  \\
\end{bmatrix} \\
\end{eqnarray*}
As before, $\Phi$ is a real linear, multiplicative map which sends the identity to the identity.

Finally, note that if $\odot$ denotes the algebraic tensor product, then there exists a complex linear, homomorphism $\pi: B_n \odot B_n \rightarrow L_{\mathbb C}(B_n)$ uniquely determined by $\pi(a \odot b) = L_{a,b}$ where $L_{a,b} : B_n \rightarrow B_n$ is defined by $L_{a,b}(x) = axb$.  Moreover, using the universality of $B_n$, $\pi$ is injective for $n >1$.  This can be seen by identifying $B_n$ with the complex $*$-algebra generated by the $n$ canonical unitaries in $L(\mathbb F_n)$.  From here, it is enough to show that for any finite list of $2$-tuples of elements in $\mathbb F_n$, $(a_i,b_i)$ with $a_i \neq e$, there exists a single group element $g \in \mathbb F_n$ such that for all $i$, $ga_ig^{-1} \neq b_i$.  Alternatively, using factoriality of $L(\mathbb F_n)$ one can invoke invoke Corollary 4.20 of \cite{takesaki}.  In any case, putting all this together with the fact that the algebraic tensor product of injective maps is injective, one has the following result which will be useful for the unitary calculus:
\begin{lemma} If $n >1$, then $\pi \odot I_2:  M_2(B_n \odot B_n) \rightarrow M_2(L_{\mathbb C}(B_n))$ is an injective $*$-homomorphism.  
\end{lemma}
\subsection{Derivatives, Derivations, Rank, and Nullity}

Throughout this section $\odot$ will designate the algebraic tensor product and $\otimes$ will denote the von Neumann algebra tensor product.  If $(A, \varphi)$, $(B, \psi)$ are tracial complex $*$-algebras, then there exists a trace on $A \odot B$ given by $(\varphi \odot \psi)(a\odot b) = \varphi(a) \cdot \psi(b)$. Similarly if $(A, \varphi)$ and $(B, \psi)$ are tracial von Neumann algebras, then there exists a canonical tracial state on $A \otimes B$ given by $(\varphi \otimes \psi)(a \otimes b) = \varphi(a) \cdot \psi(b)$.  In this case $A \odot B$ canonically embeds into $A \otimes B$ by a trace-preserving, injective, unital $*$-homomorphism $\iota$ (this is a bloated way of saying that the minimal norm is in fact a norm and not just a seminorm).

The following discussion will motivate the definition of $D^sF(X)$.  This notion is defined in terms of derivations and will be convenient for describing operator calculus.  It will be important to connect this to the ordinary derivative which will be used in the microstate setting.  One can succinctly express their connection in terms of a commutative diagram (Remark 3.8) which will be fleshed out in the discussion below.
   
Suppose $\mathfrak{A}_n$ is the universal, unital complex $*$-algebra on $n$ generators $\mathbb X= \{X_1, \ldots, X_n\}$. $\mathfrak{A}_n^{op}$ is the opposite algebra associated to $\mathfrak{A}_n$.  For each $1 \leq j \leq n$ denote by $A_j$ the unital $*$-algebra generated by $\mathbb X - \{X_j\}$.  If $\xi =\{\xi_1,\ldots, \xi_n\}$ is an $n$-tuple in a tracial von Neumann algebra $M$, then there exist canonical $*$-homomorphisms $\pi_{\xi}: \mathfrak{A}_n \rightarrow M$, $\pi_{\xi}^{op}: \mathfrak{A}_n \rightarrow M^{op}$, such that $\pi_{\xi}(X_i) = \xi_i \in M$ and $\pi^{op}_{\xi}(X_i) = \xi_i^{op} \in M^{op}$, $1\leq i \leq n$.  Denote by $\sigma_M : M \odot M^{op} \rightarrow B(L^2(M))$ the unique complex $*$-homomorphism such that for any $a, b, \xi \in M$ and $\xi \in M \subset L^2(M)$, $\sigma_M(a \odot b^{op})(\xi) = a \xi b$.  Define $\pi_M(\xi): \mathfrak{A}_n\odot \mathfrak{A}_n^{op} \rightarrow B(L^2(M))$ by $\pi_M = \sigma_M \circ (\pi_{\xi} \odot \pi^{op}_{\xi})$.

Fix $1 \leq r \leq n$, $f \in \mathfrak{A}_n$, and set $A = A_r$, $x= X_r$.  $f$ induces a well-defined function from $\oplus_{i=1}^n M$ into $M$, which will be denoted again by $f$.  Regarding $\oplus_{i=1}^n M$ and $M$ as real Banach spaces with the operator norm, $f$ is differentiable (the operator norm is necessary, since the $L^2$-norm is not submultiplicative).  In particular, writing $f = \sum \lambda_{i_1,\ldots, i_{p+1}} a_{i_1} x^{q_{i_1}} \cdots a_{i_p} x^{q_{i_p}} a_{i_{p+1}}$ where $q_{i_j} \in \{1, *\}$ and $a_{i_j} \in A$, the partial derivative of $f$ with respect to the jth variable at $\xi = (\xi_1,\ldots, \xi_n) \in \oplus_{i=1}^n M$ is the bounded, real linear map $\partial_jf(\xi)$ on $M$ given by 

\begin{eqnarray*} & & \sum \lambda_{i_1,\ldots, i_{p+1}} \left [ \pi_M(\xi)(a_{i_1} \odot a_{i_2} \cdots a_{i_p} x^{q_{i_p}} a_{i_{p+1}}) J^{\delta_{q_{i_1}, *}} +  \cdots + \pi_M(\xi)(a_{i_1} x^{q_{i_1}} a_{i_2} \cdots a_{i_p} \odot a_{i_{p+1}}) J^{\delta_{q_{i_p}, *}} \right] \\
\end{eqnarray*}

\noindent where $J$ is the (real linear) conjugation extension on $L^2(M)$.   This is effectively the free partial derivative except that the appearance of a conjugated element requires an added $J$ term.  As in the previous subsection set $M_1 = M^{sa}$, $M_2 =i M^{sa}$, $H_i = L^2(M_i)$.  The partial derivative operator $\partial_jf(\xi)$ above is clearly bounded w.r.t. the $L^2$-norm and one can regard $\partial_jf(\xi)$ as an element of $B_{\mathbb R}(L^2(M))$.  Applying $\Phi$ of the preceding section to $\partial_jf(\xi)$ so regarded, one arrives at a $2 \times 2$ matrix representation of $\partial_jf(\xi)$ realized as an element of $M_2(\sigma_M(M \odot M^{op})) \subset M_2(B(L^2(M)))$.  Recall that this is obtained by breaking the action of $\partial_jf(\xi)$ up into actions on the $H_i$, and then extending these restricted real linear operators into complex linear operators. 

Alternatively, the resultant matrix representation of $\partial_jf(\xi)$ can be expressed in terms of derivations.  To see this note that $\{x,x^*\}$ generates the free complex $*$-algebra generated by a single indeterminate and $A=A_r$ is free from the unital complex algebra generated by $\{x,x^*\}$.  Endowing $A[x,x^*] \odot A[x,x^*]^{op}$ with the canonical $A[x,x^*]$-$A[x,x^*]^{op}$ bimodule structure, define the complex linear derivations $\partial_{sa}, \partial_{sk} : A[x,x^*] \rightarrow A[x,x^*] \odot A[x,x^*]^{op}$ by the relations $\partial_{sa}(A) = \partial_{sk}(A) = 0$, $\partial_{sa}(x) = \partial_{sa}(x^*) = \partial_{sk}(x) = I \odot I$ and $\partial_{sk}(x^*) = -I \odot I$. The $\Phi$-matrix of the operator $\partial_jf(\xi)$ with respect to the real inner product decomposition $L^2(M) = H_1 \oplus H_2$ can be described with $\partial^{sa}$ and $\partial^{sk}$.  Consider

\begin{eqnarray*} \begin{bmatrix}
e_1 \pi_M(\xi)(\partial_{sa} f) e_1 & e_1\pi_M(\xi)(\partial_{sk} f) e_2 \\ 
e_2 \pi_M(\xi)(\partial_{sa} f) e_1 & e_2\pi_M(\xi)(\partial_{sk} f) e_2 \\
\end{bmatrix} =
\begin{bmatrix}
\pi_M(\xi)(\partial_{sa} f_1) e_1 & \pi_M(\xi)(\partial_{sk} f_1) e_2 \\ 
\pi_M(\xi)(\partial_{sa} f_2) e_1 & \pi_M(\xi)(\partial_{sk} f_2) e_2 \\
\end{bmatrix}\\
\end{eqnarray*}

\noindent where $f_1 = (f+f^*)/2$, $f_2 = (f-f^*)/2$, $e_1 = (I + J)/2$, and $e_2 = (I-J)/2$.  This is the matrix representation of $\partial_j f(\xi)$ w.r.t. $H_1 \oplus H_2$ .  $\Phi(\partial_jf(\xi))$ is then

\begin{eqnarray*} 
\begin{bmatrix}
\pi_M(\xi)(\partial_{sa} f_1) & \pi_M(\xi)(\partial_{sk} f_1)  \\ 
\pi_M(\xi)(\partial_{sa} f_2)  & \pi_M(\xi)(\partial_{sk} f_2)  \\
\end{bmatrix} & \in & M_2(B(L^2(M))).\\
\end{eqnarray*}

From the standpoint of $*$-moments, however, these expressions are problematic.  This is because they all involve $\pi_M$ which is defined with $\sigma_M$ which in turn isn't always an embedding.  However, omitting the application of $\sigma_M$ in the definition $\pi_M = \sigma_M \circ (\pi_{\xi} \odot \pi_{\xi}^{op})$ in the above yields
\begin{eqnarray*} 
\begin{bmatrix}
(\pi_{\xi} \odot \pi^{op}_{\xi})(\partial_{sa} f_1) &(\pi_{\xi} \odot \pi^{op}_{\xi})(\partial_{sk} f_1)  \\ 
(\pi_{\xi} \odot \pi^{op}_{\xi})(\partial_{sa} f_2)  & (\pi_{\xi} \odot \pi^{op}_{\xi})(\partial_{sk} f_2)  \\
\end{bmatrix} & \in & M_2(M \odot M^{op}) .\\
\end{eqnarray*}
 This matrix has $*$-moments which can be described entirely in terms of the moments of $\xi$ and (after applying the dummy embedding $\iota$ which will replace the algebraic tensor product with the von Neumann algebra tensor product) its entries will lie in the tracial von Neumann algebra $M \otimes M^{op}$.  
 
Notice that in the special case of $M = M_k(\mathbb C)$ all these matricial expressions and their images under $\Phi$ are elements in tracial von Neumann algebras and the expressions have the same $*$-moments.   Indeed in this case $\Phi$ is a trace-preserving, real $*$-embedding by Lemma 3.3, $\sigma_M$ is a trace-preserving $*$-isomorphism, and $M \odot M^{op} \simeq M \otimes M^{op} \simeq B(L^2(M))$.  
 
 The discussion motivates the following definitions and remark.  

\begin{definition} Suppose that $\mathfrak{A}$ is a unital, complex $*$-algebra and $x \in \mathfrak{A}$ such that the unital $*$-algebra generated by $x$, $\mathbb C[x,x^*]$, is isomorphic to the universal complex $*$-algebra on a single element.  Assume moreover that $A \subset \mathfrak{A}$ is a unital $*$-algebra and $A$ and $\mathbb C[x,x^*]$ are algebraically free and denote the unital $*$-subalgebra they generate by $A[x,x^*]$.  The maps $\partial_{sa}, \partial_{sk} : A[x,x^*] \rightarrow A[x,x^*] \odot A[x,x^*]^{op}$ are the unique derivations defined by the relations $\partial_{sa}(A) = \partial_{sk}(A) = 0$, $\partial_{sa}(x) = \partial_{sa}(x^*) = \partial_{sk}(x) = I \odot I^{op}$ and $\partial_{sk}(x^*) = -I \odot I^{op}$, where $A[x,x^*] \odot A[x,x^*]^{op}$ is given the natural $A[x,x^*]-A[x,x^*]^{op}$ bimodule structure.  For $f \in A[x,x^*]$ write $f = f_1 + f_2$ where $f_1 = (f + f^*)/2$ and $f_2 = (f-f^*)/2$.  Define

\begin{eqnarray*} D_{(x,x^*)}f = \begin{bmatrix}
\partial_{sa} f_1 & \partial_{sk} f_1 \\ 
\partial_{sa} f_2 & \partial_{sk} f_2 \\
\end{bmatrix} \in M_2( A[x,x^*] \odot A[x,x^*]^{op}).\\
\end{eqnarray*}  
\end{definition}

\begin{definition} Fix an $n$-tuple $X=\{x_1,\ldots, x_n\}$ in the tracial von Neumann algebra $M$ and suppose $f \in \mathfrak{A}_n$ and $1 \leq j \leq n$.  Suppose $\mathbb X = \{X_1,\ldots, X_n\}$ are the canonical generators for $\mathfrak{A}_n$ and $A_j$ is the unital $*$-subalgebra generated by $\mathbb X - \{X_j\}$.   Set $x =X_j$, $A=A_j$, and consider $D_{(x,x^*)}$ in Definition 3.6.  The jth partial S-derivative of $f$ at $X$ is the element 
\begin{eqnarray*}
(\partial_j^s f)(X) & = & (\pi_{X} \otimes \pi_{X}^{op} \otimes I_2) \circ (D_{(x_j,x_j^*)}f) \\ & \in & M_2(M \odot M^{op}) \\
& \subset & M_2(M \otimes M^{op}).\\
\end{eqnarray*}
\end{definition}
 
\begin{remark} Suppose $f \in \mathfrak A_n$ and $X = \{x_1,\ldots, x_n\}$ is an $n$-tuple in the tracial von Neumann algebra $M$.  As before $f$ can be regarded as a smooth function from the real Banach space $\oplus_{i=1}^n M$ into $M$.   Fix $1 \leq j \leq n$ and set $A= A_j \subset \mathfrak{A}_n$, $x=X_j \in \mathfrak{A}_n$.  Denote by $(\partial_jf)(X)$ the jth partial derivative operator of $f$ evaluated at $X$, regarded as a bounded real linear map on the real Hilbert space $L^2(M)$.   The discussion and definitions above yield the following commutative diagram:

\begin{eqnarray*}
\begin{tikzpicture}[scale=4]
\node (A) at (0,1.2) {$B_{\mathbb R}(L^2(M))$};
\node (B) at (1.2,1.2) {$\mathfrak{A}_n$};
\node (C) at (0,0) {$M_2(\sigma_M(M \otimes M^{op}))$};
\node (D) at (1.2,0) {$M_2(M \odot M^{op})$};
\node (E) at (1.8,0.6) {$M_2(A[x,x^*] \odot A[x,x^*]^{op})$};
\node (F) at (2.2, 0) {$M_2(M \otimes M^{op})$};
\path[->,font=\scriptsize]
(B) edge node[above]{$(\partial_j \cdot)(X)$} (A)
(A) edge node[left]{$\Phi$} (C)
(B) edge node[left]{$(\partial_j^s \cdot)(X)$} (D)
(B) edge node[midway, right]{$\hspace{.1in} D_{(x,x^*)}$} (E)
(E) edge node [below, right]{$\hspace{.1in}(\pi_X \odot \pi_X^{op}) \otimes I_2$} (D)
(D) edge node[above]{$\sigma_M \otimes I_2$} (C)
(D) edge [commutative diagrams/hook] node[above]{$\iota \otimes I_2$} (F);
\end{tikzpicture}
\end{eqnarray*}
\end{remark}

\begin{definition} Fix an $n$-tuple $X=\{x_1,\ldots, x_n\}$ in the tracial von Neumann algebra $M$ and an $m$-tuple $F= \{f_1,\ldots,f_m\}$ in $\mathfrak{A}_n$.  For each $1\leq i \leq m$, $1 \leq j \leq n$, consider the $\partial^s_jf_i(X)$, the $j$th partial S-derivative of $f_i$ at $X$ realized in the von Neumann algebra tensor product $M_2(M \otimes M^{op})$.  The S-derivative of $F$ at $X$ is 
\begin{eqnarray*} D^sF(X)=
\begin{bmatrix}
(\partial_1^s f_1)(X) & \cdots & (\partial^s_n f_1)(X) \\ \vdots &  & \vdots \\
(\partial_1^s f_m)(X) & \cdots  & (\partial^s_n f_m)(X) \\
\end{bmatrix} \in M_{m\times n}(M_2(M \otimes M^{op})).
\end{eqnarray*}
\end{definition}


The S-derivative of $F$ at $X$ is an operator whose $*$-moments (when defined) can be characterized in terms of the moments of $X$.  It is naturally connected to the ordinary derivative of $F$ at $X$ when $F$ is regarded as a function on a Banach space direct sum of copies of $M$ (Remark 3.8 and the commutative diagram).  Moreover, one can make sense of its rank and nullity as described in Section 2.6 when $D^sF(X)$ is regarded as a matrix with $M\otimes M^{op}$ entries.  I'll show shortly that this rank and nullity reflect the rank and nullity of the corresponding (ordinary) derivative of $F$ at a microstate.  These notions will be used to establish global dimension and entropy bounds on the microstate spaces.

\begin{definition} Consider $D^sF(X)$ regarded as an element $M_{2m \times 2n}(M\otimes M^{op}) = M_{m\times n}(M_2(M \otimes M^{op}))$.  $\text{Nullity}(D^sF(X))$ and $\text{Rank}(D^sF(X))$ are the rank and nullity of $D^sF(X)$ so regarded and in the sense of Section 2.6.  More explicitly set $P = 1_{\{0\}}(D^sF(X)^* D^sF(X))$ and $Q = 1_{(0,\infty)}(D^sF(X)^* D^sF(X))$.   $P, Q \in M_{2n}(M \otimes M^{op})$.  Consider the tracial state $\psi = tr_{2n} \otimes (\varphi \otimes \varphi^{op})$ on $M_{2n}(M \otimes M^{op})$.  The nullity and rank of $F$ at $X$ are

\[  \text{Nullity}(D^sF(X)) = 2n \cdot \psi(P) \in [0,2n]
\]

\noindent and

\[  \text{Rank}(D^sF(X)) = 2n \cdot \psi(Q) \in [0,2n],
\]

\noindent respectively.
\end{definition}

\begin{remark} When $G = \{g_1,\ldots,g_p\}$ is another $p$-tuple of elements in $\mathfrak{A}_n$, then one can consider the obvious $(n+p)$-tuple $F \cup G$ and $D^s(F \cup G)(X)$ denotes the element of $M_{(m+p) \times n}(M_2(M \otimes M^{op}))$ obtained by 'stacking' $D^sF(X)$ on top of $D^sG(X)$.  In this case if $Q_F$ is the projection onto the kernel of $D^sF(X)$ and $Q_G$ is the projection onto the kernel of $D^sG(X)$, then the projection onto the kernel of $D^s(F \cup G)(X)$ is $Q_F \wedge Q_G$.  
\end{remark}

\subsection{Microstate approximations}

In this section I'll see how properties of the $S$-derivative of a tuple $F$ at $X$ are reflected by microstates of $X$.   Suppose $A$ and $B$ are tracial $*$-algebras (real or complex).  Given an indexing set $J$ for subsets $X =\langle x_j \rangle_{j \in J}$ and $Y = \langle y_j \rangle_{j \in J}$ in $A$ and $B$, I will write $X \approx Y$ if $X$ and $Y$ have the same $*$-moments.  

Suppose for each $1 \leq i \leq p$, $f_i \in \mathfrak{A}_n$ and $F = (f_1, \ldots, f_p)$.  Regard $F$ as a map from $(M_k(\mathbb C))^n$ into  $(M_k(\mathbb C))^p$ where $F(\xi) = (f_1(\xi),\ldots, f_p(\xi))$, $\xi \in (M_k(\mathbb C))^n$.  As such the derivative of $F$ at $\xi \in (M_k(\mathbb C))^n$ has the matrix representation

\[ DF(\xi) =
\begin{bmatrix} 
\partial_1 f_1(\xi) & \cdots & \partial_n f_1(\xi) \\ \vdots &  & \vdots \\
\partial_1 f_p(\xi) & \cdots  & \partial_n f_p(\xi) \\
\end{bmatrix} \in M_{p\times n}(B_{\mathbb R}(M_k(\mathbb C)))
\]

\noindent where $\partial_if_j(\xi)$ are the partial derivatives of the $f_i$ regarded as functions from $\oplus_{i=1}^n M_k(\mathbb C)$ to $M_k(\mathbb C)$; this was discussed in the preceding subsection in the context of a general von Neumann algebra.  Proposition 3.3 asserts that $\Phi$ is a real, trace-preserving $*$-embedding. $\sigma_M \otimes I_2$ is a bijective, trace-preserving $*$-isomorphism of complex $*$-algebras for $M = M_k(\mathbb C)$ .  Thus, by the commutative diagram of Remark 3.8,
\begin{eqnarray*}
 \langle (\partial_if_j)(\xi) \rangle_{1\leq  i \leq p, 1\leq j \leq n} & \approx & \langle \Phi((\partial_if_j)(\xi)) \rangle_{1 \leq i \leq p, 1 \leq j \leq n} \\                                    & \approx & \langle (\sigma_M \otimes I_2)(\partial^s_i f_j)(\xi) \rangle_{1 \leq i \leq p, 1 \leq j \leq n}\\      & \approx & \langle (\partial^s_i f_j)(\xi) \rangle_{1 \leq i \leq p, 1 \leq j \leq n}\\                                      
\end{eqnarray*}
Denote by $\Lambda =\langle e_{ij} \rangle_{1 \leq i,j \leq n}$ the canonical system of matrix units for $M_n(\mathbb C)$.  Set
\[
\Lambda_1 = \langle e_{ij} \otimes I_{B_{\mathbb R}(M_k(\mathbb C))} \rangle_{1 \leq i,j \leq n},
\] 
\[
\Lambda_2 = \langle e_{ij} \otimes I_{M_2(M_k(\mathbb C) \otimes M_k(\mathbb C)^{op})} \rangle_{1 \leq i,j \leq n},
\] 
and
\[
\Lambda_3 = \langle e_{ij} \otimes I_{M_2(M \otimes M^{op})} \rangle_{1 \leq i,j \leq n}.
\]  
It follows from the $*$-distributional equivalences above that 
\begin{eqnarray*}
(DF(\xi)^*DF(\xi), \Lambda_1) \subset M_n(B_{\mathbb R}(M_k(\mathbb C)) =M_n(\mathbb C) \otimes B_{\mathbb R}(M_k(\mathbb C))
\end{eqnarray*}
and 
\begin{eqnarray*}
(D^sF(\xi)^*D^sF(\xi), \Lambda_2) \in M_n(M_2(M_k(\mathbb C) \otimes M_k(\mathbb C)^{op})) = M_n(\mathbb C) \otimes (M_2(M_k(\mathbb C) \otimes M_k(\mathbb C)^{op}))
\end{eqnarray*}
have the same $*$-moments.  Notice that the algebras on the right hand sides are given the canonical real and complex traces, respectively.

Now the trace of any $*$-monomial of the entries of $D^sF(\xi)^*D^sF(\xi)$ and the trace of the corresponding $*$-monomial of the entries of $D^sF(X)^*D^sF(X)$ can be made arbitrarily close provided that the $*$-moments of $\xi$ are sufficiently close to those of $X$.   Thus, for any $R >0$, there exists an $R_1 >0$ dependent only on $F$ and $X$ such that for any given $m_1 \in \mathbb N$ and $\gamma_1 >0$, there exists a $m \in \mathbb N$ and $\gamma >0$ such that if $\xi \approx^{R,m, \gamma} X$, then $(D^sF(\xi)^*D^sF(\xi), \Lambda_2) \approx^{R_1, m_1, \gamma_1} (D^sF(X)^* D^sF(X), \Lambda_3)$.  From the above algebraic identifications
\begin{eqnarray*}
(D^sF(\xi)^*D^sF(\xi), \Lambda_2) \approx (DF(\xi)^*DF(\xi), \Lambda_1).  
\end{eqnarray*}
Putting these two facts together yields:

\begin{proposition} Suppose $X$ is an $n$-tuple of elements in $M$ and $F$ is a $p$-tuple of elements in  $\mathfrak{A}_n$.  For any $R>0$ there exists an $R_1>0$ dependent only on $F,X,R$ such that if $m_1 \in \mathbb N$ and $\gamma_1 >0$, then there exist $m \in \mathbb N$ and $\gamma >0$ such that if $\xi \approx^{R,m,\gamma} X$, then $DF(\xi)^*DF(\xi) \approx^{R_1, m_1,\gamma_1} D^sF(X)^*D^sF(X)$.
\end{proposition}  

Combining Lemma 2.9 with the proposition above, one has the following:

\begin{proposition} Suppose $X$ is an $n$-tuple of elements in $M$ and $F$ is a $p$-tuple of elements in  $\mathfrak{A}_n$.  If $\alpha,r, R >0$, then there exist $m \in \mathbb N$ and $\gamma >0$ such that if $\xi \in \Gamma_R(X;m,k,\gamma)$, then the real dimension of the range of the projection $1_{[0,\alpha]}(DF(\xi)^*DF(\xi))$ on $(M_k(\mathbb C))^n$ is no greater than

\begin{eqnarray*} 2nk^2 \cdot \psi(1_{[0,\alpha]}(D^sF(X)^*D^sF(X)) + r)
\end{eqnarray*}

\noindent where $\psi = (tr_n \otimes (tr_2 \otimes( \varphi \otimes  \varphi^{op})))$ is the canonical trace on $M_n(M_2(M \otimes M^{op}))$.  In particular, if $s, R >0$, then there exist $m \in \mathbb N$ and $\rho, \gamma >0$ such that if $\xi \in \Gamma_R(X;m,k,\gamma)$ and $0 < t < \rho$, then the real dimension of the range of the projection $1_{[0,t]}(|DF(\xi)|)$ on $(M_k(\mathbb C))^n$ is no greater than

\begin{eqnarray*} (\text{Nullity}(D^sF(X)) + s) k^2.
\end{eqnarray*}
\end{proposition}  

\begin{proof}  By Proposition 3.12 there exists an $R_1$ dependent only on $F, X, R$ such that for any $m_0 \in \mathbb N$ and $\gamma_0 >0$ there exist $m_1 \in \mathbb N$, $\gamma_1 >0$ such that if $\xi \approx^{R,m_1,\gamma_1} X$, then $DF(\xi)^* DF(\xi) \approx^{R_1, m_0,\gamma_0} D^sF(X)^* D^sF(X)$. 
Applying Lemma 2.9, there exists $m_2 \in \mathbb N$, and $\gamma_2 >0$ such that if $Y$ is a positive semidefinite, symmetric, real linear operator on a finite dimensional real vector space and $Y \approx^{R_1,m_2,\gamma_2} D^sF(X)^*D^sF(X)$, then the trace of the spectral projection $1_{[0, \alpha]}(Y)$ is no greater than
\begin{eqnarray*}
\psi(1_{[0,\alpha]}(D^sF(X)^*D^sF(X))) + r.
\end{eqnarray*}
Setting $m_0=m_2$ and $\gamma_0=\gamma_2$ into the first sentence, there exist $m \in \mathbb N$ and $\gamma >0$ such that if  $\xi \approx^{R,m,\gamma} X$, then $DF(\xi)^* DF(\xi) \approx^{R_1, m_2,\gamma_2} D^sF(X)^* D^sF(X)$.  Thus, if $\xi \in \Gamma_R(X;m,k,\gamma)$, then 
\begin{eqnarray*}
\psi_k(1_{[0,\alpha]}(DF(\xi)^*DF(\xi))) \leq \psi(1_{[0,\alpha]}(D^sF(X)^*D^sF(X))) + r
\end{eqnarray*}
where $\psi_k$ is the normalized real trace on the space of real linear operators on the real vector space $(M_k(\mathbb C))^n$.  Multiplying both sides by $2nk^2$, it follows that the real dimension of the projection $1_{[0,\alpha]}(DF(\xi)^*DF(\xi))$ on $(M_k(\mathbb C))^n$ is no greater than
\begin{eqnarray*}
&& 2nk^2 \cdot (\psi(1_{[0,\alpha]}(D^sF(X)^*D^sF(X))) + r).\\
\end{eqnarray*}
The second claim readily follows from this.
\end{proof}

\subsection{Self-adjoint and Unitary Calculus}

It's natural to wonder how the notions of nullity and rank behave when one deals with self-adjoint or unitary tuples and the mapping $F$ preserves these classes.  In these situations, either class satisfies additional single 
variable relations, i.e., $X-X^*=0$ or $X^*X=I$, respectively.  These relations should 'transform' the $2 \times 2$ matrix entries of $D^sF(X)$ into one operator entry through a change of variables.  Moreover, the nullity or rank of the full $2\times 2$ case and the single operator entry case should be connected in a natural way, and the resultant differential calculus on the self-adjoints or unitaries should be simple to compute (or at least one that is no more difficult than that of $D^s$).  This is indeed the case and the results are stated and proven here.  

While such observations could be made using the notion of a Hilbert manifold, I avoid them here and will proceed in a more pedestrian way.

First for the self-adjoint situation.  In this case the differential calculus is consistent with that of \cite{v4} and is connected to the $*$-calculus in the obvious way.

\begin{definition} Fix an $n$-tuple $X=\{x_1,\ldots, x_n\}$ in the tracial von Neumann algebra $M$ and suppose $f \in \mathfrak{A}_n$ and $1 \leq j \leq n$.   Set $x =X_j$, $A=A_j$ as in Definition 3.7.  $\partial_{sa}: A[x, x^*] \rightarrow A[x,x^*] \odot A[x,x^*]^{op}$ is the derivation determined by the relations $\partial_{sa}(A)=0$ and $\partial_{sa}(x) = \partial_{sa}(x^*) = I \odot I^{op}$ where $A[x,x^*] \odot A[x,x^*]^{op}$ is given the natural $A[x,x^*]-A[x,x^*]^{op}$ bimodule structure.  The jth partial $S^{sa}$-derivative of $f$ at $X$, $\partial_i^{sa}f(X)$, is the element $(\pi_X \otimes \pi_X^{op}) \circ \partial_{sa}f \in M \otimes M^{op}$ where $\pi_X:\mathfrak{A}_n \rightarrow M$ is the unique $*$-homomorphism such that $\pi_X(X_j) = x_j$.
\end{definition}

\begin{definition} Suppose $X \subset M$ and $F = \{f_1,\ldots, f_m\} \subset \mathfrak{A}_n$ are $n$ and $m$-tuples and $F$ consists of self-adjoint elements.   For each $1\leq i \leq m$, $1 \leq j \leq n$, consider the $\partial^{sa}_jf_i(X)$, the $j$th partial $S^{sa}$-derivative of $f_i$ at $X$ realized in the von Neumann algebra tensor product $M \otimes M^{op}$.  The self-adjoint S-derivative of $F$ at $X$ is
\begin{eqnarray*}
D^{sa}F(X) =
\begin{bmatrix}
(\partial_{1}^{sa} f_1)(X) & \cdots & (\partial_{n}^{sa} f_1)(X) \\ \vdots &  & \vdots \\
(\partial_{1}^{sa}f_m)(X) & \cdots  & (\partial_{n}^{sa} f_m)(X) \\
\end{bmatrix} \in M_{m\times n}(M  \otimes M^{op}).
\end{eqnarray*}
\end{definition}

\begin{remark}  Suppose $X, F$ are as in Definition 3.15.  One can consider the nullity and rank of $D^{sa}F(X) \in M_{m\times n}(M  \otimes M^{op})$ as described in Section 2.6.  Explicitly set $P = 1_{\{0\}}(D^{sa}F(X)^* D^{sa}F(X))$ and $Q = 1_{(0,\infty)}(D^{sa}F(X)^* D^{sa}F(X))$.   $P, Q \in M_{n}(M \otimes M^{op})$.  Consider the tracial state $\psi = tr_{n} \otimes (\varphi \otimes \varphi^{op})$ on $M_{n}(M \otimes M^{op})$.  The selfadjoint nullity and rank of $F$ at $X$ are

\[  \text{Nullity}(D^{sa}F(X)) = n \cdot \psi(P) \in [0,n]
\]

\noindent and

\[  \text{Rank}(D^{sa}F(X)) = n \cdot \psi(Q) \in [0,n],
\]
respectively.
\end{remark}

\begin{proposition} Suppose $X \subset M$ and $F \subset \mathfrak{A}_n$ are $n$ and $m$-tuples of elements and $L = \{(X_1-X_1^*)/2,\ldots, (X_n-X_n^*)/2\} \subset \mathfrak{A}_n$.  If $G = F \cup L$ denotes the joined $(m+n)$-tuple of elements of $\mathfrak{A}_n$ and the elements of $F$ are self-adjoint, then $\text{Nullity}(D^sG(X)) = \text{Nullity}(D^{sa}F(X))$.  Moreover, if $\mu$ is the spectral distribution of $|D^sG(X)|$ and $\nu$ is the spectral distribution of $|D^{sa}F(X)|$, then there exists a $c>0$ such that for any $t \in (0,1)$, $\mu((0,t]) \leq \nu((0,ct])$. 
\end{proposition}

\begin{proof}  By Lemma 2.11 it will suffice to prove the dimension and spectral decay claims for a matrix with entries in $M\otimes M^{op}$ which is $(M \otimes M^{op})$-row equivalent to $D^sG(X)$.  $D^sG(X) = D^s(F \cup L)$ has the matricial form
\begin{eqnarray*}
\begin{bmatrix} 
D^sF(X) \\
D^sL(X) \\ 
\end{bmatrix} & \in & M_{(m+n) \times n}(M_2(M \otimes M^{op})) \\
\end{eqnarray*}
Expanding the terms in both $D^sF(X)$ and $D^sL(X)$, and writing $F= \{f_1,\ldots, f_m\}$ where $f_i = f_i^*$ by hypothesis,
\begin{eqnarray*}  D^sG(X) & = & \begin{bmatrix}
A_{11} & \cdots & \cdots &  A_{1n} \\ 
\vdots  & \cdots & \cdots & \vdots  \\
A_{m1} & \cdots & \cdots & A_{mn} \\
E & 0 & \cdots & 0 \\
0 & E & \ddots  & 0 \\
\vdots & \ddots & \ddots & 0 \\
0 & \cdots & 0 & E \\
\end{bmatrix} \\
\end{eqnarray*}
where $e = \begin{bmatrix} 
0 & 0 \\
0 & 1 \\ 
\end{bmatrix}$, $E = I_{M \otimes M^{op}} \otimes e = \begin{bmatrix} 
0 & 0 \\
0 & 1 \\ 
\end{bmatrix}$, and 
\begin{eqnarray*}
A_{ij} = \partial^s_jf_i(X) = \begin{bmatrix}
\partial_i^{sa} f_j(X) & \partial_i^{sk} f_j(X) \\ 
0 & 0 \\ 
\end{bmatrix} \\
\end{eqnarray*}
are elements of $M_2(M \otimes M^{op})$.  As elsewhere, I make the obvious identication of operator-valued matrices and their representation as elements in $M \otimes M^{op}$ tensored by the matrix algebras.  Inspecting $E$ and $A_{ij}$ it follows that the matricial form of $D^sG(X)$, regarded as a $2(n+m) \times 2n$ matrix with entries in $M \otimes M^{op}$, is $(M \otimes M^{op})$-row equivalent to
\begin{eqnarray*} 
T& = & D^{sa}F(X) \otimes e + I_{M_n(M \otimes M^{op})} \otimes e^{\bot}  \\ 
  & \in & M_{n}(M \otimes M^{op}) \otimes M_2(\mathbb C) \\ 
  & = &M_{2n}(M \otimes M^{op}).\\
\end{eqnarray*}
Note that this row equivalence is made in the sense of section 2.6 where the $T$ is identified with the matrix $T_0 \in M_{2(m+n)\times 2n}$ obtained by taking $T$ and filling the last $2m$ rows with $0$.
It follows that $|T|^2 = |D^{sa}F(X)|^2 \otimes e + I_{M_n(M \otimes M^{op})} \otimes e^{\bot}$.  The two terms of this sum are positive elements with their respective ranges contained in the ranges of the orthogonal projections, $I_{M_n(M \otimes M^{op})} \otimes e$ and $I_{M_n(M \otimes M^{op})} \otimes e^{\bot}$.  It follows that $|T| = |D^{sa}F(X)| \otimes e + I_{M_n(M \otimes M^{op})} \otimes e^{\bot}$.  If $\mu$ and $\nu$ are the spectral trace measures associated to $|T|$ and $|D^{sa}F(X)|$, then the orthogonality comment applied to the decomposition of $|T|$ implies that $\mu = (\nu + \delta_{\{1\}})/2$.  It follows from this that $\text{Nullity}(T) = \text{Nullity}(D^{sa}F(X))$ and for all $t \in (0,1)$, $\mu((0,t]) = \nu((0,t])/2 < \nu((0,t]$.  Since $T$ is $(M \otimes M^{op})$-row equivalent to $D^sG(X)$ this completes the proof.
\end{proof}
 
Turning towards the unitary situation, suppose now that $X$ consists of unitary elements.  I will argue as in the self-adjoint case, except that this time instead of using the linear equation $X= X+X^*$, I will use the unitary relation $X^*X=I$ to conclude that the kernel is contained in the skew-adjoints (the tangent space of the unitaries).   I will then find a related linear operator whose kernel and determinant is equivalent to the kernel and determinant of the differential restricted to this tangent space.  As in the self-adjoint case, the end result will express the rank as a matrix over $M \otimes M^{op}$ as opposed to $M_2(M \otimes M^{op})$.  First, a preliminary computation to motivate the definition.

\begin{definition} For each $f \in \mathfrak{A}_n$, define 
\begin{eqnarray*}
m_f = \frac{1}{2} \cdot \begin{bmatrix} 
f \odot I + I \odot f^* & f \odot I - I \odot f^* \\
f \odot I - I \odot f^* & f \odot I + I \odot f^* \\ 
\end{bmatrix}  \in M_2(\mathfrak{A}_n \odot \mathfrak{A}_n^{op}) \\
\end{eqnarray*}
Given an $n$-tuple $X$ of elements in the tracial von Neumann algebra, $m_{f(X)} = (\pi_X \odot \pi_X \otimes I_2)(m_f) \in M_2(M \odot M^{op}) \subset M_2(M \otimes M^{op})$.
\end{definition}

\begin{remark} If $X =\{a\}$ and $f \in \mathfrak{A}_1$ is the trivial element $f=X_1$, then $m_{f(X)}=m_a$ is the $2 \times 2$ tensor matrix representation of the left multiplication operator on $L^2(M)$ by $a$, i.e., $m_a = \Phi(\pi(a))$ where $\pi$ is the left regular representation of $M$ on $L^2(M)$.  See Example 3.1.
\end{remark}

\begin{lemma} If  $f \in \mathfrak{A}_n$ is a noncommutative $*$-monomial and $Y =\{y_1, \ldots, y_n\} \subset M$ is an $n$-tuple of unitaries in $M$, then the following hold:
\begin{itemize}
\item For any $1\leq j \leq n$, $(m_{f(Y)^*} ((\partial_j^sf)(Y)) m_{y_j}) \in M_2(M \otimes M^{op})$ is a diagonal matrix.
\item If $f = Z_1 X_j^{k_1} Z_2 X_j^{k_2} \cdots Z_p X_j^{k_p} Z_{p+1}$ where $Z_i \in A_j$ are $*$-monomials in $\mathbb X - \{X_j\}$ and $k_i \in \{1, *\}$, $1 \leq i \leq p+1$, and $\pi_Y: \mathfrak{A}_n \rightarrow M$ is the canonical $*$-homomorphism such that $\pi_Y(X_i) = y_i$, then $[m_{f(Y)^*}((\partial_j^sf)(Y)) m_{y_j}]_{22}$ is
\begin{eqnarray*}
\sum_{i=1}^p (-1)^{\delta_{k_i, *}} \pi_Y \left (Z_{p+1}^* \cdots (X_j^{k_{i+1}})^* Z_{i+1}^* X_j^{\delta_{k_i,*}} \right) \otimes \left(\pi_Y\left((X_j^{\delta_{k_i,*}})^* Z_{i+1} X_j^{k_{i+1}} \cdots Z_{p+1} \right)\right)^{op} \in M \otimes M^{op}.
\end{eqnarray*}
\end{itemize}
\end{lemma}

\begin{proof}  First note that it suffices to do this when $n >1$ since $\mathfrak{A}_n$ canonically embeds into $\mathfrak{A}_{n+1}$ and $Y$ can be augmented into the $(n+1)$-tuple $\{y_1,\ldots, y_n, I\}$.

Denote by $B_n$ the algebra obtained by quotienting $\mathfrak{A}_n$ by the ideal generated by the unitary relations $X_jX_j^* = X_j^*X_j = I$ and by $q: \mathfrak{A}_n \rightarrow B_n$ the quotient map.  Equivalently, $B_n$ is the universal unital complex $*$-algebra generated by $n$ unitaries.  Recall that $\pi_Y: \mathfrak{A}_n \rightarrow M$ is the complex $*$-homomorphism uniquely defined by $\pi_Y(X_j) = y_j$.  Because the $y_j$'s are unitary, there exists a canonical complex $*$-homomorphism $\sigma_Y: B_n \rightarrow M$ uniquely defined by $\sigma_Y(q(X_n)) = y_j$ such that the following diagram commutes:
\begin{eqnarray*}
\begin{tikzcd}[column sep=5em, row sep=3em]
\mathfrak{A}_n \arrow{r}{q}  \arrow[swap]{rd}{\pi_Y} 
  & B_n \arrow{d}{\sigma_Y} \\
    & M
\end{tikzcd}
\end{eqnarray*}
Each of these maps yields an opposite morphism on the opposite domain and range; as a map of sets, these opposite morphisms agree with the original morphism.  Abusing notation I'll use the same letter to denote the map and its induced opposite map.  Tensoring the maps with their opposite maps and amplifying by the $2 \times 2$ matrices yields the following commutative diagram:
\begin{eqnarray*}
\begin{tikzcd}[column sep=5em, row sep=6em]
M_2(\mathfrak{A}_n \odot \mathfrak{A}_n^{op}) \arrow{r}{q \odot q \odot I_2}  \arrow{rd}[swap]{(\pi_Y \odot \pi_Y) \odot I_2} 
  & M_2(B_n \odot B_n^{op})\arrow{d}{(\sigma_Y \odot \sigma_Y) \odot I_2} \\
    & M_2(M \otimes M^{op})
\end{tikzcd}
\end{eqnarray*}
Thus,
\begin{eqnarray*}
m_{f^*(Y)} (\partial^sf)(Y)) m_{y_j} & = & (\pi_Y \odot \pi_Y \odot I_2)(m_f \cdot (D_{(x,x^*)}f) \cdot m_{X_j}) \\
                                                       & = & (\sigma_Y \odot \sigma_Y \odot I_2)(q \odot q \otimes I_2)(m_{f^*} (D_{(x,x^*)}f) m_{X_j}). \\
\end{eqnarray*}
From the above, the claim reduces to looking at $(q \odot q \odot I_2)(m_{f^*} (D_{(x,x^*)}f) m_{X_j})$.   Set $a_1 = (q \odot q \odot I_2)(m_{f^*})$, $a_2 = (q \odot q \odot I_2)(D_{(x,x^*)}f)$, $a_3 = (q \odot q \odot I_2)(m_{X_j})$.  

Recall that there is an injective complex linear homomorphism $\pi: B_n \odot B_n^{op} \rightarrow L_{\mathbb C}(B_n)$ where $\pi(a \odot b): B_n \rightarrow B_n$ is defined by $\pi(a \odot b)(x) = axb$.  Recall also the real linear, multiplicative map $\Phi_{B_n}: L_{\mathbb R}(B_n) \rightarrow M_2(L_{\mathbb C}(B_n))$ given in Subsection 3.1 (just after Lemma 3.4).  $f = Z_1 X_j^{k_1} Z_2 X_j^{k_2} \cdots Z_p X_j^{k_p} Z_{p+1}$ where $Z_i \in A_j$ are $*$-monomials in $\mathbb X - \{X_j\}$ and $k_i \in \{1, *\}$, $1 \leq i \leq p$.  Write $X = q(\mathbb X)$, $x_i = q(X_i)$, and $z_j = q(Z_j)$.  By definition of $B_n$ the $x_i$ and $z_i$ are unitary elements of $B_n$.  One easily checks the identities
\[ (\pi \odot I_2)(a_1) = \Phi_{B_n}(\pi(f^*(X) \odot I)),\]

\[ (\pi \odot I_2)(a_2) =  \Phi_{B_n}(T),\]

\[ (\pi \odot I_2)(a_3) = \Phi_{B_n}(\pi(x_j \odot I)),\]

\noindent where $T \in L_{\mathbb R}(B_n)$ is the real linear operator defined by 
\begin{eqnarray*}
T = \sum_{i=1}^p \pi(z_1 x_j^{k_1} \cdots z_i \odot z_{i+1} x_j^{k_{i+1}} \cdots z_{p+1}) J^{\delta_{k_i, *}}.
\end{eqnarray*}
Using the fact that all terms in $B_n$ appearing in the elementary tensors below are unitary, 
\begin{eqnarray*} & & \pi(f(X)^* \odot I) T \pi(x_j \odot I) \\ & = & \pi(f(X)^* \odot I) \sum_{i=1}^p \pi(z_1 x_j^{k_1} \cdots z_i \odot z_{i+1} x_j^{k_{i+1}} \cdots z_{p+1}) J^{\delta_{k_i, *}} \pi(x_j \odot I) \\  & = & \pi(f(X)^* \odot I) \sum_{i=1}^p \pi(z_1 x_j^{k_1} \cdots z_i (x_j)^{\delta_{k_i,1}} \odot (x_j^*)^{\delta_{k_i,*}} z_{i+1} x_j^{k_{i+1}} \cdots z_{p+1}) J^{\delta_{k_i, *}}  \\ & = & \sum_{i=1}^p \pi(z_{p+1}^* \cdots (x_j^{k_{i+1}})^* z_{i+1}^* (x_j^{k_i})^* (x_j)^{\delta_{k_i,1}} \odot (x_j^*)^{\delta_{k_i,*}} z_{i+1} \cdots z_{p+1}) J^{\delta_{k_i, *}} \\ \\ & = & \sum_{i=1}^p \left [ \pi \left (z_{p+1}^* \cdots (x_j^{k_{i+1}})^* z_{i+1}^* x_j^{\delta_{k_i,*}} \odot (x_j^{\delta_{k_i,*}})^* z_{i+1} x_j^{k_{i+1}} \cdots z_{p+1} \right ) \right ]  J^{\delta_{k_i, *}}.
\end{eqnarray*}

\noindent Both $B_n^{sa}$ and $B_n^{sk}$ (the real subspaces of self-adjoint and skew-adjoint elements) are invariant real subspaces of $J^{\delta_{k_i, *}}$ and so they are invariant under the above operator as well.    
From the definition of $\Phi_{B_n}$ and the fact that $B_n^{sa}$ and $B_n^{sk}$ are invariant under $(\pi(f(X)^* \otimes I) T \pi(x_j \otimes I)$ it follows that $\Phi_{B_n}(\pi(f(X)^* \odot I) T \pi(x_j \odot I)) \in M_2(B_n \odot B_n^{op})$ is diagonal.  Moreover, note that if $\xi \in B_n^{sk}$, then from the above,
\begin{eqnarray*}
& & \pi(f(X)^* \odot I) T \pi(x_j \odot I)(\xi) \\ & = & \sum_{i=1}^p \left [ \pi \left (z_{p+1}^* \cdots (x_j^{k_{i+1}})^* z_{i+1}^* x_j^{\delta_{k_i,*}} \odot (x_j^{\delta_{k_i,*}})^* z_{i+1} x_j^{k_{i+1}} \cdots z_{p+1} \right ) \right ]  J^{\delta_{k_i, *}}(\xi) \\ & = & \sum_{i=1}^p (-1)^{\delta_{k_i, *}} \left (z_{p+1}^* \cdots (x_j^{k_{i+1}})^* z_{i+1}^* x_j^{\delta_{k_i,*}} \right) (\xi) \left((x_j^{\delta_{k_i,*}})^* z_{i+1} x_j^{k_{i+1}} \cdots z_{p+1} \right).  \\
\end{eqnarray*}

To finish the proof, set $b_1 = \pi(f^*(X)\odot I)$, $b_2 = T$, $b_3 = \pi(x_j \odot I)$.  $a_i \in M_2(B_n \odot B_n^{op})$ and $b_i \in L_{\mathbb R}(B_n)$ and $\Phi_{B_n}(b_i) = (\pi \otimes I_2)(a_i)$ by the three identities stated above.
\begin{eqnarray*}
(\pi \otimes I_2)(a_1 a_2 a_3) & = & (\Phi_{B_n}(b_1 b_2 b_2)) \\ & = & \Phi_{B_n}\left (\pi(f(X)^* \odot I) T \pi(x_j \odot I)\right) \\ & \in & M_2(B_n\odot B_n^{op})\\
\end{eqnarray*}
is a diagonal matrix by the preceding paragraph.  Since $n >1$, Lemma 3.5 implies that $\pi \otimes I_2$ is injective, whence $a_1a_2a_3$ is diagonal.  Similarly, from the representations of $[\Phi_{B_n}\left (\pi(f(X)^* \odot I) T \pi(x_j \odot I)\right)]_{22}$ in the preceding paragraph and the injectivity of $\pi \otimes I_2$, it follows that $[a_1 a_2 a_3]_{22}$ is
\begin{eqnarray*}
\sum_{i=1}^p (-1)^{\delta_{k_i, *}} \left (z_{p+1}^* \cdots (x_j^{k_{i+1}})^* z_{i+1}^* x_j^{\delta_{k_i,*}} \right) \odot \left((x_j^{\delta_{k_i,*}})^* z_{i+1} x_j^{k_{i+1}} \cdots z_{p+1} \right)^{op} \in B_n \odot B_n^{op}.
\end{eqnarray*}

Finally, from the first paragraph, $m_{f^*(Y)}(\partial^sf(Y))m_{y_j} = (\sigma_Y \odot \sigma_Y \odot I_2)(a_1a_2a_3)$.  The established diagonality of $a_1a_2a_3$ and the form of its $22$-entry with the commutative diagrams of the first paragraph complete the proof.
\end{proof}

\begin{remark} Lemma 3.20 is a computation of the product of three matrices $m_{f(Y)^*}$, $(\partial_j^sf)(Y)$, $m_{y_j}$,  It is a check that matrix multiplication and operator composition are the same thing.  An alternative to the approach would be to stay in the matrix setting and multiply the matrices.  This is surprisingly messy.
\end{remark}

\begin{definition}  Suppose $X \subset M$ is an $n$-tuple of unitary elements and $F=\{f_1,\ldots, f_m\}$ is an $m$-tuple of $*$-monomials in $\mathfrak{A}_n$.  Define $D^uF(X)$ to be the element in $M_{m \times n}(M\otimes M^{op})$ whose $ij$th entry is 
\begin{eqnarray*} 
\partial_j^uf_i(X) & = & \left [ m_{f_i(X)^*} ((\partial^s_j f_i)(X)) m_{x_j} \right ]_{22}.\\
\end{eqnarray*} 
\end{definition}

\begin{remark}  Suppose $X, F$ are as in Definition 3.22.  As in the selfadjoint case, one can consider the nullity and rank of $D^uF(X) \in M_{m\times n}(M  \otimes M^{op})$ as described in Section 2.6 .  Note that in this case, as in the selfadjoint case, $\text{Nullity}(D^uF(X)), \text{Rank}(D^uF(X)) \in [0, n]$.  
\end{remark}

\begin{remark} The entries of $D^uF(X)$ can be described/computed in the following way.  As in Lemma 3.20 suppose $f_l = Z_1 X_j^{k_1} Z_2 X_j^{k_2} \cdots Z_p X_j^{k_p} Z_{p+1}$ is a reduced word where $Z_i \in A_j$ and $k_i \in \{1, *\}$, $1 \leq i \leq p$.  If $1 \leq i \leq n$, then the $ij$th entry of $D^uF(X)$ is obtained by taking a sum over all occurrences of $X_j^{m_j}$ in $f_i$ with $m_j \in \{1, *\}$, $f_i = w_1 X_j^{m_j} w_2$, where each occurrence contributes a summand terms of $w_2^* \otimes w_2$ when $m_j=1$ or $-w_2^*X_j \otimes X_j^* w_2$ when $m_j=*$.  Then this sum in $\mathfrak{A}_n \odot \mathfrak{A}_n^{op}$ is evaluated (by universality) at $X \subset M$ to produce an element in $M \otimes M^{op}$.
\end{remark}

\begin{remark} Suppose that $X = \{u_1,\ldots, u_n\}$ is an $n$-tuple of unitaries in $\mathcal A = \oplus_{i=1}^n M_N(\mathbb C)$, and $F$ is a $p$-tuple of noncommutative $*$-monomials.  Set $\mathcal B = \oplus_{j=1}^p M_N(\mathbb C)$.  If $U_N$ denotes the $N \times N$ unitary matrices, then the unitary groups $U(\mathcal A)$ of $\mathcal A$ and $U(\mathcal B)$ of $\mathcal B$ are $\oplus_{i=1}^N U_N$ and $\oplus_{j=1}^p U_N$.  Denote by $M_N^{sk}$ the $N\times N$ skew-adjoint matrices and set $\mathcal A^{sk} = \oplus_{i=1}^n M_N^{sk}$ and $\mathcal B^{sk} = \oplus_{i=1}^n M_N^{sk}$.  The tangent space of $U(\mathcal A)$ at $X$ is $X\mathcal A^{sk}$ and the tangent space of $U(\mathcal B)$ at $f(X)$ is $f(X)\mathcal B^{sk}$.  $F$ restrict to a well-defined map from $U(\mathcal A)$ into $U(\mathcal B)$ and thus, its differential at $X$, $DF(X): \mathcal A \rightarrow \mathcal B$ sends the tangent space $X\mathcal A^{sk}$ into $f(X)\mathcal B^{sk}$.  The rank of this map is the (finite real) dimension of $DF(X)X\mathcal A^{sk} \subset F(X) \mathcal B^{sk}$.   However, since $F(X) \in U(\mathcal B)$, this is the same as the (finite real) dimension of $F(X)^* DF(X) X \mathcal A^{sk} \subset \mathcal B^{sk}$.  This is what $D^uF(X)$ captures algebraically:

\begin{eqnarray*} \dim_{\mathbb R}(F(X)^* DF(X) X \mathcal A^{sk}) & = & N^2 \cdot \text{Rank}(D^uF(X)).
\end{eqnarray*}
\end{remark}

\begin{remark} A few simple computations on the polynomial $X_i^*X_i-I$ will be useful in what follows.  To ease the notation, for $x \in M$, I'll simply write $x \in M^{op}$ for the image of $x$ in $M^{op}$, as opposed to $x^{op}$.  This should cause no confusion here since the example consists of computations involving a single unitary.  Fix $1 \leq i \leq n$ and consider $g=X_i^*X_i-I \in \mathfrak{A}_n$.  Again, $X=\{x_1,\ldots, x_n\}$ is an $n$-tuple of unitaries in a tracial von Neumann algebra $M$.  Clearly $\partial^s_jg(X)=0$ for $j \neq i$ and
\begin{eqnarray*}
\partial_i^sg(X) =  \begin{bmatrix} 
x_i^* \otimes I + I \otimes x_i & x_i^* \otimes I - I \otimes x_i \\
0 & 0 \\ 
\end{bmatrix}  \in M_2(M \otimes M^{op}).
\end{eqnarray*}
Define 
\begin{eqnarray*}
p_i = \frac{1}{2} \cdot \begin{bmatrix} 
I -\frac{1}{2}( x_i \otimes x_i + x_i^* \otimes x_i^*) &  \frac{1}{2}(x_i \otimes x_i -  x^*_i \otimes x^*_i) \\
\frac{1}{2}(x^*_i \otimes x^*_i - x_i \otimes x_i) & I + \frac{1}{2}(x_i^* \otimes x_i^* + x_i \otimes x_i) \\ 
\end{bmatrix}  \in M_2(M \otimes M^{op})
\end{eqnarray*}
and $e_i = p_i^{\bot}$.  Note that $p_i$ is the matricial, tensor product representation of the real operator $[\sigma_M(I \otimes I) - \sigma_M(x_i \otimes x_i) \circ  J]/2$ and $e_i$ is the matricial, tensor product representation of $[\sigma_M(I \otimes I) + \sigma_M(x_i \otimes x_i) \circ J]/2$.   It is easy to check that $(\partial_i^sg(X)) p_i=0$, $(\partial_i^sg_i(X)) e_i = (\partial_i^sg(X))$, the projection onto the range of $\partial_i^sg(X)$ has (unnormalized) trace $1$, and that $p_i$ is the projection onto $\ker(\partial_i^sg)$.  

Notice also that $p_i = m_{x_i}z_i$ where 
\begin{eqnarray*}
z_i = \frac{1}{2} \cdot \begin{bmatrix} 
0 & 0 \\
x^*_i \otimes I - I \otimes x_i & x_i^* \otimes I + I \otimes x_i \\ 
\end{bmatrix}  \in M_2(M \otimes M^{op}).\\
\end{eqnarray*}
Here, $z_i$ is the tensor matricial representation of $(\sigma_M(x_i^* \otimes I) - \sigma_M(I \otimes x_i)J)/2$.  One also has
\begin{eqnarray*} \partial^s_ig(X) m_{x_i} & = & \frac{1}{2} \begin{bmatrix} 
x^*_i \otimes I + I \otimes x_i & x^*_i \otimes I - I \otimes x_i \\
0 & 0 \\ 
\end{bmatrix}  \begin{bmatrix} 
x_i \otimes I + I \otimes x_i^* & x_i \otimes I - I \otimes x_i^* \\
x_i \otimes I - I \otimes x^* & x_i \otimes I + I \otimes x_i^* \\ 
\end{bmatrix}\\
                                                                                       & = & 2 \cdot \begin{bmatrix} 
I_{M \overline{\otimes} M^{op}} & 0 \\
0 & 0 \\ 
\end{bmatrix}.\end{eqnarray*}
\end{remark}
  
\begin{proposition}
Suppose $X =\{x_1,\ldots, x_n\} \subset M$ and $F =\{f_1,\ldots,f_m\} \subset \mathfrak{A}_n$ are $n$ and $m$-tuples and $G =\{X_1^*X_1-I, \ldots, X_n^*X_n-I\} \subset \mathfrak{A}_n$.  If every element of $F$ is a noncommutative $*$-monomial, every element of $X$ is a unitary, and $H=\langle f_i -I \rangle_{i=1}^m \cup G$ is the joined $(m+n)$-tuple of elements in $\mathfrak{A}_n$, then $\text{Nullity}(D^sH(X)) = \text{Nullity}(D^uF(X))$.  Moreover, if $\mu$ and $\nu$ are the spectral distributions of $|D^sH(X)|$ and $|D^uF(X)|$, then for any $t \in (0,1)$, $\mu((0,t]) \leq \nu((0,t])$. 
\end{proposition}

\begin{proof}  Writing out the definition,
\begin{eqnarray*}  D^sH(X) & = & \begin{bmatrix} 
D^sF(X) \\
D^sG(X)  \\
\end{bmatrix} \in M_{(m+n) \times n}(M_2(M\otimes M^{op})).
\end{eqnarray*} 
\noindent Denote by $A$ the $n \times n$ diagonal matrix whose $ii$th entry is $m_{x_i}$ and set $E_1 = I_{M_n(M \otimes M^{op})} \otimes e_{11} \in M_{n\times n}(M_2(M \overline{\otimes} M^{op}))$.  Note that $A$ is a unitary element since the $x_i$ are unitaries (Remark 3.19).  Remark 3.26 shows that 
\begin{eqnarray*} 
D^sG(X) A & = & 2E_1 \in M_{n\times n}(M_2(M \otimes M^{op})).\\
\end{eqnarray*}
Set $E_2 = E_1^{\bot} = I_{M_n(M \otimes M^{op})} \otimes e_{22}$.  Denote by $W$ the $m\times m$ diagonal matrix whose $ii$th element is  $m_{f_i(X)^*}$.  Note that $W$, like $A$, is unitary since the elements of $F$ are $*$-monomials and the elements of $X$ are unitaries.  By Lemma 3.20, for any $1 \leq i \leq m$ and $1 \leq j \leq n$, $(WD^sF(X)A)_{ij} \in M_2(M \otimes M^{op})$ is a diagonal element whose $22$-entry is $D^uF(X)_{ij}$.  Consequently, $|WD^sF(X)A|^2$ commutes with $E_2=I_{M_n(M\otimes M^{op})} \otimes e_{22}$, as well as $E_1 = E_2^{\bot}$.

Compute:
\begin{eqnarray*} 
                                A^* D^sH(X)^*D^sH(X) A & = & A^* D^sF(X)^*D^sF(X) A + A^* D^sG(X)^*D^sG(X) A \\
                                                                  &= & (A^*D^sF(X)^*W^*)(WD^sF(X)A)(E_2 +E_1) + 2E_1 \\
                                                                  & = & |D^uF(X)|^2 \otimes e_{22}  + \\ & & E_1A^*D^sF(X)^*W^* WD^sF(X)AE_1 + 2E_1 \\
                                                                  & \geq & |D^uF(X)|^2 \otimes e_{22}  + E_1. \\
\end{eqnarray*}
From the above $\ker(|D^sH(X)A|) = \ker(|D^uF(X)| \otimes e_{22})$.  As $A$ is a unitary, $\text{Nullity}(|D^sH(X)|) = \text{Nullity}(|D^sH(X)A|) = \text{Nullity}(|D^uF(X)|)$, establishing the nullity claim.  

For the spectral distribution claim, set $T = D^uF(X) \otimes e_{22}$.  Taking square roots of the above computation, $|D^sH(X)A| \geq |T| + E_1$.  Note that $|T|=|D^uF(X)| \otimes e_{22}$ and $E_1$ are positive elements with orthogonal supports, $E_1 +E_2 =I$, the (normalized) trace of $E_1$ is $1/2$, and $\ker(|T|) = \ker(|D^uF(X)|) \otimes e$.  It follows from this that if $m$ is the spectral distribution of $|T+E_1| = |T| + E_1$, then $m = (\nu + \delta_{1})/2$. 

If $\mu_1$ is the spectral distribution of $|D^sH(X)A|$, then $\mu_1 = \mu$ since $A$ is unitary.  By Weyl's Inequality for positive operators (Section 2.6), for any $t \in [0,1)$, $\mu([0,t]) = \mu_1([0,t]) \leq m([0,t]) = \nu([0,t])/2$.  Setting $t=0$ on the RHS equation yields $m(\{0\}) = \nu(\{0\})/2$.  On the other hand, by the nullity equation just established, the trace of the projection onto $\ker(|D^uF(X)|)$ is twice the trace of the projection onto $\ker(|D^sH(X)|)$, i.e., $\nu(\{0\}) = 2 \mu(\{0\})$.  Thus,  $m(\{0\}) = \mu(\{0\})$.  Combined with the fact that for all $t \in [0,1)$, $\mu([0,t]) \leq m([0,t])$, it follows that for all $t \in (0,1)$, $\mu((0,t]) \leq m((0,t]) = \nu((0,t])$.
\end{proof}

\section{Single Spectral Splits: Local Projection and Dimension Bounds}

The main goal of this section is to find upper bound free entropy and Hausdorff dimension estimates for the solution space $F(X)=Y$.  Here $F$ denotes a finite tuple of noncommutative $*$-polynomials and $X$ is a finite tuple in a tracial von Neumann algebra.  The estimates will be corollaries of a more general estimate which is derived form the standard 'local to global' manifold argument and some specific Euclidean covering estimates.  This section will be broken up into three parts: the Euclidean covering estimates, the localization construction and free entropy dimension implications, and examples.

\subsection{Euclidean Estimates}

For a linear operator $T$ on a vector space define $T^{\bot} = I - T$.  This notation coincides with the usual inner product space one when $T$ is an orthogonal projection.

\begin{lemma} Suppose $K \subset \mathbb R^d$ is an open, convex subset containing the origin $0$ and $f: K \rightarrow \mathbb R^m$ is a $C^1$-function.  If for all $x \in K$, $\|Df(0) - Df(x)\| < r$ and $T$ is a linear operator on $\mathbb R^d$ such that for any $x \in \mathbb R^d$, $\|Df(0)T(x)\| \geq \beta \|T(x)\|$, then for any $x, y \in K$,

\begin{eqnarray*} \|f(y) - f(x)\| & \geq & \beta \cdot \|T(y)-T(x)\| - \|Df(0)\| \cdot \|T^{\bot}(y)-T^{\bot}(x)\| - r\|y-x\|.
\end{eqnarray*}
\end{lemma}

\begin{proof}  Suppose $x,y \in K$.  Set $B = \int_0^1 Df(x + t(y-x)) \, dt$; using the convexity assumption and the bound on the derivative, $\|(B - Df(0)) \| <  r.$  By the mean value theorem 

\begin{eqnarray*} \|f(y) - f(x) \| & = & \| B(y-x) \| \\
                                   & \geq & \|Df(0)(y-x)\| - r\|y-x\| \\
                                   & \geq & \|Df(0)( T+ T^{\bot})(y -x))\| - r\|y-x\|\\
                                   & \geq & \|Df(0)T(y-x)\| - \|Df(0) T^{\bot}(y -x)\| - r\|y-x\|\\
                                   & \geq & \beta \cdot \|T(y) - T(x)\| - \|Df(0)\| \cdot \|T^{\bot}(y)-T^{\bot}(x)\| - r \cdot \|y-x\|. \\        
\end{eqnarray*}
\end{proof}

\begin{lemma} Suppose $K\subset \mathbb R^d$ is an open, convex set, $f:K \rightarrow \mathbb R^m$ is a $C^1$-function, and $x_0 \in E \subset K$.  Assume $Q$ is an orthogonal projection such that for some $\beta \in (0,1)$ and any $x \in \mathbb R^d$, $\|Df(x_0)Q(x)\| \geq \beta \|Q(x)\|$.  Denote by $A$ the affine map which sends $x \in \mathbb R^d$ to $Q^{\bot}(x- x_0)$.  If $t \in (1-\frac{\beta}{8(\|Df(x_0)\|+1)},1)$ and for all $x \in K$, $\|Df(x_0) - Df(x)\| < \frac{\beta}{4}$, then for any $\epsilon >0$,

\begin{eqnarray*} K_{\epsilon}(E) & \leq & K_{(1-t)\epsilon}(A(E)) \cdot S_{\frac{\beta \epsilon}{4}}(f(E)).\\
\end{eqnarray*}
\end{lemma}

\begin{proof} Without loss of generality assume $x_0 =0$ so that $A=Q^{\bot}$.  Fix an $\epsilon$-separated subset $\Delta$ of $E$ of maximal cardinality.  Find a cover for $Q^{\bot}(E)$ by open $(1-t)\epsilon$ balls with minimal cardinality and denote the set of centers of these balls by $\Theta$.  For every $x \in \Theta$ define $F_x = \{y \in \Delta: \|x- Q^{\bot}(y)\| < (1-t)\epsilon\}$.  Clearly $\Delta = \cup_{x \in \Theta} F_x$.   Choosing $z \in \Theta$ so that $\#F_{z} = \max\{\#F_x : x \in \Theta\}$,

\begin{eqnarray*} S_{\epsilon}(E) & = & \#\Delta \\
                                                      & \leq & \#(\cup_{x \in \Theta} F_x) \\
                                                      & \leq & \# \Theta \cdot \max\{\#F_x : x \in \Theta\} \\
                                                      & \leq & K_{(1-t)\epsilon}(Q^{\bot}(E)) \cdot \#F_{z}.\\
\end{eqnarray*}

Suppose $x$ and $y$ are two distinct points in $F_z$.   Since $F_z \subset \Delta$, $\|x-y\| \geq \epsilon$.   On the other hand, by definition $\|z -Q^{\bot}(x)\| < (1-t)\epsilon$ and $\|z -Q^{\bot}(y)\| < (1-t)\epsilon$ so $\|Q^{\bot}(x) - Q^{\bot}(y)\| < 2(1-t)\epsilon$.

\begin{eqnarray*} \|Q(x) - Q(y)\|^2 & = & \|x-y\|^2 - \|Q^{\bot}(x) - Q^{\bot}(y)\|^2 \\                  
                                                                         & > & \|x-y\|^2 - 4(1-t)^2 \epsilon^2 \\
                                                                         & \geq & (1-4(1-t)^2) \|x-y \|^2 \\  
                                                                         & \geq & \frac{7}{8} \cdot \|x -y\|^2.  
\end{eqnarray*}

\noindent Applying the preceding lemma with $Q=T$ and $r = \beta/4$ yields

\begin{eqnarray*} \|f(y) - f(x) \| & = & \beta \cdot \|Q(x) - Q(y)\| - \|Df(0)\| \cdot \|Q^{\bot}(x) -Q^{\bot}(y)\| - \frac{\beta}{4} \cdot \|y-x\| \\                               & \geq & \frac{3\beta}{4} \cdot \|x - y\| - \|Df(0)\| \cdot 2(1-t) \epsilon - \frac{\beta}{4} \cdot \|y-x\| \\               
                                                & \geq & \frac{\beta}{2} \cdot \|x-y\| - \frac{\beta \epsilon}{4}\\
                                                & \geq & \frac{\beta \epsilon}{4}.\\
\end{eqnarray*}

\noindent This being true for any distinct $x, y \in F_z \subset E$, $S_{\frac{\beta \epsilon}{4}}(f(E)) \geq\#F_z$.  Using the inequality from the previous paragraph,

\begin{eqnarray*} K_{\epsilon}(E) & \leq & S_{\epsilon}(E) \\
                                                      & \leq & K_{(1-t)\epsilon}(Q^{\bot}(E)) \cdot \#F_z\\
                                                      & \leq & K_{(1-t) \epsilon}(A(E)) \cdot S_{\frac{\beta \epsilon}{4}}(f(E)). \\
\end{eqnarray*}
\end{proof}

\begin{remark} The spectral projections of $|Df(x_0)|$ (regarded as a symmetric, positive semidefinite operator on $\mathbb R^n$) always satisfy the inequality for $Q$ in the lemma, i.e., $\|Df(x_0)Qv\| \geq\beta\|Qv\|$ when $Q=1_{[\beta,\infty)}(|Df(x_0)|)$.
\end{remark}

\begin{remark} One can carry out a manifold-themed argument of Lemma 4.2 by quantifying the usual rank theorem in multivariable calculus, (e.g., Theorem 9 (1) in \cite{spivak}).  While this may seem more intuitive, there is a notational cost, as one must perform more bookkeeping of Lipschitz constants of charts.  In the course of doing this the dependencies of the constants lose some transparency.  
\end{remark}

Notice that in the above lemma, the upper bound for the $\epsilon$-covering number of $E$ is scaled by quantities which are not arbitrarily close to $1$, e.g., $1-t < 1/8$.  One can improve the concluding estimate if one assumes $Q$ is a spectral projection of $|Df(x_0)|$, but the improvement only allows $t>1/2$ and doesn't permit arbitrarily close values to $0$.  This is largely irrelevant for the free entropy/Hausdorff arguments presented here, but they will be unsuitable for the entropy estimates in Section 5.  

\subsection{Free Entropy Dimension Inequalities}

Recall in Section 3.2 that $\mathfrak{A}_n$ denotes the universal, unital, complex $*$-algebra $\mathfrak{A}_n$ on $n$-indeterminates.  Fix a $p$-tuple $F = (f_1,\ldots,f_p) \subset \mathfrak{A}_n$ and any uniformly bounded positive Borel function $\phi:[0,\infty) \rightarrow \mathbb R_+$.  Define successively,

\[ \Pi_{F, \phi,R}(X,r, \epsilon;m,k,\gamma) = \sup_{\xi \in \Gamma_R(X;m,k,\gamma)} \log (K_{\epsilon}(\phi(|DF(\xi)|)[B_{\infty}(\xi,r) \cap \Gamma_R(X;m,k,\gamma) - \xi)])),
\]

\[ \Pi_{F, \phi,R }(X,r, \epsilon;m,\gamma) = \limsup_{k \rightarrow \infty} k^{-2} \cdot \Pi_{F, \phi, R}(X,r, \epsilon;m,k,\gamma),  
\]

\[ \Pi_{F, \phi, R}(X,r, \epsilon) = \inf \{\Pi_{F, \phi, R}(X,r, \epsilon;m,\gamma): m\in \mathbb N, \gamma >0\},
\]

\[ \Pi_{F,\phi, R}(X,r) = \limsup_{\epsilon \rightarrow 0^+} \frac{\Pi_{F, \phi,R}(X,r, \epsilon)}{|\log \epsilon|}.
\] 
Observe that in the above $DF(\xi)$ is the derivative of $F$ regarded as a smooth map from the real Hilbert space $(M_k(\mathbb C))^n$ into $(M_k(\mathbb C))^p$ and $|DF(\xi)|$ is its absolute value.  In 
the proof that follows below $\|DF(\xi)\|$ will designate the operator norm of $DF(\xi)$ computed w.r.t. the canonical real inner product norms on $(M_k(\mathbb C))^n$ and $(M_k(\mathbb C))^p$.


%

\begin{theorem} Suppose $R > \max_{x \in X} \|x\|$.  There exists $C, \kappa >1$ dependent on $F$ and $R$ such that if $\beta >0$, $\phi_{\beta} = 1_{[0,\beta)}$, and $0< r < \min\{\frac{1}{2}, \frac{\beta}{8\kappa}\}$, then for any $\epsilon >0$

\begin{eqnarray*} \mathbb K_{\epsilon, R}(X) & \leq & 2 n \log(R) + 2 n \cdot|\log r|  + \Pi_{F,\phi_{\beta},R}(X, r, \tfrac{\beta \epsilon}{9C}) + \mathbb S_{\frac{\beta\epsilon}{4}}(F(X):X). \\
\end{eqnarray*}
Hence, 
\begin{eqnarray*} \delta_0(X) \leq \Pi_{F, \phi_{\beta}, R}(X, r) + \delta_0(F(X):X).
\end{eqnarray*}
\end{theorem}

\begin{proof} Fix $\kappa, C >1$ dependent on $F$ and $R$ such that for any $k \in \mathbb N$, $\xi, \eta \in ((M_k(\mathbb C))_{R+1})^n$, $\|DF(\xi) - DF(\eta)\| \leq \kappa \|\xi - \eta\|_{\infty}$ and $\|F(\xi)\|_{\infty}, \|DF(\xi)\|+1 < C$.  Suppose $0< r < \min\{\frac{1}{2}, \frac{\beta}{8\kappa}\}$ and $\epsilon >0$. Given $m \in \mathbb N$ and $\gamma >0$, there exist $m < m_1 \in \mathbb N$ and $\gamma >\gamma_1 >0$ with the property that for any $k$, $F(\Gamma_R(X;m_1,k,\gamma_1)) \subset \Gamma_C(F(X):X;m,k,\gamma)$.


For each $k$ find an $r$-cover $\langle w_{(j,k)} \rangle_{j \in J_k}$ for $\Gamma_R(X;m_1,k,\gamma_1) \subset ((M_k(\mathbb C))_R)^n$ with respect to the operator norm such that 

\begin{eqnarray*} \#J_k \leq \left(\frac{2R}{r} \right)^{2nk^2}.
\end{eqnarray*}

\noindent Using the triangle inequality, I can assume that $w_{(j,k)} \in \Gamma_R(X;m_1,k,\gamma_1)$ at the expense of replacing the $r$-cover condition with a $2r$-cover condition.  Fix $k \in \mathbb N$, $j \in J_k$, and set $E = \Gamma_R(X;m_1,k,\gamma_1) \cap B_{\infty}(w_{(j,k)}, 2r)$.  Now $w_{(j,k)} \in E \subset B_{\infty}(w_{(j,k)}, 2r)$ with $B_{\infty}(w_{(j,k)}, 2r)$ clearly convex and open with respect to the $\|\cdot\|_2$-metric (all norms on a finite dimensional space being equivalent).  Moreover, for any $\xi, \eta \in B_{\infty}(w_{(j,k)}, 2r)$, $\|DF(\xi) - DF(\eta)\| \leq \kappa \|\xi -\eta\|_{\infty} < \kappa \cdot 2r < \frac{\beta}{4}$.  Applying Lemma 4.2 with $t = 1 - \frac{\beta}{9C}$ and $A$ equal to the contractive mapping $\xi \mapsto  1_{[0,\beta)}(|DF(w_{(j,k)}|)(\xi - w_{(j,k)})$, for any $\epsilon >0$
\begin{eqnarray*}
K_{\epsilon}(E) & \leq & K_{\frac{\beta \epsilon}{9C}}(A(E)) \cdot S_{\frac{\beta \epsilon}{4}}(F(E)).\\ 
\end{eqnarray*}
Observe that 
\begin{eqnarray*} K_{\frac{\beta \epsilon}{9C}}(A(E)) & = & K_{\frac{\beta \epsilon}{9C}}[1_{[0,\beta)}(|DF(w_{(j,k)})|)(E - w_{(j,k)})] \\                        & = & K_{\frac{\beta \epsilon}{9C}}[1_{[0,\beta)}(|DF(w_{(j,k)})|)(\Gamma_R(X;m_1,k,\gamma_1) \cap B_{\infty}(w_{(j,k)}, r) - w_{(j,k)}] \\ 
                                     & \leq & \Pi_{F,\phi_{\beta},R}(X, r, \tfrac{\beta \epsilon}{9C}; m_1,k,\gamma_1).
\end{eqnarray*}
Hence, generously majorizing,
\begin{eqnarray*}  K_{\epsilon}(E) & \leq & K_{\frac{\beta \epsilon}{9C}}(A(E)) \cdot S_{\frac{\beta \epsilon}{4}}(F(E)) \\
                                    & \leq & \Pi_{F,\phi_{\beta},R}(X, r, \tfrac{\beta \epsilon}{9C}; m_1,k,\gamma_1) \cdot S_{\frac{\beta \epsilon}{4}}(F(E))\\  
                                    & \leq & \Pi_{F,\phi_{\beta},R}(X, r, \tfrac{\beta \epsilon}{9C}; m_1,k,\gamma_1) \cdot S_{\frac{\beta \epsilon}{4}}(\Gamma_C(F(X):X;m,k,\gamma)).\\         
\end{eqnarray*}
 
Using the subadditivity of covering numbers and the estimate above:

\begin{eqnarray*} K_{\epsilon}(\Gamma_R(X;m_1,k,\gamma_1)) & \leq & \sum_{j \in J_k} K_{\epsilon}\left (\Gamma_R(X;m_1,k,\gamma_1) \cap B_{\infty}(w_{(j,k)}, 2r) \right)\\ & \leq & \#J_k \cdot \Pi_{F,\phi_{\beta}}(X, r, \tfrac{\beta \epsilon}{9C}; m_1,k,\gamma_1) \cdot S_{\frac{\beta \epsilon}{4}}(\Gamma_C(F(X):X;m,k,\gamma))  \\ & \leq & \left(\frac{2R}{r} \right)^{2nk^2} \cdot \Pi_{F,\phi_{\beta},R}(X, r, \tfrac{\beta \epsilon}{9C};m_1,k,\gamma_1) \cdot S_{\frac{\beta \epsilon}{4}}(\Gamma_C(F(X):X;m,k,\gamma)). \\
\end{eqnarray*}
Thus,

\begin{eqnarray*} \mathbb K_{\epsilon, R}(X) & \leq & \mathbb K_{\epsilon, R}(X;m_1,\gamma_1) \\ 
& = & \limsup_{k \rightarrow \infty} k^{-2} \cdot \left[ \log(K_{\epsilon}(\Gamma_R(X;m_1,k,\gamma_1)) \right ] \\ & \leq & 2n \log(2R) + 2n \cdot|\log r| + \Pi_{F,\phi_{\beta},R}(X, r, \tfrac{\beta \epsilon}{9C};m_1,\gamma_1) + \mathbb S_{\frac{\beta \epsilon}{4}}(F(X):X;m,\gamma).\\
\end{eqnarray*}
This holds for any choice of $m, \epsilon$ and $\gamma$  with $m_1 > m$ and $\gamma_1 < \gamma$ so,
\begin{eqnarray*} \mathbb K_{\epsilon,R}(X) & \leq & 2 n \log(2R) + 2 n \cdot|\log r|  + \Pi_{F,\phi_{\beta},R}(X, r, \tfrac{\beta \epsilon}{9C}) + \mathbb S_{\frac{\beta\epsilon}{4}}(F(X):X). \\
\end{eqnarray*}
as claimed.

For the second inequality, taking $R > \max_{x \in X} \|x\|$, dividing both sides by $|\log \epsilon|$, and taking a limit as $\epsilon \rightarrow 0$, it follows from the cutoff covering formulation of $\delta_0$ (Section 2.3) that

\begin{eqnarray*} \delta_0(X) & \leq & \limsup_{\epsilon \rightarrow 0^+} \frac{2n \log 2R + 2n \cdot |\log r|}{|\log \epsilon|} + \limsup_{\epsilon \rightarrow 0^+} \frac{\Pi_{F,\phi_{\beta},R}(X, r, \tfrac{\beta \epsilon}{9C})}{|\log(\tfrac{\beta \epsilon}{9C})|} \cdot \frac{|\log(\tfrac{\beta \epsilon}{9C})|}{|\log\epsilon|} + \\ & & \limsup_{\epsilon \rightarrow 0^+} \frac{\mathbb S_{\frac{\beta \epsilon}{4}}(F(X):X)}{|\log (\frac{\beta \epsilon}{4})|} \cdot \frac{|\log(\tfrac{\beta \epsilon}{4})|}{|\log\epsilon|}\\ & \leq & \Pi_{F,\phi_{\beta}, R}(X, r) + \delta_0(F(X):X).
\end{eqnarray*}
\end{proof}

\begin{remark}  In the proof of Theorem 4.5 one localizes the microstate space into balls of operator norm radius of order $r$.  This is necessary to obtain uniform control of the derivatives via the Lipschitz estimate $\|DF(\xi) - DF(\eta) \| \leq \kappa \cdot \|\xi - \eta\|_{\infty}$.  Such uniform estimates  really require the operator norm and are not available with the $\|\cdot\|_2$-norm.  However, notice that the $\|\cdot\|_2$-norm is also used in a crucial way through the Euclidean/orthogonal estimates of Lemma 4.2.  The use of both the $L^{\infty}$ and $L^2$ norms appears to be a minor detail here (in terms of metric entropy the $L^p$-norms are all equivalent by \cite{sr} up to an exponential factor).  However, it will have rather severe consequences in subsequent entropy estimates involving iterative spectral splits.\end{remark}

\begin{proposition} For any $\beta >0$ define $\phi_{\beta} = 1_{[0,\beta)}$.  If $t, R >0$, then there exists a $\rho$ such that for any $0 < \beta < \rho$ and $r, \epsilon >0$, 

\begin{eqnarray*}
\Pi_{F, \phi_{\beta},R}(X, r, \epsilon) \leq (\text{Nullity}(D^sF(X)) + t) \cdot  \log \left(\frac{2r}{\epsilon} \right).
\end{eqnarray*}

\noindent Consequently, 

\begin{eqnarray*} \sup_{r, R>0} \Pi_{F,\phi_{\beta},R}(X,r) \leq \text{Nullity}(D^sF(X)) +t.
\end{eqnarray*}

\end{proposition}

\begin{proof} Suppose $t, R >0$.  By Proposition 3.13 there exist $m \in \mathbb N$ and $\rho, \gamma > 0$ such that if $\xi \in \Gamma_R(X;m,k,\gamma)$ and $0 < \beta < \rho$, then the real dimension of the range of the real orthogonal projection $1_{[0,\beta)}(|DF(\xi)|)$ on $(M_k(\mathbb C))^n$ is no greater than $(\text{Nullity}(D^sF(X))+ t)k^2$.  It follows from this and coarse covering estimates for Euclidean balls (Section 2.4)  that for any $\epsilon >0$,

\begin{eqnarray*} \Pi_{F,\phi_{\beta},R}(X,r, \epsilon;m,k,\gamma) \leq (\text{Nullity}(D^sF(X)) + t)k^2 \cdot  \log \left(\frac{2r}{\epsilon} \right).
\end{eqnarray*}

\noindent Thus, $\Pi_{F, \phi_{\beta},R}(X, r, \epsilon) \leq (\text{Nullity}(D^sF(X)) + t) \cdot  \log \left(\frac{2r}{\epsilon} \right)$.  

It follows that $\Pi_{F,\phi_{\beta}, R}(X,r) \leq \text{Nullity}(D^sF(X)) + t$.  This is true for all $0 < r, R$ and completes the proof. 
\end{proof}

Combining Proposition 4.7 with Theorem 4.5 yields

\begin{corollary} \begin{eqnarray*} \delta_0(X) & \leq & \text{Nullity}(D^sF(X)) + \delta_0(F(X):X) \\
                                                & =  & 2n - \text{Rank}(D^sF(X)) + \delta_0(F(X):X). \\
                           \end{eqnarray*}
\end{corollary}

When $X$ consists of self-adjoints or unitaries and $F$ preserves either condition, then one can replace the $D^s$ calculus with the self-adjoint or unitary calculus discussed in Section 3 to arrive at the following:  

\begin{corollary} If $X \subset M$ and $F \subset \mathfrak{A}_n$ are self-adjoint $n$ and $m$-tuples, then  
\begin{eqnarray*} \delta_0(X) & \leq & \text{Nullity}(D^{sa}F(X)) + \delta_0(F(X):X) \\
                                                & =  & n - \text{Rank}(D^{sa}F(X)) + \delta_0(F(X):X). \\
\end{eqnarray*}
\end{corollary}

\begin{proof} Define $L =\{(X_1 - X_1^*)/2,\ldots, (X_n- X_n^*)/2\} \subset \mathfrak{A}_n$.  If $G = F \cup L$, then by Proposition 3.17, $\text{Nullity}(D^sG(X)) = \text{Nullity}(D^{sa}F(X))$.  Also, $L(X) = \{0,\ldots, 0\}$.  Thus, by Corollary 4.8,

\begin{eqnarray*} \delta_0(X) & \leq & \text{Nullity}(D^sG(X)) + \delta_0(G(X):X) \\
& = & \text{Nullity}(D^{sa}F(X)) + \delta_0(F(X) \cup L(X):X) \\
& = & \text{Nullity}(D^{sa}F(X)) + \delta_0(F(X):X) \\
& = & n - \text{Rank}(D^{sa}F(X)) + \delta_0(F(X):X). \\
\end{eqnarray*}
\end{proof}

The unitary case follows from similar considerations:

\begin{corollary} If $X \subset M$ and $F \subset \mathfrak{A}_n$ are $n$ and $m$-tuples where the elements of $X$ are unitaries and $F$ consists of $*$-monomials, then
\begin{eqnarray*} \delta_0(X) & \leq & \text{Nullity}(D^{u}F(X)) + \delta_0(F(X):X) \\
                                                & =  & n - \text{Rank}(D^{u}F(X)) + \delta_0(F(X):X). \\
\end{eqnarray*}
\end{corollary}

\begin{proof} Define $G =\{X_1^*X_1-I,\ldots, X_n^* X_n-I\} \subset \mathfrak{A}_n$.  If $H = F \cup G$, then by Proposition 3.27, $\text{Nullity}(D^sH(X)) = \text{Nullity}(D^uF(X))$.  Also, $G(X) = \{0,\ldots, 0\}$.  Thus, by Corollary 4.8,

\begin{eqnarray*} \delta_0(X) & \leq & \text{Nullity}(D^sH(X)) + \delta_0(H(X):X) \\
& = & \text{Nullity}(D^uF(X)) + \delta_0(F(X) \cup G(X):X) \\
& = & \text{Nullity}(D^uF(X)) + \delta_0(F(X):X) \\
& = & n - \text{Rank}(D^uF(X)) + \delta_0(F(X):X). \\
\end{eqnarray*}
\end{proof}

\subsection{Examples}

The first example is reassuring but not particularly enlightening:

\begin{example} Suppose $X = \{x_1, x_2\}$ consists of commuting self-adjoint elements, $x_1$ has no eigenvalues, and $F = \{f\}$ where $f = X_2 X_1 - X_1^* X_2^* \in \mathfrak{A}_2$.  Clearly $F(X) = (f(X)) = 0$.  By definition,
\begin{eqnarray*} D^{sa}F(X) & = & 
\begin{bmatrix}
(\partial_{1}^{sa} f)(X) & (\partial_{2}^{sa} f)(X) \\
\end{bmatrix} \\ & = & 
\begin{bmatrix}
x_2 \otimes I - I \otimes x_2 & I \otimes x_1 - x_1 \otimes I \\
\end{bmatrix} \\
& \in & M_{1\times 2}(M \otimes M^{op}). \\
\end{eqnarray*}
Since $x_1$ is self-adjoint and has no eigenvalues, $I \otimes x_1 - x_1 \otimes I$ is injective.  To see this observe that the moments of $I \otimes x_1 - x_1 \otimes I$ are the moments of the convolution of the spectral distribution of $x_1$ with that of $-x_1$; a convolution of non-atomic measures being nonatomic, the spectral distribution of $I \otimes x_1 - x_1 \otimes I$ is non-atomic, and in particular, the singleton set $\{0\}$ has $0$ measure. By faithfulness of the trace and the spectral theorem $I \otimes x_1 -x_1 \otimes I \in M \otimes M^{op}$ is injective and thus has dense range.  In turn, this implies that $D^{sa}F(X)$ has dense range, so that $\text{Rank}(D^{sa}F(X)) =1$.  Thus $\text{Nullity}(D^{sa}F(X)) =2-1=1$ and as one would expect from Corollary 4.9
\begin{eqnarray*} \delta_0(X) & \leq & \text{Nullity}(D^{sa}F(X)) + \delta_0(F(X):X) \\
                                                & \leq & 1 + 0 \\
                                                & = & 1.\\
\end{eqnarray*}
\end{example}

Here is a slightly more complex example: 

\begin{example} Suppose $X = \{x_1, \ldots, x_n\}$ consists of unitaries in $M$ and $F = \{f\}$ where $f = A X_1^{s_1} B X_1^{s_2} \in \mathfrak{A}_n$, $A$ and $B$ are $*$-monomials in $X_2,\ldots, X_n$, and either $s_1=s_2=1$ or $s_1 = 1$ and $s_2=*$.   Set $a = A(X)$ and $b = B(X)$; $a$ and $b$ are $*$-monomials in $x_2,\ldots, x_n$.   Assume that $bx_1$ has no eigenvalues when $s_1=s_2=1$ or that $b$ has no eigenvalues when $s_1=1, s_2=*$, and that in either case, $f(X)=I$.
\begin{eqnarray*} D^{u}F(X) & = & 
\begin{bmatrix}
(\partial_{1}^{u} f)(X) & \cdots &  (\partial_{n}^{u} f)(X) \\
\end{bmatrix} \\
& \in & M_{1\times n}(M \otimes M^{op}). \\
\end{eqnarray*}
Using the fact that $s_i \in \{1,*\}$, the unitary calculus rule in Remark 3.24 yields 
\begin{eqnarray*} \partial_1^uf(X) & = & \begin{cases} x_1^*b^* \otimes b x_1 + I_M \otimes I_{M^{op}} &\mbox{if } s_1 = s_2=1 \\
(x_1 \otimes x_1^*) (b^* \otimes b - I_M \otimes I_{M^{op}}) & \mbox{if } s_1=1, s_2=*\\
\end{cases}\\                      
\end{eqnarray*}
The assumption on the absence of eigenvalues shows that in either case above, the tensor product operators are injective.  This is because as in the self-adjoint case of Example 4.1 one can identify the tensor product operators as a product of multiplication operators on two, independent, nonatomic probability spaces with supports in the unit circle and argue accordingly.  $\partial_1^uf(X) \in M \otimes M^{op}$ is injective and thus has dense range.  In turn, this implies that $\text{Rank}(D^uF(X))$ has dense range so that $\text{Rank}(D^uF(X)) =1$.  Thus, $\text{Nullity}(D^uF(X)) =n -1$.  Since $f(X)=I$ Corollary 4.10 implies that 
\begin{eqnarray*} \delta_0(X) & \leq & \text{Nullity}(D^{u}F(X)) + \delta_0(F(X):X) \\
                                                & \leq & 1 + 0 \\
                                                & = & n-1.\\
\end{eqnarray*}
   
The absence of eigenvalues condition (diffuseness) occurs naturally.  Consider for instance, the canonical unitaries associated to the generators $a,b$ for the group $\Gamma$ generated by the single relation $a^mb^{s_1}a^nb^{s_2}=e$ where $m, n \in \mathbb Z$ and $s_1,s_2 \in \{1,-1\}$.  By the above computation $\delta_0(\Gamma) \leq 1$.  When $s_1=1$ and $s_2 =-1$ these groups are the Baumslag-Solitar groups, $\Gamma_{m,n}$.   In this case $b$ is in the normalizer of $a$ and the inequality reduces to a case first obtained by \cite{v2} (see also the strengthened generalizations in \cite{gs}).  Both sets of authors actually show that the group von Neumann algebras are strongly $1$-bounded.  I will find a different proof of this in terms of the spectral distribution of the derivative.   Note that the isomorphism classes of these group von Neumann algebras has been studied and partially classified in \cite{ns}.
\end{example}

\begin{example}  Suppose $X = \{x_1,\ldots, x_n\}$ consists of self-adjoint elements and $F =\{f_1,\ldots, f_p\}$ are self-adjoint.  Assume further that $\delta_0(X)=n$, the maximum possible value.   By Corollary 4.9,
\begin{eqnarray*}
n & = & \delta_0(X) \\
   & \leq & n - \text{Rank}(D^{sa}F(X)) + \delta_0(F(X):X) \\
   & \leq & n - \text{Rank}(D^{sa}F(X)) + \delta_0(f(X)), \\
\end{eqnarray*}
whence, $\text{Rank}(D^{sa}F(X)) \leq \delta_0(f(X))$.  Thus, computing the rank of $D^{sa}F(X)$ gives a lower bound on $\delta_0(f(X))$. 

When $p=1$ $F$ consists of a single selfadjoint element $f$.  By \cite{v1}, $\delta_0(f(X))=1$ iff $f(X)$ has no eigenvalues.  So showing that $\text{Rank}(D^{sa}F(X)) =1$ guarantees that $F(X)$ has no eigenvalues.  This example is connected to the rank/nullity computation of matrices with operator-valued entries arising from freely independent self-adjoint/unitary tuples (\cite{ss}) as well as the free entropy dimension computations for a single self-adjoint polynomial under maximal free entropy dimension assumptions on the tuple (\cite{msm}).  
\end{example}

\begin{example} Suppose $X = \{x_1,\ldots, x_n\} \subset M$ and $F=\{f_1,\ldots, f_{n-1}\} \subset \mathfrak{A}_n$ consist of self-adjoint elements such that for each $1 \leq i \leq n-1$, $\partial^{sa}_i f_i(X) \in M \otimes M^{op}$ has dense range.  It is easily seen that $D^{sa}F(X)$ is the upper triangular $(n-1) \times n$ matrix 

\begin{eqnarray*}
\begin{bmatrix}
(\partial_{1}^{sa} f_1)(X) & * & \cdots & * & (\partial_{n}^{sa} f_1)(X) \\ 0 & (\partial_{2}^{sa} f_2)(X) &  \cdots & & \vdots \\ \vdots & \ddots & \ddots & & \vdots \\
0 & \cdots & 0 & (\partial_{n}^{sa} f_{n-1})(X) & (\partial_{n}^{sa} f_{n-1})(X) \\
\end{bmatrix} \in M_{(n-1)\times n}(M \otimes M^{op}).
\end{eqnarray*}

\noindent By Proposition 2.12, $\text{Rank}(D^{sa}F(X))= n-1$ so that from Corollary 4.9, $\delta_0(X) \leq 1 + \delta_0(F(X):X)$. 

\cite{v5} studied finite sequences of Haar unitaries $u_1, \ldots, u_n$ satisfying the pairwise commutation relation $u_i u_{i+1} = u_{i+1} u_i$ for $1 \leq i \leq n-1$.  It was shown that $\delta_0(u_1,\ldots, u_n) \leq 1$.  This was subsequently generalized and strengthened in \cite{gs}.  The example here provides another way to see how pairwise commutativity of generators affects free entropy dimension.  

Indeed, suppose $x_1, \ldots x_n$ are self-adjoint, diffuse elements in a tracial von Neumann algebra and $x_i x_{i+1} = x_{i+1} x_i$ for $1 \leq i \leq n-1$.    Set $f_i = X_i X_{i+1} - X^*_{i+1} X^*_i$ for $1\leq i \leq n-1$ and $F = \{f_1, \ldots, f_{n-1}\}$, $X = \{x_1, \ldots, x_n\}$.  For each $i$, $(\partial_i^{sa} f_i)(X) = I \otimes x_{i+1} - x_{i+1}\otimes I \in M \otimes M^{op}$.   As observed in Example 4.1, each of these operators has dense range, regarded as operators on $L^2(M \otimes M^{op})$.  $F(X) =0$ so that $\delta_0(F(X):X)=0$.  Applying the preceding paragraph yields $\delta_0(X) \leq 1$.  Under the additional assumption of embeddability of $X$ into an ultraproduct of the hyperfinite $\mathrm{II}_1$-factor, $\delta_0(X)=1$ by \cite{j0}.  As in Example 4.2, I'll show that all of these examples are strongly $1$-bounded by studying the spectral distribution of the derivative of $F$.
\end{example}

\begin{example} Suppose $\Gamma$ is a one-relator group on $n$ generators such that its relator is not a proper power (the relator cannot be written as a proper power of another element).  The relator yields a $*$-monomial $w$ on $n$-indeterminates such that when applied to the canonical $n$-tuple of group unitaries $X$, satisfies the property that $\text{Rank}(D^uw(X))=1$.  This follows from combining \cite{b}, \cite{h}, and \cite{l}; see Proposition 7.5 for a full proof.  It follows from Corollary 4.10 that $\delta_0(\Gamma) = \delta_0(X) \leq n -1$.  This is to be expected.  Indeed, denoting by $\delta^*$ the nonmicrostates free entropy in \cite{v4}, \cite{cs} combined with \cite{dl} yields 
\begin{eqnarray*}
\delta_0(\Gamma) & \leq & \delta^*(\Gamma) \\
                              & \leq & \beta_1^{(2)}(\Gamma) +1\\
                              & = & (n-2) + 1 \\
                              & = & n-1.\\  
\end{eqnarray*}
This result and in fact a stronger entropy inequality will be stated and proven in Section 7.
\end{example}

\section{Iterating Spectral Splits I: Heuristics and Technical Estimates}

In this section I'll discuss how to upgrade the results of the preceding section to obtain finiteness results for free packing and Hausdorff entropy and establish the necessary technical machinery.   These results will be applied in later sections to provide new examples of strongly $1$-bounded von Neumann algebras arising from one relator discrete groups.   

Establishing entropy upper bounds in this context is considerably more difficult than establishing the free entropy dimension bounds of the previous section.  Before getting into the details I'll give an overview of how they are related (the impatient, nuts-and-bolts reader can skip to 5.2 for the beginning of the proofs).  After this informal discussion, I'll build tools for the next section.  They come in three parts: 1) Euclidean estimates; 2) elementary von Neumann algebra approximations; 3) estimates for polynomials of matrices.  
\subsection{Overview}
I will review in broad terms the previous section's main free entropy dimension argument and explain how it fails to provide an entropy bound.  Then I'll discuss how to overcome this.  In what follows $DF(X)$ denotes the derivative of $F$ at $X$ in a general informal sense, i.e., the distinction between $D^s$ and the normal derivative will be blurred.  

Roughly, the proof of Theorem 4.7 goes like this.  Suppose $F(X) =0$.   By definition, $\Gamma_R(X;m,k,\gamma)$ is contained in the ball of operator norm radius $R$ in $(M_k(\mathbb C))^n$ so one can cover $\Gamma_R(X;m,k,\gamma)$ by no more than $(2R/r)^{2nk^2}$ balls of $\emph{operator norm}$ radius $r <1$ with centers in $\Gamma_R(X;m,k,\gamma)$.  The intersection of $\Gamma_R(X;m,k,\gamma)$ with each of these balls of operator norm radius $r$ has the property that $DF$ varies (in $\emph{operator norm}$ as a real Hilbert space operator from $(M_k(\mathbb C))^n$ into $(M_k(\mathbb C))^d$) by no more than $r$ times a fixed constant determined by $F$, $R$, and the submultiplicativity of the operator norm.  This uniform bound on the derivative combined with orthogonality estimates and a spectral splitting parameter $\beta$ allows one to dominate the $\|\cdot\|_2$-metric entropy on each $r$ operator norm neighborhood by the entropy of an $\|\cdot\|_2$ $r$-ball in the kernel of the differential plus a $t_1k^2$-approximate subspace.  Here $t_1$ depends on how small $r$ is, which in turn is driven by the trace of the spectral projection $1_{(0,\beta)}(|DF(X)|)$ .  An $\epsilon$-covering (w.r.t. $\|\cdot\|_2$) bound for this approximate "kernel" $r$-ball multiplied by the initial $r$-covering bound gives a bound for  $K_{\epsilon}(\Gamma_R(X;m,k,\gamma))$:

\begin{eqnarray*}   \left(\frac{2R}{r} \right)^{2nk^2} \cdot \left(\frac{r}{\epsilon} \right)^{(\text{Nullity}(DF(X)) + t_1) k^2}.
\end{eqnarray*}

\noindent The microstates limiting process extracts the normalized exponent as a bound - $\text{Nullity}(DF(X)) +t_1$ and this will be an upper bound for the free entropy dimension.  Letting $t_1 \rightarrow 0$ shows that the free entropy dimension is dominated by $\text{Nullity}(DF(X))$.  Unfortunately the bound, as is, fails to provide a free packing entropy estimate with growth exponent $\text{Nullity}(DF(X))$ because of the residual error term $t_1$.  This process was called \textbf{splitting the spectrum} in the introduction.

In the above argument one covers the microstate spaces by $r$ operator norm balls, and then covers the intersection of the microstate space with each of these $r$-balls by $\epsilon$-balls taken from perturbed copies of $\ker(DF(X))$.  Now repeat this process on each of these local subspace $\epsilon$-balls, intersecting them with the microstate space, and then covering them by appropriate balls of radius $\epsilon^2$ via some suitable Euclidean estimate (e.g. something like Lemma 4.2) again.  The advantage of zooming in at a further $\epsilon$-scale is that the differential of the polynomial tuple moves less (less curvature), and yields a smaller approximate subspace perturbation, say with error $t_2 < t_1$ that is even closer to $\text{Nullity}(DF(X))$.  One then repeats this process on the subspace balls of radius $\epsilon^3$, picking up an even better approximating subspace perturbation.  Iterating this spectral splitting process $p$ times and keeping $\epsilon$ $\emph{fixed}$ throughout, a back-of-the-envelope computation shows

\begin{eqnarray*} K_{\epsilon^p}(\Gamma_R(X;m,k,\gamma)) &\leq &  \left(\frac{2R}{r} \right)^{2nk^2} \left(\frac{r}{\epsilon} \right)^{(\text{Nullity}(DF(X)) + t_1) k^2}  \left(\frac{\epsilon}{\epsilon^2} \right)^{(\text{Nullity}(DF(X)) + t_2) k^2} \cdot \\
 && \left(\frac{\epsilon^2}{\epsilon^3} \right)^{(\text{Nullity}(DF(X)) + t_3) k^2}  \cdots \left(\frac{\epsilon^{p-1}}{\epsilon^p} \right)^{(\text{Nullity}(DF(X)) + t_p) k^2} \\
                             & \leq & \left(\frac{2R}{r} \right)^{2nk^2} \left(\frac{1}{\epsilon^p} \right)^{[\text{Nullity}(DF(X))]k^2} \left(\frac{1}{\epsilon}\right)^{(t_1 + \cdots +t_p)k^2}.
\end{eqnarray*}

\noindent The last expression will yield a finite packing entropy bound provided that the exponent term in the last line, $t_1 + \cdots + t_p$, can be uniformly bounded for all $p$.  With some work this can be guaranteed by imposing decay conditions on the spectral distribution of $|DF(X)|$ near $0$.  Somewhat surprisingly, this spectral decay condition is equivalent to $DF(X)$ having finite Fuglede-Kadison-L{\"u}ck Determinant.  This problem where one tries to bound a space which has the same local structure at every scale (in this case on the scale of powers of a fixed $\epsilon$) is similar to the situation of bounding the Hausdorff measure of a self-similar fractal.

There are, however, some strong assumptions made in the iterative spectral splitting argument which must be addressed.  Call a covering inequality \textit{asymptotically coarse} if its dominating term explicitly contains constants of the form $C^{\alpha k^2}$ where $C >1$ and $\alpha >0$.  Asymptotically coarse inequalities involving constants $C$ uniformly bounded below from $1$ can be made a fixed, finite number of times without destroying the qualitative nature of an upper bound for entropy.  They are lethal in the context above.  For example, imagine that each of the iterations of the computation above involved a covering bound involving an additional factor of $2^{k^2}$, a seemingly benign bound that appears for example in the comparison of the $\|\cdot\|_{\infty}$ and $\|\cdot \|_2$ norms in Section 2.4 or in the single spectral split argument for the dimension above.  After $p$-iterations one would end up with a constant of the form $2^{p k^2}$ and after taking appropriate limits one has an additive factor of $p \log 2$ which converges to $\infty$ as $p \rightarrow \infty$.  This would leave a vacuous upper bound of $\infty$.   In the estimate above there appear to be no asymptotically coarse inequalities.  However, one should expect these in two places:1) scaling estimates for balls; 2) norm switching.       

1) refers to finding sharp bounds for coverings of the unit ball by $\epsilon$-balls in Euclidean space.  In the above I assumed that one could find such a cover with no more than $(1/\epsilon)^n$-balls as opposed to say $(C/\epsilon)^n$.  This appears in the ratios of powers of $\epsilon$ assumed in the first pass computation above.  While false in this strict form, it is asymptotically true (\cite{rogers}). 

2) is both more subtle and troublesome.  By norm switching I'm referring to the process of relying on different norms in order to make different estimates.  Norm switching occurred in the dimension argument.  There I used properties specific to both the $L^{\infty}$ and the $L^2$-norm.    While one can account for entropy changes when moving from one norm to the other with St. Raymond's volume computations \cite{sr}, the error terms are asymptotically coarse.  

I'll deal with 2) by using the $L^2$-norm metric exclusively for covering estimates.  I'll use Chebyshev to choose 'good' projections upon which qualitative $L^{\infty}$ estimates do hold, estimate the traces of the 'bad' projections upon which they fail, then bound the covering numbers of the 'bad' projections, and 'transfer' their entropy into the (assumed) geometrically decaying spectrum of the derivative, $DF(X)$.   Exploiting $L^2$ estimates for the explicit form of the derivative of the polynomial function $F$ is crucial.  The residual set of 'bad' projections on which the uniform $L^{\infty}$ bounds fail is so small from a dimensional perspective that one can control it without introducing aymptotically coarse estimates.  This division of the operators into a 'good' part where $L^{\infty}$-bounds are available and 'bad' parts where they fail but for which an $L^2$-estimate is available is somewhat reminiscent of the proof of the Calderon-Zygmund decomposition. 

\subsection{Bindings and Fringes}

The results of this subsection will be phrased in the context of finite dimensional, real Hilbert spaces.  Throughout denote by $V$ and $W$ two finite dimensional, real Hilbert spaces, $L(V, W)$ the set of real linear operators from $V$ into $W$, and $P(V)$ the subset of $L(V,V)$ consisting of orthogonal projections.  

The key example to keep in mind is when $V$ and $W$ are of the form $(M_k(\mathbb C))^n$ for some $n$, and the real inner product is the one generated by the tracial state: $\langle \xi, \eta \rangle = \sum_{i=1}^n Re(tr_k(\eta_i^*\xi_i))$ where $\xi = (\xi_1,\ldots, \xi_n)$ and $\eta = (\eta_1,\ldots, \eta_n)$.  Notice that the real inner product norm here coincides with the standard complex inner product norm.  In this context, I want $\epsilon$-entropy estimates for level sets of a $*$-polynomial function restricted to a ball $B$ of radius $\rho$ computed w.r.t. the inner product.  It is impossible to get a global, dimension-free (ones which don't refer to the dimension of the ambient inner product space) Lipschitz bound for $f$ on $B$ since multiplication fails to be continuous w.r.t. the inner product ($L^2$)-norm.  However, such dimension-free estimates will work on almost the entire space.  Clarifying what 'almost' means is the goal of this subsection and is expressed through the concepts of a \textbf{binding} and its \textbf{fringe}. 

This short subsection will define bindings and fringes.  Their application to $*$-polynomials via Chebyshev's inequality will follow in the two subsequent sections.

Suppose $\Lambda \subset L(V, W)$ and fix a linear operator $T \in L(V, W)$.

\begin{definition} If $\epsilon >0$, an $\epsilon$-binding of $\Lambda$ focused on $T$ is a map $\Theta:\Lambda \rightarrow P(V)$ such that $\|(S-T)\Theta(S)\| < \epsilon$.  
\end{definition}


\begin{definition} Suppose $\Theta$ is an $\epsilon$-binding of $\Lambda$ focused on $T$.  The fringe of $\Theta$ on a set $K \subset V$ is 
\begin{eqnarray*}
\mathcal F(\Theta, K) =  \{(T - S)\Theta(S)^{\bot} \xi: \xi \in K, S \in \Lambda\} \subset W.
\end{eqnarray*}
If $r >0$, then $\mathcal F(\Theta, r)= \mathcal F(\Theta, B)$ where $B$ is the closed ball of radius $r$ centered at the origin.
\end{definition}

\begin{remark} Observe that if $K$ is symmetric, i.e., $x \in K$ iff $-x \in K$, then, so is $\mathcal F(\Theta, K)$ by linearity.  Balls being symmetric, $F(\Theta, r) = - F(\Theta, r)$ for any $r >0$.   Observe also that if $\Theta$ is an $\epsilon$-binding of $\Lambda$ focused on $T$ and $\Lambda_0 \subset \Lambda$, then $\Theta$ induces by restriction, an $\epsilon$-binding $\Theta_0$ of $\Lambda_0$ focused on $T$.  It is easy to see that for any set $K \subset V$, $\mathcal F(\Theta_0,K) \subset \mathcal F(\Theta, K)$.
\end{remark}

\begin{definition} If $E \subset K \subset V$ with $K$ an open, convex set, and $f: K \rightarrow W$ is a $C^1$-function, then for any $x, y \in K$ define the distance operator from $x$ to $y$ by
\begin{eqnarray*}
T_{x,y} = \int_0^1 Df(x+t(y-x)) \, dt \in L(V, W).
\end{eqnarray*}
The set of distance operators for $E$ with respect to $f$ is 
\begin{eqnarray*}
\mathcal D(f, E) = \{T_{x,y}: x, y \in E\} \subset L(V, W).
\end{eqnarray*}
\end{definition}

Recall from the mean value theorem (Section 2.5) that $T_{x,y}(x-y) = f(x)-f(y)$ which justifies the term, 'distance operator'.  

\begin{remark} Notice that if $E_1 \subset E_2$, then by definition, $\mathcal D(f, E_1) \subset \mathcal D(f, E_2)$.  In particular, if $\Theta$ is an $\epsilon$-binding of $\mathcal D(f, E_2)$ on an operator $T$, then by Remark 5.3 $\Theta$ induces by restriction an $\epsilon$-binding of $\mathcal D(f, E_1)$ on $T$.
\end{remark}

The following lemma says that the expansiveness of $f$ can be controlled locally by the fringe of an $\epsilon$-binding of the set of distance operators focused on a single derivative.

\begin{lemma} Suppose $K \subset V$ is an open, convex set, $f:K \rightarrow W$ is a $C^1$-function, $x_0 \in K$, and $\Lambda = \mathcal D(f, K)$.  If $\Theta$ is an $\epsilon$-binding of $\Lambda$ focused on $Df(x_0)$, then for each $x, y \in K$ there exists a corresponding operator $A: V \rightarrow W$ such that $\|A\| < \epsilon$ and 
\begin{eqnarray*} f(x) - f(y) - Df(x_0)(x-y) -A(x-y) \in \mathcal F(\Theta, \|x-y\|).
\end{eqnarray*}
\end{lemma}

\begin{proof}  Suppose $x, y \in K$.  Set $S =\int_0^1 Df(x + t(y-x)) \, dt$.  $S \in \Lambda$ so by definition of an $\epsilon$-binding, $\|(S-Df(x_0))\Theta(S)\| < \epsilon$.   By the mean value theorem,
\begin{eqnarray*} f(x) - f(y) & = & S(x-y) \\
                                            & = & Df(x_0)(x-y) + (S-Df(x_0))\Theta(S)(x-y)  + (S- Df(x_0))\Theta(S)^{\bot}(x-y) \\   
\end{eqnarray*}
\noindent Regrouping terms gives
\begin{eqnarray*} f(x) - f(y) - Df(x_0)(x-y) - (S-Df(x_0))\Theta(S)(x-y) & = & (S- Df(x_0))\Theta(S)^{\bot}(x-y)\\
                                                                                                            & \in & \mathcal F(\Theta, \|x-y\|).\\
\end{eqnarray*}
\noindent Set $A = (S-Df(x_0))\Theta(S)$.   By definition, $\|A\| < \epsilon$ and the above completes the proof.
\end{proof}

\begin{lemma} Suppose $x_0 \in E \subset B \subset V$ with $B$ an open ball of radius $\rho >0$, $f:B \rightarrow W$ is a $C^1$-function, and $\Lambda = \mathcal D(f, B)$.  If for all $x \in E$, $\|f(x)\| <\gamma$ and $\Theta$ is an $\rho_1$-binding of $\Lambda$ focused on $Df(x_0)$, then there exists an orthogonal operator $U$ such that for any spectral projection $Q$ of $|Df(x_0)|$,

\begin{eqnarray*} Df(x_0)Q(E- x_0) \subset UQU^* \left[ \mathcal N_{2 \rho \rho_1 + 2 \gamma}(\mathcal F(\Theta, 2 \rho)) \right ].
\end{eqnarray*} 
\end{lemma}

\begin{proof} By Lemma 5.6, for any $x \in E \subset B$ there exists a corresponding operator $A:V \rightarrow W$ such that $\|A\| < \rho_1$ and $f(x) - f(x_0) - Df(x_0)(x-x_0) - A(x-x_0) \in \mathcal F(\Theta,\|x-x_0\|)$.  Since $x, x_0 \in E$, $\|f(x)\|, \|f(x_0)\| < \gamma$ so that

\begin{eqnarray*} Df(x_0)(x-x_0) & \in & -A(x-x_0) - \mathcal N_{2\gamma}(\mathcal F(\Theta, \|x-x_0\|)) \\
                                                     & \in & \mathcal N_{2 \rho \rho_1 + 2 \gamma}(\mathcal F(\Theta, 2 \rho)).\\
\end{eqnarray*}

\noindent  Applying the polar decomposition shows 

\begin{eqnarray*} Df(x_0)Q(x-x_0) & = & U|Df(x_0)|Q(x-x_0) \\
                                                        & = & U Q|Df(x_0)|(x-x_0) \\
                                                        & = & UQU^* Df(x_0)(x-x_0).\\
                                                        \end{eqnarray*}

\noindent  Putting these two observations together,
\begin{eqnarray*}
Df(x_0)Q(x - x_0) & = & UQU^*Df(x_0)(x-x_0) \\
                             & \in & UQU^* \left[ \mathcal N_{2 \rho \rho_1 + 2 \gamma}(\mathcal F(\Theta, 2 \rho)) \right ]. \\
\end{eqnarray*}
$x \in E$ was arbitrary so I'm done.
\end{proof}

\begin{lemma} Suppose $x_0, E, B, f, \Lambda, \gamma, \Theta, \rho,$ and $\rho_1$ are as in the hypotheses of Lemma 5.7.  For $\alpha >0$ define $Q = 1_{[\alpha, \infty)}(|Df(x_0)|)$.  If $t \in (0,1)$, $\epsilon >0$, and $\beta = \alpha t \epsilon - 4 \rho \rho_1 - 4 \gamma > 0$, then 
\begin{eqnarray*} K_{\epsilon}(E) \leq K_{(1-t^2)^{1/2}\epsilon}(Q^{\bot}(E)) \cdot S_{\beta}(\mathcal F(\Theta, 2\rho)).
\end{eqnarray*}
\end{lemma}
\begin{proof}  By orthogonality and Proposition 2.1(v)
\begin{eqnarray*}
K_{\epsilon}(E) & \leq & K_{(1-t^2)^{1/2}\epsilon}(Q^{\bot}(E)) \cdot K_{t\epsilon}(Q(E)) \\
                         & \leq & K_{(1-t^2)^{1/2}\epsilon}(Q^{\bot}(E)) \cdot S_{t\epsilon}(Q(E)). \\
\end{eqnarray*}
\noindent   By the spectral theory, for any $\xi, \eta$ in the range of $Q$, $\|Df(x_0)(\xi)  -Df(x_0)(\eta)\| \geq \alpha \|\xi - \eta\|$.  It follows from this that $S_{t\epsilon}(Q(E)) \leq S_{\alpha t \epsilon}(Df(x_0)Q(E))$.  To complete the proof it suffices to show that $S_{\alpha t \epsilon}(Df(x_0)Q(E)) \leq S_{\beta}(\mathcal F(\Theta, 2\rho))$.

By Lemma 5.7, 
\begin{eqnarray*}
Df(x_0)(Q(E)) - Df(x_0)(Q(x_0)) & = & Df(x_0)(Q(E - x_0)) \\ & = & UQU^* \left[ \mathcal N_{2 \rho \rho_1 + 2 \gamma}(\mathcal F(\Theta, 2 \rho)) \right ]. \\
\end{eqnarray*}
From this, the fact that $UQU^*$ is a contraction, and Proposition 2.1(iv)
\begin{eqnarray*}
S_{\alpha t \epsilon}(Df(x_0)Q(E)) & \leq & S_{\alpha t \epsilon}(UQU^*  \left[ \mathcal N_{2 \rho \rho_1 + 2 \gamma}(\mathcal F(\Theta, 2 \rho)) \right ]) \\
                                                       & \leq & S_{\alpha t \epsilon} \left[ \mathcal N_{2 \rho \rho_1 + 2 \gamma}(\mathcal F(\Theta, 2 \rho)) \right ]  \\             & \leq & S_{\alpha t \epsilon - 4 \rho \rho_1-4 \gamma}(\mathcal F(\Theta, 2 \rho)) \\                                         & = & S_{\beta}(\mathcal F(\Theta, 2 \rho)). \\
\end{eqnarray*}
\end{proof}

\subsection{Estimates in Tracial von Neumann Algebras}

The estimates here will be used to construct bindings and fringes for maps that arise as finite tuples of $*$-polynomials.  Essentially they say that the $\|\cdot\|_2$ norm is quasi-submultiplicative, i.e., $\|xyp\|_{\infty} \leq C \|x\|_2 \|y\|_2$ on a projection $p$ with trace almost equal to $1$ and moreover, this can be controlled with $C$.   Again the context to keep in mind below is when $M=M_k(\mathbb C)$ (and in fact the only case I'll need), however I'll phrase the results in the tracial von Neumann algebra setting.

\begin{lemma} If $z \in M$ and $C >0$, then there exists a projection $p \in M$ such that $\|zp\|_{\infty} < C\|z\|_2$ and $\varphi(p) \geq 1- C^{-2}$.
\end{lemma}

\begin{proof} This is Chebyshev's inequality.  Set $p = 1_{[0, C\|z\|_2]}(|z|)$ and denote by $u$ the partial isometry in the polar decomposition of $z$.  $\|zp\|_{\infty} = \|u|z| p\|_{\infty} \leq  \| |z| p \|_{\infty} <  C\|z\|_2$ which yields the first inequality.  For the second,  
\begin{eqnarray*} 
0 & \leq & (C\|z\|_2)^2 \cdot p^{\bot} \\
                                & \leq & |z|^2 p^{\bot} \\
                                 & \leq & |z|^2.\\
\end{eqnarray*}
Taking traces yields $(C \|z\|_2)^2 \cdot \varphi(p^{\bot}) \leq \varphi(|z|^2) < \|z\|_2^2$.  Grouping terms, $1 - \varphi(p) = \varphi(p^{\bot}) \leq C^{-2}$ and the second inequality follows.
\end{proof}

\begin{lemma} If $x, q \in M$ with $qH \subset cl(xH)$, then there exists a projection $e \in M$ such that $eH \subset (\ker x)^{\bot}$ satisfying $xeH \subset qH$ and $\varphi(e) = \varphi(q)$.
\end{lemma}

\begin{proof} I'll prove this first in the case where $x \geq 0$ and then use the polar decomposition to arrive at the general claim.  So assume $x \geq 0$ and $qH \subset cl(xH)$.  $f = 1_{(0,\infty)}(x)$ is the projection onto $cl(xH)$.  Define for each $n$, $f_n = 1_{(1/n, \infty)}(x)$, $f_0 = 1_{\{0\}}(x)$, and note that $f_0$ is the projection onto $\ker x$.   For each $n$, note that $q\wedge f_n \leq f_n$, $f_n$ commutes with $x$, and that $x f_n = f_n x f_n$ is invertible when regarded as an element of $f_nMf_n$.  It follows that there exists a projection $e_n \leq f_n$ such that $x e_n H = (q \wedge f_n)H$ and $\varphi(e_n) = \varphi(q \wedge f_n)$.  Indeed, this is obtained by taking the projection onto the range of $(f_n xf_n)^{-1} (q \wedge f_n)$ where the inverse is taken w.r.t. the compression $f_nMf_n$.  Notice also from this definition and the Borel spectral theorem that for any $n$, $e_n \leq e_{n+1}$.

There exists a sequence of increasing projections $\langle e_n \rangle_{n=1}^{\infty}$ such that for each $n$, $\varphi(e_n) = \varphi(q \wedge f_n)$, and $xe_n H \subset (q \wedge f_n)H$.  Denote by $e$ the strong operator topology limit of the $e_n$ (or equivalently, the projection onto the closure of $\cup_{n=1}^{\infty} e_nH$).  Because $e_n \leq f_n \leq f$, $e \leq f$ where $f$ is the projection onto $cl(xH) = (\ker x)^{\bot}$ ($x \geq 0$). For any $n$,
\begin{eqnarray*}
xe_nH & \subset & (q \wedge f_n)H \\
            & \subset  & qH. \\
\end{eqnarray*}
\noindent It follows that $xeH \subset qH$.  Moreover, since $\lim_{n \rightarrow \infty} f_n = f$ (the projection onto $cl(xH)$) and $q\leq f$ by hypothesis,
\begin{eqnarray*}
\varphi(e) & = & \lim_{n \rightarrow \infty} \varphi(e_n) \\
                & = & \lim_{n \rightarrow \infty} \varphi(q \wedge f_n) \\
                & = & \lim_{n \rightarrow \infty} \varphi(q \wedge f) \\
                & = & \varphi(q).
\end{eqnarray*}
\noindent This finishes the proof in the case where $x$ is positive.

For the general case, suppose $x, q \in M$ are given with $qH \subset cl(xH)$.  Consider the polar decomposition $x^* = u|x^*|$.  Thus, $u \in M$ is a partial isometry with initial range $cl(|x^*|H) = cl(xH)$ and final range $cl(x^*H) = (\ker x)^{\bot}$.  Apply the preceding result to the positive element $|x^*|$ and $q$, noting that $qH \subset cl(xH) = cl(|x^*|H)$.  There exists a projection $p \in M$ such that $pH \subset (\ker |x^*|)^{\bot} = (\ker x^*)^{\bot}$, $|x^*|pH \subset qH$, and $\varphi(p) = \varphi(q)$.   Set $e = upu^* \in M$.  $eH \subset (\ker x)^{\bot}$ and
\begin{eqnarray*}
xeH & = & |x^*|u^*(upu^*)H \\
       & = & |x^*| p u^* H\\
       & = & |x^*| p H \\
       & \subset & qH.\\ 
\end{eqnarray*}
\noindent $\varphi(e) = \varphi(upu^*) = \varphi(p) = \varphi(q)$, completing the proof.
\end{proof}

\begin{lemma} Suppose $x, q \in M$ with $q$ a projection.  There exists a projection $p \in M$ such that $\varphi(p) \geq\varphi(q)$ and $xpH \subset qH$. 
\end{lemma}

\begin{proof} Denote by $f$ the projection onto $cl(xH)$.  From the polar decomposition, if $e$ is the projection onto $\ker x$, then $\varphi(e) + \varphi(f) =1$.  Thus, $\varphi(e) = 1 -\varphi(f) = \varphi(f^{\bot})$.  Now
\begin{eqnarray*}
1 & \geq & \varphi(q \vee f) \\
   & = & \varphi(q) + \varphi(f) - \varphi(q\wedge f). \\
\end{eqnarray*}
\noindent Thus, $\varphi(e) + \varphi(q\wedge f) = 1 -\varphi(f) + \varphi(q \wedge f) \geq \varphi(q)$.

Obviously $q \wedge f \leq f$.  Invoke Lemma 5.10 to produce a projection $p_1 \in M$ such that $p_1H \subset (\ker x)^{\bot}$, $xp_1H \subset (q \wedge f)H$, and $\varphi(p_1) = \varphi(q \wedge f)$.  Set $p = e + p_1$.  Since $e$ and $p_1$ are orthogonal, 
\begin{eqnarray*}
\varphi(p) & = & \varphi(e) + \varphi(p_1) \\
                & = & \varphi(e) + \varphi(q \wedge f) \\
                & \geq & \varphi(q).\\
\end{eqnarray*}
\noindent Lastly, $xpH = xp_1H \subset (q \wedge f)H \subset qH$.
\end{proof}

\begin{lemma} If $z_1, \ldots, z_n \in M$ and $C>0$, then there exists a projection $p \in M$ such that $\|z_1 \cdots z_np\|_{\infty} < C^n \|z_1\|_2 \cdots \|z_n\|_2$ and $\varphi(p) \geq 1 - nC^{-2}$.
\end{lemma}

\begin{proof}  I will prove this by induction on $n$. The base case $n=1$ is covered in Lemma 5.9.  Assume it's true for $n=k$.  Suppose $z_1, \ldots z_{k+1} \in M$.  By the inductive hypothesis there exists a projection $p_0 \in M$ such that $\|z_2 \cdots z_{k+1} p_0\|_{\infty} < C^k \|z_2\|_2 \cdots \|z_{k+1}\|_2$ and $\varphi(p_0) > 1- kC^{-2}$.   By Lemma 5.9 there exists a projection $p_1$ such that $\|z_1 p_1\|_{\infty} < C \|z_1\|_2$ and $\varphi(p_1) > 1 -C^{-2}$.  By Lemma 5.11 there exists a projection $q \in M$ such that $\varphi(q) \geq \varphi(p_1)$ and $(z_2 \cdots z_{k+1}) qH \subset p_1H$.  Set $p = q \wedge p_0$.  Clearly $p = qp = p_0 p $.  
\begin{eqnarray*}
\|z_1 \cdots z_{k+1} p \|_{\infty} & = & \|z_1 (z_2 \cdots z_{k+1} q p)\|_{\infty}  \\
                                           & = & \|z_1 p_1(z_2 \cdots z_{k+1} q p)\|_{\infty}  \\
                                           & \leq & \|z_1 p_1\|_{\infty} \cdot  \|p_1 (z_2 \cdots z_{k+1}qp)\|_{\infty}  \\
                                           & \leq & C \|z_1\|_2 \cdot \|z_2 \cdots z_{k+1}p_0 p\|_{\infty}  \\
                                           & \leq & C^{k+1} \cdot \|z_1\|_2 \cdots \|z_{k+1}\|_2.\\
\end{eqnarray*}
\noindent Also, $\varphi(p) = \varphi(q \wedge p_0) = \varphi(p_0) + \varphi(q) - \varphi(p_0\vee q) > 1 - (k+1)C^{-2}$.  This verifies the condition for $n=k+1$ and completes the proof.
\end{proof}

By taking adjoints one immediately gets:

\begin{corollary} If $z_1, \ldots, z_n \in M$ and $C>0$, then there exists a projection $p \in M$ such that $\|p z_1 \cdots z_n\|_{\infty} < C^n \|z_1\|_2 \cdots \|z_n\|_2$ and $\varphi(p) \geq 1 - nC^{-2}$.
\end{corollary}

Combining Corollary 5.13 with Lemma 5.12 yields:

\begin{corollary} If $x_1, \ldots, x_k, y_1,\ldots, y_n \in M$ and $C>0$, then there exists a projection $p \in M$ such that 
$\|x_1 \ldots x_k p\|_{\infty} < C^{k} \|x_1\|_2 \cdots \|x_k\|_2$, $\|p y_1 \cdots y_n\|_{\infty} < C^n \|y_1\|_2 \cdots \|y_n\|_2$, and $\varphi(p) \geq 1 - (k+n) C^{-2}$.
\end{corollary}

\subsection{Polynomials, derivatives, and coverings for matrices}

In this last subsection I'll construct bindings for the derivative of a tuple of $*$-polynomials and estimate its fringe entropy.  The estimates will invoke the results in subsection 5.2 while the bindings will be created from the projections in subsection 5.3 upon which quasi-submultiplicative estimates are valid.

As in Section 3, $\mathfrak{A}_n$ denotes the universal, unital, complex $*$-algebra on $n$-indeterminates, $X_1,\ldots, X_n$.  Unless otherwise stated, $F \in \mathfrak{A}_n$, i.e., $F$ is a noncommutative $*$-polynomial in $n$-indeterminates.  $F$ can uniquely be written in the reduced form $c_1w_1 + \cdots +c_nw_n$ where the $c_i$ are nonzero complex numbers and the $w_i$ are distinct $*$-monomials in the $X_i$.  Define $c(F)$ to be the maximum over the set $\{|c_i|, \ell(w_i): 1 \leq i \leq n\} \cup \{n\}$ where $\ell(\cdot)$ is the length function defined on the $*$-monomials of $\mathfrak{A}_n$.  As usual, $\deg(F) = \max_{1\leq i \leq n} \ell(w_i)$.  When $F=\{f_1,\ldots, f_p\}$ is a finite $p$-tuple of elements in $\mathfrak{A}_n$, then $c(F) = \max_{1\leq i \leq p} c(f_i)$ and $\deg(F) = \max_{1 \leq i \leq n} \deg(f_i)$. 

For $k$ fixed $F$ induces an obvious smooth map from $(M_k(\mathbb C))^n$ into $M_k(\mathbb C)$.  In this subsection for a fixed $\xi \in (M_k(\mathbb C))^n$, $DF(\xi): (M_k(\mathbb C))^n \rightarrow M_k(\mathbb C)$ denotes the derivative of $F$ at $\xi$ and for each $1 \leq i \leq n$ $\partial_i F(\xi): M_k(\mathbb C) \rightarrow M_k(\mathbb C)$ denotes the $i$th partial derivative of $F$ w.r.t. the ith coordinate of $\xi = (\xi_1,\ldots, \xi_n)$ (Section 2.5).  Denote by $B_{k, R}$ the ball of $\|\cdot\|_2$-radius $R$ in $M_k(\mathbb C)$ and $(B_{k,R})^n \subset (M_k(\mathbb C))^n$ the direct sum of $n$ copies of $B_{k,R}$.  $\sigma_k: M_k(\mathbb C) \otimes M_k(\mathbb C)^{op} \rightarrow L(M_k(\mathbb C))$ is the complex, trace preserving $*$-isomorphism determined by $\sigma_k(a \otimes b^{op})(\xi) = a\xi b$, $a,b, \xi \in M_k(\mathbb C)$.  Throughout this subsection $i$ and $j$ denote integer indices.  

\begin{lemma} Suppose $1 \leq j \leq n$.  There exist $N \in \mathbb N$ and for $1 \leq i \leq N$, $\lambda_i \in \mathbb C$, and $*$-monomials $w_i, v_i \in \mathfrak{A}_{2n}$ all dependent only on $F$ and $j$ such that for each $i$ either $w_i$ or $v_i$ has an occurence of the last $n$ of the $2n$-indeterminate generators for $\mathfrak{A}_{2n}$ and for any $X, H \in (M_k(\mathbb C))^n$,
\begin{eqnarray*} (\partial_jF)(X+H) = (\partial_jF)(X) + \sum_{i=1}^N \lambda_i \cdot \sigma_k(w_i(X,H) \otimes v_i(X,H)^{op}) \circ J^{d_i}.
\end{eqnarray*}
Moreover, $N$ and $\max_{1\leq i \leq N} |\lambda_i|$ are no greater than $c(F)\cdot(2n)^{c(F)}$ and $\max_{1\leq i \leq N} (\ell(w_i)+\ell(v_i)) < \deg(F)$.
\end{lemma}
\begin{proof} By writing $F$ in reduced form there exist an $N \in \mathbb N$ and for each $1 \leq i \leq N$, $*$-monomials $a_i, b_i \in \mathfrak{A}_n$, complex numbers $\lambda_i$, and $d_i \in \{0,1\}$ depending on $F$ and $j$ such that for any $Y \in (M_k(\mathbb C))^n$,
\begin{eqnarray*}
\partial_jF(Y) = \sum_{i=1}^N \lambda_i \sigma_k(a_i(Y) \otimes b_i(Y)^{op}) \circ J^{d_i}
\end{eqnarray*}
where $\max_{1 \leq i \leq N} (\ell(a_i) + \ell(b_i)) < \deg(F)$.  Substituting $Y=X+H$ and $Y=X$ into the above equation and subtracting yields:
\begin{eqnarray*}
\partial_jF(X+H) - \partial_jF(X) = \sum_{i=1}^N \lambda_i \cdot \sigma_k[a_i(X+H) \otimes b_i(X+H)^{op} - a_i(X) \otimes b_i(X)^{op}] \circ J^{d_i}.\\
\end{eqnarray*}
Each bracketed summand on the RHS can be expanded as a further sum of terms of the form $w_i(X,H) \otimes v_i(X,H)$ where $w_i,v_i \in \mathfrak{A}_{2n}$ and either $w_i$ or $v_i$ has an occurence of the last $n$ of the $2n$-indeterminates, i.e., that one of the $H$ terms appears in $w_i(X,H) \otimes v_i(X,H)$.  Moreover, $\deg(w_i) + \deg(v_i) \leq \deg(a_i)+\deg(b_i)$.  This establishes the equation.   

The last statement is a consequence of the fact that the $\partial_jF(Y)$ is obtained from the reduced form of $F$ and the expansion of the perturbed elementary tensors $a_i(X+H) \otimes b_i(X+H)^{op}$ can be written as a sum of no more than $(2n)^{c(F)}$ elementary tensors of the form $w_i(X,H) \otimes v_i(X,H)$ with $\deg(w_i) + \deg(v_i) \leq \deg(a_i)+\deg(b_i) < \deg(F)$. 
\end{proof}

\begin{lemma} For each $1 \leq j \leq n$ there exist an $N_0, N \in \mathbb N$ and for $1 \leq i \leq N$, $c_i \in \mathbb C$, $s_i \in \{1,*\}$, $1 \leq q_i \leq n$, $*$-monomials $a_i, b_i \in \mathfrak{A}_{2n}$, and $*$-monomials $m_i \in \mathfrak{A}_n$ such that for any $n$-tuples $X$ and $H=(h_1,\ldots, h_n)$ in $(M_k(\mathbb C))^n$, if $T = \int_0^1 \partial_j F(X+tH) \, dt$, then for any $\xi \in M_k(\mathbb C)$
\begin{eqnarray*}
(T-\partial_j F(X))\xi = \sum_{i=1}^{N_0} c_i \cdot a_i(X,H)  h_{q_i}  m_i(X) \xi^{s_i} b_i(X,H) + \sum_{i=N_0+1}^N c_i \cdot a_i(X,H) \xi^{s_i} m_i(X) h_{q_i} b_i(X,H).
\end{eqnarray*}
Moreover, $N$ and  $\max_{1\leq i \leq N} |c_i|$ are no greater than $c(F)\cdot(2n)^{c(F)}$ and $\max_{1\leq i \leq N} (\ell(a_i)+\ell(b_i) + \ell(m_i)+1) < \deg(F)$.
\end{lemma}

\begin{proof}  Invoke Lemma 5.15 to produce $N, \lambda_1,\ldots, \lambda_N \in \mathbb C$, and the $*$-monomials $w_1,\ldots, w_N, v_1, \ldots, v_N$ as in its statement.  Denote by $n_i$ the sum of the exponents arising from the last $n$ variables and their adjoints which appears in $w_i(X,H)$ and $v_i(X,H)$.  Set $c_i = \lambda_i \int_0^1 t^{n_i} \, dt$.  Compute:
\begin{eqnarray*}
T - \partial_jF(X) & = & \int_0^1 \partial_jF(X + tH) - \partial_jF(X) \, dt \\
                           & = & \int_0^1 \sum_{i=1}^N \lambda_i \cdot \sigma_k(w_i(X, tH) \otimes v_i(X,tH)^{op}) \circ J^{d_i} \, dt \\
                           & = & \sum_{i=1}^N \left(\lambda_i \cdot \sigma_k(w_i(X,H) \otimes v_i(X, H)^{op}) \circ J^{d_i}  \cdot \int_0^1 t^{n_i} \, dt \right )\\
                            & = & \sum_{i=1}^N \left(c_i \cdot \sigma_k(w_i(X,H) \otimes v_i(X, H)^{op}) \circ J^{d_i} \right ). \\
\end{eqnarray*}
Define $s_i =1$ if $d_i =0$ and $s_i = *$ if $d_i =1$.  Substituting $\xi \in M_k(\mathbb C)$ in the equation above produces
\begin{eqnarray*}
(T - \partial_jF(X))\xi & = & \sum_{i=1}^N c_i \cdot w_i(X,H) \xi^{s_i} v_i(X, H). \\
\end{eqnarray*}
Now for each $i$, either $w_i$ or $v_i$ has a nontrivial occurrence of one of the last $n$ of the $2n$-indeterminates.  In the first case there exists some $n+1 \leq i_j \leq 2n$ such that $w_i = a_j X_{i_j}  m_j$ where $a_j \in \mathfrak{A}_{2n}$,  and $m_j$ is a word in $X_1,\ldots, X_n \in \mathfrak{A}_{2n}$ (the first $n$ indeterminates).  In the latter case there exists some $n+1 \leq i_j \leq 2n$ such that $v_i = m_j X_{i_j} b_j$ where $b_j \in \mathfrak{A}_{2n}$ and $m_j$ is a word in $X_1,\ldots X_n \in \mathfrak{A}_{2n}$ (the first $n$ indeterminates).  Evaluating either of these expression at the $2n$-tuple $(X,H)$ and regrouping indices establishes the equation.

The last statement follows from the last statement of Lemma 5.15 and the fact that $|c_i| \leq |\lambda_i|$.
\end{proof}

\begin{remark}  Lemma 5.16 has a similar statement where one replaces the mean value integral operator by just the point derivative, i.e., with the same notation as in Lemma 5.15 there exist constants $c_1, \ldots, c_N$ such that for any $\xi \in M_k(\mathbb C)$,
\begin{eqnarray*}
(\partial_jF(X+H)-\partial_j F(X))\xi & = & \sum_{i=1}^{N_0} c_i \cdot a_i(X,H)  h_{q_i}  m_i(X) \xi^{s_i} b_i(X,H) + \\ &&\sum_{i=N_0+1}^N c_i\cdot a_i(X,H) \xi^{s_i} m_i(X) h_{q_i} b_i(X,H).\\
\end{eqnarray*}
This statement is just a slight reformulation of Lemma 5.15 (as one doesn't need to integrate out any constants as in Lemma 5.16).
\end{remark}

Rephrasing Lemma 5.16 and Remark 5.17 with $H = X-Y$ gives the following:

\begin{corollary} For $1 \leq j \leq n$ there exists $N \in \mathbb N$ and for $1 \leq i \leq N$, $c_i \in \mathbb C$, $s_i \in \{1,*\}$ dependent only on $F$ such that if $X, Y \in (B_{k,R})^n$, then there exist $a_i, b_i,$ and $m_i$ each products of $r_i, d_i,$ and $n_i$ elements from $B_{k,R}$ with $r_i+ d_i + n_i+1 < \deg(F)$ such that if $T = \int_0^1 \partial_j F(X+t(Y-X))\, dt$ or $T = \partial F(Y)$, then for any $\xi \in M_k(\mathbb C)$,
\begin{eqnarray*}
(T-\partial_j F(X))\xi = \sum_{i=1}^{N_0} c_i a_i (x_{q_i} - y_{q_i}) m_i \xi^{s_i} b_i + \sum_{i=N_0+1}^N c_i a_i \xi^{s_i} m_i (x_{q_i}-y_{q_i}) b_i.
\end{eqnarray*}
Moreover, $N$ and $\max_{1 \leq i \leq N} |c_i|$ are no greater than $c(F)\cdot(2n)^{c(F)}$.
\end{corollary}

\begin{lemma} Suppose $1 \leq j \leq n$, $R\geq1$, $B = c(F)^2 \cdot(2n)^{2c(F)}$ and $X, Y \in (B_{k,R})^n$.  If $\|X-Y\|_2 < \epsilon<1$, and either $T = \int_0^1 \partial_j F(X+t(Y-X))\, dt$ or $T = \partial_j F(Y)$, then there exists a projection $p \in M_k(\mathbb C)$ such that $\varphi(p) > 1 - B \epsilon^{\frac{1}{B}}$ and for any $\xi \in M_k(\mathbb C)$,
\begin{eqnarray*}
\|(T-\partial_j F(X))p\xi p\|_2 & \leq & \left (B \cdot R^{\deg(F)} \cdot \epsilon^{1/2} \right) \cdot \|p\xi p\|_2. \\
\end{eqnarray*}
\end{lemma}
\begin{proof} For $1 \leq j \leq n$ invoke Corollary 5.18 to produce the $N \in \mathbb N$ and $c_i \in \mathbb C$, $s_i \in \{1,*\}$ dependent on $F$ such that if $X, Y \in (B_{k,R})^n$, then there exists the corresponding $c_i, a_i, b_i, m_i, q_i, r_i, d_i,$ and $n_i$ as in the conclusion of the corollary.  Recall the summation expression of $(T-\partial_jF(X))\xi$:
\begin{eqnarray*}
(T-\partial_j F(X))\xi = \sum_{i=1}^{N_0} c_i a_i (x_{q_i} - y_{q_i}) m_i \xi^{s_i} b_i + \sum_{i=N_0+1}^N c_i a_i \xi^{s_i} m_i (x_{q_i}-y_{q_i}) b_i.
\end{eqnarray*}
When $1 \leq i \leq N_0$, applying Corollary 5.14 to $M=M_k(\mathbb C)$ and $C= \epsilon^{-\frac{1}{4\deg(F)}}>1$ provides a (complex linear) orthogonal projection $p_i$ such that 
\begin{eqnarray*}
\|a_i (x_{q_i} - y_{q_i}) m_ip_i\|_{\infty} & \leq & (CR)^{r_i+n_i+1} \cdot \epsilon \\
                                                              & \leq & R^{r_i+n_i+1} \cdot \epsilon^{3/4}, \\
\end{eqnarray*}
$\|p_ib_i\|_{\infty} < (CR)^{d_i} \leq R^{d_i} \cdot \epsilon^{-\frac{1}{4}}$, and $tr_k(p_i) > 1 - \deg(F) \cdot \epsilon^{\frac{1}{2 \deg(F)}}$.  Thus, for any $\xi \in M_k(\mathbb C)$ and (complex linear) orthogonal projection $e \leq p_i$

\begin{eqnarray*}
\|a_i(x_{q_i} - y_{q_i})m_i (e\xi^{d_i}e) b_i\|_2 & \leq & \|a_i(x_{q_i} - y_{q_i})m_i e\|_{\infty} \cdot \|e \xi^{d_i}e\|_2 \cdot \|eb_i\|_{\infty} \\ 
                                                                       & \leq & \|a_i(x_{q_i} - y_{q_i})m_i p_i \|_{\infty} \cdot \|e \xi^{d_i}e\|_2 \cdot \|p_i b_i\|_{\infty} \\
                                                                           & \leq & R^{r_i+n_i+1} \cdot \epsilon^{3/4} \cdot R^{d_i} \cdot \epsilon^{-\frac{1}{4}}\cdot \|e \xi e\|_2 \\
                                                                           & \leq & R^{\deg(F)} \cdot \epsilon^{1/2} \cdot \|e \xi e\|_2.\\
\end{eqnarray*}
For $N_0 \leq i \leq N$ an analogous argument yields a projection $p_i$ such that $tr_k(p_i) > 1 - \deg(F) \cdot \epsilon^{\frac{1}{2 \deg(F)}}$ and for any projection $e \leq p_i$, $\|a_i (e\xi^{d_i}e) m_i (x_{q_i} - y_{q_i})b_i\|_2 \leq R^{\deg(F)} \cdot \epsilon^{1/2} \cdot \|e\xi e\|_2$.

Define $p = \wedge_{i=1}^N p_i$.  Generously majorizing and using the fact that $c(F)^2 < B$,
\begin{eqnarray*}
tr_k(p)  & > & 1 - \sum_{i=1}^N tr_k(p_i^{\bot}) \\
            & = & 1 - N \deg(F) \epsilon^{\frac{1}{2 \deg(F)}} \\
            & > & 1 - B \epsilon^{\frac{1}{B}}. \\
\end{eqnarray*}
Moreover, since $p \leq p_i$ for each $i$, it follows from the above that for $1 \leq i \leq N_0$,
\begin{eqnarray*}
\|a_i(x_{q_i} - y_{q_i})b_i (p\xi^{s_i}p) m_i\|_2 & \leq & R^{\deg(F)} \cdot \epsilon^{1/2} \cdot \|p\xi p\|_2\\
\end{eqnarray*}
and for $N_0 < i \leq N$,
\begin{eqnarray*}
\|a_i (p\xi^{s_i}p) b_i (x_{q_i} - y_{q_i})m_i\|_2 & \leq & R^{\deg(F)} \cdot \epsilon^{1/2} \cdot \|p\xi p\|_2.\\
\end{eqnarray*}
Thus, using the bound on $N$ and the $|c_i|$ provided in Corollary 5.18
\begin{eqnarray*}
\|(T-\partial_j F(X))(p\xi p)\|_2 & \leq & \sum_{i=1}^{N_0} |c_i| \cdot  \|a_i (x_{q_i} - y_{q_i}) m_i (p\xi^{s_i} p) b_i \|_2 + \\ &  & \sum_{i=N_0+1}^N |c_i| \cdot \|a_i (p \xi^{s_i} p) m_i (x_{q_i}-y_{q_i}) b_i\|_2 \\ 
& \leq & (c(F) \cdot(2n)^{c(F)})^2 \cdot R^{\deg(F)} \cdot \epsilon^{1/2} \cdot \|p\xi p\|_2 \\
&= & (B \cdot R^{\deg(F)} \cdot \epsilon^{1/2}) \cdot \|p\xi p\|_2.\\
\end{eqnarray*}
\end{proof}

Applying Lemma 5.19 to each partial derivative and taking the intersection of the associated projections yields: 

\begin{corollary} Suppose $R>1$, $X, Y \in (B_{k,R})^n$ and $B=n \cdot c(F)^2 \cdot(2n)^{2c(F)}$.  If  $\|X-Y\|_2 < \epsilon <1$, and either $T = \int_0^1 DF(X+t(Y-X))\, dt$ or $T=DF(Y)$, then there exists a projection $p \in M_k(\mathbb C)$ such that $\varphi(p) > 1 - B\epsilon^{\frac{1}{B}}$ and if $P = \oplus_{i=1}^n p$, then for any $\xi \in (M_k(\mathbb C))^n$,
\begin{eqnarray*}
\|(T-DF(X))P\xi P\|_2 & \leq & \left (B \cdot R^{\deg(F)} \cdot \epsilon^{1/2} \right) \cdot \|P\xi P\|_2. \\
\end{eqnarray*}
\end{corollary}

Recall in subsection 5.2 that for a convex set $K \subset (M_k(\mathbb C))^n$,  the set of distance operators for $K$ with respect to $f$, $\mathcal D(F,K)$, is the set of all operators of the form
\begin{eqnarray*}
\int_0^1 DF(x+t(y-x)) \, dt
\end{eqnarray*}
where $x, y \in K$.  Couching Corollary 5.20 in the terminology of subsection 5.2 gives the following:

\begin{lemma} Suppose $B=2n \cdot c(F)^2 \cdot(2n)^{2c(F)}$ and $R>1>\rho >0$.  If $\xi_0 \in (B_{k,R})^n$, and $K = B_2(\xi_0, \rho)$, then $\mathcal D(F, K)$ has a $(BR^{\deg(F)}\rho^{1/2})$-binding $\Theta$ focused on $DF(\xi_0)$ with the property that for any $T \in \mathcal D(F,K)$, $\Theta(T) = (\sigma_k(e \otimes e^{op}),\ldots, \sigma_k(e \otimes e^{op}))$ where $e \in M_k(\mathbb C)$ is an orthogonal projection satisfying $tr_k(e) > 1-B \rho^{\frac{1}{B}}$.
\end{lemma}

\begin{proof}   Suppose $T \in \mathcal D(F,K)$.   For some $\xi, \eta \in K = B_2(\xi_0, \rho)$, $T= T_{\xi, \eta} = \int_0^1 DF(\xi + t(\eta-\xi)) \, dt$.  By Corollary 5.20 there exist projections $p, q \in M_k(\mathbb C)$ such that $\varphi(p), \varphi(q) > 1-(B\rho^{\frac{1}{B}})/2$ and if $P = \oplus_{i=1}^n p$ and $Q= \oplus_{i=1}^n q$, then for any $\xi \in (M_k(\mathbb C))^n$,
\begin{eqnarray*}
\|(T - DF(\xi))P\xi P\|_2 \leq  \left (B/2 \cdot R^{\deg(F)} \cdot \rho^{1/2} \right) \cdot \|P\xi P\|_2 \\
\end{eqnarray*}
and
\begin{eqnarray*}
\|(DF(\xi) - DF(\xi_0))Q\xi Q\|_2 \leq \left (B/2 \cdot R^{\deg(F)} \cdot \rho^{1/2} \right) \cdot \|Q\xi Q\|_2. \\
\end{eqnarray*}
Set $e = p \wedge q$.  $tr_k(e) > 1 - B\rho^{\frac{1}{B}}$.  Putting $E = \oplus_{i=1}^n e$ and $\Theta(T) = \sigma_k(E \otimes E)$ the triangle inequality and the two inequalities above show that for any $\xi \in (M_k(\mathbb C))^n$,
\begin{eqnarray*}
                   \|(T - DF(\xi_0))\Theta(T)(\xi)\|_2 & = & \|(T - DF(\xi_0))E\xi E\|_2 \\
                                         & \leq & \|(T - DF(\xi))E\xi E \|_2 +  \|(DF(\xi) - DF(\xi_0))E\xi E\|_2  \\ 
                                         & = & \|(T - DF(\xi))PE\xi EP \|_2 +  \|(DF(\xi) - DF(\xi_0))QE\xi EQ \|_2  \\
                                         & \leq & \left( B \cdot R^{\deg(F)} \cdot \rho^{1/2} \right) \cdot \|E\xi E\|_2 \\
                                         & \leq & \left( B \cdot R^{\deg(F)} \cdot \rho^{1/2} \right) \cdot \|\xi\|_2. \\
\end{eqnarray*}
The above holds for any $T \in \mathcal D(F,K)$ and yields a map $\Theta: \mathcal D(F,K) \rightarrow P(\oplus_{i=1}^n M_k(\mathbb C))$ such that 
\begin{eqnarray*} 
\|(T - DF(\xi_0))\Theta(T)\| \leq (B \cdot R^{\deg(F)} \cdot \rho^{1/2}).  
\end{eqnarray*}
By definition, $\Theta$ is a $(BR^{\deg(F)}\rho^{1/2})$-binding of $\mathcal D(F,K)$ focused on $DF(\xi_0)$.  The statement about the projectional form of $\Theta$ is immediate.
\end{proof}

Extending this to a $p$-tuple of elements is easily done and involves straightforward manipulations of the multivariable derivative formalism in Section 2.5:

\begin{corollary} Suppose $F = (f_1,\ldots, f_p) \in (\mathfrak{A}_n)^p$, $B=2np \cdot c(F)^2 \cdot(2n)^{2c(F)}$ and $R>1>\rho >0$.  If $\xi_0 \in (B_{k,R})^n$, and $K = B_2(\xi_0, \rho)$, then $\mathcal D(F, K)$ has a $(BR^{\deg(F)}\rho^{1/2})$-binding $\Theta$ focused on $DF(\xi_0)$ such that for any $T \in \mathcal D(F,K)$, $\Theta(T) = (\sigma_k(p \otimes p^{op}),\ldots, \sigma_k(p \otimes p^{op}))$ and $p \in M_k(\mathbb C)$ is an orthogonal projection with $tr_k(p) > 1-B \rho^{\frac{1}{B}}$.
\end{corollary}

Lemma 5.21 and Corollary 5.22 construct bindings for the distance operators in a ball, focused on the derivative evaluated at the center of the ball.  It remains to look at the fringe of the binding and establish the appropriate entropy estimates on it.

\begin{definition} If $R, r >0$ and $d \in \mathbb N$, then $E(R,r, k,d) \subset M_k(\mathbb C)$ consists of all $k \times k$ matrices $x$ of rank no more than $d$ such that $\|x\|_r \leq R$.
\end{definition}

\begin{lemma} Suppose $B = c(F) \cdot (2n)^{c(F)}$. There exists an $N_1 \in \mathbb N$, $N_1 \leq B+1$ such that if $R >0$, $x,y \in (B_{k, R})^n$, $P = (\sigma_k(p \otimes p^{op}),\ldots, \sigma_k(p \otimes p^{op}))$ for a projection $p \in M_k(\mathbb C)$, and $T = \int_0^1 DF(x+t(y-x)) \, dt$, then 
 \begin{eqnarray*}
 (T - DF(x))(P^{\bot}(B_{k,R})^n) \subset \boxplus_{i=1}^{nN_1} E\left(N_1 \cdot \max(R,R^{N_1}), 2/N_1, k, 2k(1-tr_k(p)) \right).
 \end{eqnarray*}
\end{lemma}
\begin{proof} It suffices to produce an $N_1$ such that for each $1 \leq j \leq n$,
\begin{eqnarray*}
(T- \partial_jF(x))(\sigma_k(p \otimes p^{op})^{\bot}(B_{k,R})) \in \boxplus_{i=1}^{N_1} E(N_1 \cdot \max(R,R^{N_1}), 2/N_1,k, 2k(1-tr_k(p)))
\end{eqnarray*} 
where $T=\int_0^1 \partial_jF(x + t(y-x)) \, dt$.  By Corollary 5.18  there exist $N \in \mathbb N$ and $c_1,\ldots, c_N \in \mathbb \mathbb N$ dependent only on $F$ and $a_j,b_j,$ and $m_j$ where for each $j$, $a_j$, $b_j$, and $m_j$ are each products of $r_j, d_j,$ and $n_j$ elements from $B_{k,R}$ with $r_j+ d_j + n_j$ no greater than the degree of $F$ and such that for any $\xi \in M_k(\mathbb C)$,
\begin{eqnarray*}
(T-\partial_j F(x))\xi = \sum_{i=1}^{N_0} c_i a_i (x_{q_i} - y_{q_j}) m_i \xi^{s_i} b_i + \sum_{i=N_0+1}^N c_i a_i \xi^{s_i} m_i (x_{q_i}-y_{q_i}) b_i.
\end{eqnarray*} 
Moreover, $N$ and $C= \max_{1 \leq i \leq N} |c_i|$ are no greater than $B$. Thus, I can find an $N_1$ so that $\max\{C, N\} < N_1 \leq B+1$.  $\sigma_k(p \otimes p^{op})^{\bot} = \sigma_k(p^{\bot} \otimes I^{op}) -  \sigma_k(p \otimes (p^{\bot})^{op})$ so that if $\xi \in B_{k,R}$, then $\eta=\sigma_k(p \otimes p^{op})^{\bot}(\xi)$ has (complex) rank no greater than $2k(1-tr_k(p))$ and $\|\eta\|_2 \leq R$.  Substituting this into the above yields
\begin{eqnarray*}
(T-\partial_j F(x))(\sigma_k(p \otimes p^{op})^{\bot}\xi) = \sum_{i=1}^{N_0} c_i a_i (x_{q_i} - y_{q_i}) m_i \eta^{s_i} b_i + \sum_{i=k+1}^N c_i a_i \eta^{s_i} m_i (x_{q_i}-y_{q_i}) b_i.
\end{eqnarray*} 
For $1 \leq i \leq N_0$, by H{\"o}lder's inequality, the fact that $a_i$, $b_i$, and $m_j$ are products of elements in $B_{k,R}$ of length no greater than $k_j$, $s_j$, and $n_j$ and since $x, y \in (B_{k,R})^n$, it follows that $a_i (x_{q_i} - y_{q_i}) b_i \eta^{d_i} m_j \in E(\max(R,R^N), 2/N, k, 2k(1-tr_k(p)))$. 
\begin{eqnarray*}
c_i a_i (x_{q_i} - y_{q_i}) m_i \eta^{d_i} b_i \in E(C \cdot \max(R,R^{\deg(F)}), 2/\deg(F), k, 2k(1-tr_k(p))).   
\end{eqnarray*}
A completely analogous argument shows that for $k+1 \leq i \leq N$ 
\begin{eqnarray*}
c_i a_i \eta^{d_i} m_i(x_{q_i}-y_{q_i}) b_i \in E(C \cdot \max(R,R^{\deg(F)}), 2/\deg(F), k, 2k(1-tr_k(p))).  
\end{eqnarray*}
Putting these two facts together and using sumset notation produces
\begin{eqnarray*}
 (T - \partial_j F(x))(\sigma_k(p \otimes p^{op})^{\bot}(B_{k,R})) & \subset & \boxplus_{i=1}^{N} E\left (C \cdot \max(R,R^N), 2/N, k, 2k(1-tr_k(p)) \right) \\
& \subset & \boxplus_{i=1}^{N_1} E\left (N_1 \cdot \max(R,R^{N_1}), 2/N_1, k, 2k(1-tr_k(p)) \right). \\
\end{eqnarray*}
\end{proof}

The extension of Lemma 5.24 to a general $p$-tuple is immediate: 

\begin{corollary} Suppose $F = (f_1,\ldots, f_p) \in (\mathfrak{A}_n)^p$ and $B = c(F) \cdot (2n)^{c(F)}$.   There exists an $N_1 \in \mathbb N$, $N_1 \leq B+1$ such that if $R >0$ and $x,y \in (B_{k, R})^n$, $P = (\sigma_k(p \otimes p^{op}),\ldots, \sigma_k(p \otimes p^{op}))$ for a projection $p \in M_k(\mathbb C)$, and $T = \int_0^1 DF(x+t(y-x)) \, dt$, then 
 \begin{eqnarray*}
 (T - DF(x))(P^{\bot}(B_{k,R})^n) \subset \oplus_{j=1}^p \left(  \boxplus_{i=1}^{nN_1} E\left(N_1 \cdot \max(R,R^{N_1}), 2/N_1, k, 2k(1-tr_k(p)) \right) \right).
 \end{eqnarray*}
\end{corollary}

Notice that in Corollary 5.25 and Lemma 5.24 above the Schatten $2/N_1$-quasi-norm is used.  Computing the $\epsilon$-neighborhood of the resultant $E$-set w.r.t. the $L^2$-norm shows that it has suitable entropy estimates which translate to bounds on the fringes.  This follows from a routine repetition of the original estimates in \cite{sr} modulo technical details (see Proposition A.5 in the appendix).  Indeed, putting Corollary 5.22 together with Corollary 5.25 and Proposition A.5 yield the following main result of this subsection.  In what follows the covering number quantities $K_{\epsilon}$ will be taken w.r.t. the usual normalized inner product norm $\|\cdot\|_2$ on $(M_k(\mathbb C))^p$.

\begin{proposition} Suppose $F = (f_1,\ldots, f_p) \in (\mathfrak{A}_n)^p$, $R > 1 > \rho >0$, and $B = 2np \cdot c(F)^2 \cdot (2n)^{2c(F)}$.  There exists a constant $D_B >0$ dependent only on $B$ such that if $\xi_0 \in (B_{k, R})^n$ and $K=B_2(\xi_0, \rho)$, then $\mathcal D(F,K)$, has a $(BR^{\deg(F)}\rho^{1/2})$-binding $\Theta$ focused on $DF(\xi_0)$ such that for any $1> \epsilon >0$,
\begin{eqnarray*}
K_{\epsilon}\left (\mathcal F(\Theta, R) \right) & \leq & \left(\frac{D_B \cdot (B R^B +1)^2\sqrt{p}}{\epsilon} \right)^{16B^3 \rho^{\frac{1}{2B}}k^2}.\\
\end{eqnarray*}
\end{proposition}
\begin{proof} If $B = 2np \cdot c(F)^2 \cdot (2n)^{2c(F)}$, then Corollary 5.22 shows that that for any $\xi_0, \rho$ and $K$ as in the proposition's statement, $\mathcal D(F,K)$ has a $(BR^{\deg(F)}\rho^{1/2})$-binding $\Theta$ focused on $DF(\xi_0)$ such that for any $T \in \mathcal D(F, K)$, $\Theta(T) = (\sigma_k(p \otimes p^{op}),\ldots, \sigma_k(p \otimes p^{op}))$ where $p \in M_k(\mathbb C)$ is an orthogonal projection with $tr_k(p) > 1-Bp^{\frac{1}{B}}$.  

Now if $T \in \mathcal D(F,K)$, then $T = T_{x,y} = \int_0^1 DF(x + t(y-x)) \, dt$ for some $x,y \in K = B_2(\xi_0,\rho) \subset (B_{k,R+1})^n$.  By Corollary 5.25 there exists an $N_1 \in \mathbb N$, $N_1 < B$ such that 
\begin{eqnarray*}
(T - DF(x))(\Theta(T)^{\bot}(B_{k,R})^n) \subset \oplus_{j=1}^p \left( \boxplus_{i=1}^{nN_1} E(N_1 R^{N_1}, \tfrac{2}{N_1}, k, 2kB\rho^{1/B}) \right).
\end{eqnarray*}
By definition of the fringe of $\Theta$,
\begin{eqnarray*}
\mathcal F(\Theta, R) \subset \oplus_{j=1}^p \left( \boxplus_{i=1}^{nN_1} E(N_1 R^{N_1}, \tfrac{2}{N_1}, k, 2kB\rho^{1/B}) \right).
\end{eqnarray*}
Using Proposition 2.1 (i) (monotonicity of $K_{\epsilon}$), Lemma 2.2, and Proposition A.5 there exists a universal constant $D_{2/N_1} >0$ dependent only on $2/N_1$ such that for any $1> \epsilon >0$,
\begin{eqnarray*}
K_{\epsilon}(\mathcal F(\Theta,R)) & \leq & K_{\epsilon}\left [\oplus_{j=1}^p  \left( \boxplus_{i=1}^{nN_1} E(N_1 R^{N_1}, \tfrac{2}{N_1}, k, 2kB\rho^{1/B})  \right) \right ] \\ 
                                                        & \leq & \left(\frac{D_{2/N_1} (N_1 R^{N_1} +1)^2 \sqrt{p}}{\epsilon} \right)^{8p nN_1 \sqrt{2B\rho^{1/B}}k^2}\\
                                                        & \leq & \left(\frac{D_B (B R^{B} +1)^2 \sqrt{p}}{\epsilon} \right)^{16B^3 \rho^{\frac{1}{2B}}k^2}
\end{eqnarray*} 
where  $D_B=D_{2/N_1}$.
\end{proof}

\section{Iterating Spectral Splits II: Geometric Decay and finite $\alpha$-covering entropy}

The previous section discussed a heuristic argument for obtaining $\alpha$-covering entropy bounds as well as some technical devices for dealing with 'asymptotically coarse' estimates.  In this section I'll put these parts together to arrive at the main entropy result.  I'll then compute some simple examples (commutators, normalizers, skew-normalizers, and staggered relations) which hold in a general tracial von Neumann algebra setting.  

\subsection{The Main Estimates}

Here is the notion of geometric decay:

\begin{definition} Suppose $\mu$ is a Borel measure on $\mathbb R$ with support contained on $[0,\infty)$.  $\mu$ has geometric decay if there exists an $\epsilon_0 \in (0,1)$ such that $\sum_{n=1}^{\infty} \mu((0, \epsilon_0^n)) < \infty$.
\end{definition}

\begin{definition} Suppose $T$ is an operator such that $|T|$ lies in $(M,\varphi)$ and denote by $\mu$ the spectral distribution of $|T|$.  $T$ has geometric decay (w.r.t. $M$) iff $\mu$ has geometric decay.
\end{definition}

The lemma below shows that $T$ has geometric decay iff $T$ is of determinant class (in the sense of \cite{luckbook}) iff $\det_{FKL}(T) > -\infty$.  I will use the above terminology/definition instead of the phrase 'determinant class' or 'finiteness of $\det_{FKL}$' at times to reinforce this discrete formulation.

\begin{lemma} Suppose $\mu$ is a probability measure with compact support contained in $[0,\infty)$.  The following conditions are equivalent:
\begin{enumerate}
\item $1_{(0,\infty)}(t) \cdot \log t$ is integrable w.r.t. $\mu$ on $(0, \infty)$.
\item For any $\epsilon \in (0,1)$, $\sum_{n=1}^{\infty} n \mu([\epsilon^{n+1}, \epsilon^n)) < \infty$.
\item For some $\epsilon_0 \in (0,1)$, $\sum_{n=1}^{\infty} n \mu([\epsilon_0^{n+1}, \epsilon_0^n)) < \infty$.
\item For any $\epsilon \in (0,1)$, $\sum_{n=1}^{\infty} \mu((0, \epsilon^n)) < \infty$.
\item For some $\epsilon_0 \in (0,1)$, $\sum_{n=1}^{\infty} \mu((0, \epsilon_0^n)) < \infty$.
\end{enumerate}
\end{lemma}

\begin{proof}  For $\epsilon \in (0,1)$ fixed,
\begin{eqnarray*} \int_{(0,1)} |\log t| \, d\mu(t) & = & \sum_{n=0}^{\infty} \int_{[\epsilon^{n+1},\epsilon^n)} |\log t| \, d\mu(t) \\ 
                                                                & \leq & |\log \epsilon| \sum_{n=0}^{\infty} (n+1) \mu([\epsilon^{n+1}, \epsilon^n)) \\ 
                                                                & \leq & |\log \epsilon| \cdot \sum_{n=0}^{\infty} n \mu([\epsilon^{n+1}, \epsilon^n))  + |\log \epsilon| \\
                                                                & \leq & \int_{(0,1)} |\log t| \, d\mu(t) + |\log \epsilon|.\\
\end{eqnarray*}

\noindent The equivalence of (1)-(3) follows.  For any $\epsilon \in (0,1)$, by Fubini

\begin{eqnarray*} \sum_{n=1}^{\infty} n \mu([\epsilon^{n+1}, \epsilon^n)) & = & \sum_{n=1}^{\infty} \sum_{j=1}^n \mu([\epsilon^{n+1}, \epsilon^n)) \\ & = & \sum_{j=1}^{\infty} \sum_{n=j}^{\infty} \mu([\epsilon^{n+1}, \epsilon^n)) \\
                              & = & \sum_{j=1}^{\infty} \mu((0, \epsilon^j)).\\
\end{eqnarray*}
\noindent So (2) $\Leftrightarrow$ (4) and (3) $\Leftrightarrow$ (5) completing the proof.
\end{proof}

Formulations (4) and (5) were what I was originally interested in using as an upper bound for the metric entropies.   Recall from Section 2.7 that if $x \in M$ is a positive operator and $\mu$ is the spectral distribution of $x$ induced by $\varphi$, then $x$ is of determinant class iff
\begin{eqnarray*}
\int_{(0,\infty)} |\log (\lambda)| \, d\mu(\lambda) < \infty.
\end{eqnarray*}
By Lemma 6.3 this condition is satisfied iff for some $\epsilon_0 >0$, $\sum_{n=1}^{\infty} \mu((0,\epsilon_0^n)) < \infty$ iff $x$ has geometric decay at $0$.  Thus,

\begin{corollary} Suppose $T$ is an operator such that $|T|$ lies in $(M,\varphi)$.  $T$ has geometric decay iff $|T|$ is of determinant class.
\end{corollary}

\begin{remark} If the spectral distribution of $|T|$ is absolutely continuous w.r.t. Lebesgue measure with a density $f$ such that $|f(t)| \leq C t^{\alpha}$ for some constant $C$ and $\alpha >-1$, then $T$ has geometric decay.  In particular, if $f \in L^{\infty}(\mathbb R)$, then $T$ has geometric decay.
\end{remark}

\begin{remark} Suppose $T \in M$ is a normal operator and denote by $\mu$ the spectral distribution of $T$ induced by $\varphi$.  Assume that $\mu$ has the property that for any $\lambda \in \mathbb C$ there exists a corresponding $\epsilon >0$ such that 
\begin{eqnarray*}
\sum_{k=0}^{\infty} \mu(B(\lambda, \epsilon^k)) < \infty.
\end{eqnarray*}
Suppose $p$ is a nonzero polynomial.  Then $p(T)$ has geometric decay.  Indeed, there exist complex numbers $\lambda_1,\ldots, \lambda_d$ such that for any $\lambda \in \mathbb C$, $p(\lambda) = \Pi_{j=1}^d (\lambda - \lambda_j)$.  For any $\epsilon >0$ the spectral theorem yields
\begin{eqnarray*}
\varphi(1_{(0,\epsilon)}(|p(T)|)) & = & \varphi((1_{(0,\epsilon)} \circ |p|))(T) \\
                                                  & = & \mu(\{\lambda \in \mathbb C: \Pi_{j=1}^d |\lambda - \lambda_j| < \epsilon\})\\
                                                  & \leq & \sum_{j=1}^d \mu(B(\lambda_j, \epsilon^{1/d})).\\
\end{eqnarray*}
Choose for each $j$ an $\epsilon_j>0$ such that $\sum_{k=0}^{\infty} \mu(B(\lambda_j, \epsilon_j^k)) < \infty$ and set $\epsilon_0 = \min_{1 \leq j \leq d} \epsilon_j^d$.  It follows that 
\begin{eqnarray*} \sum_{k=1}^{\infty} \varphi(1_{(0,\epsilon_0^k)}(|p(T)|)) & \leq & \sum_{k=1}^{\infty} \sum_{j=1}^d \mu(B(\lambda_j, \epsilon_0^{k/d})) \\
                                                                                                                    & \leq & \sum_{j=1}^d \sum_{k=1}^{\infty} \mu(B(\lambda_j, \epsilon_j^k)) \\
                                                                                                                    & < & \infty.
\end{eqnarray*}
Thus, $p(T)$ has geometric decay.

The spectral distribution $\mu$ of $T$ will satisfy these local density condition when $T$ is a normal operator with an $L^{\infty}$ density w.r.t. Lebesgue measure on $\mathbb R^2 \simeq \mathbb C$ or when $T$ is a unitary operator with an $L^{\infty}$ density w.r.t. Lebesgue measure on the unit circle.  Related computations have been made previously, explicitly computing the Fuglede-Kadison-L{\"u}ck determinant and identifying it with Mahler measure (\cite{luckbook}, \cite{pdh}).
\end{remark}

\begin{definition} $D^sF(X)$ has geometric decay if $|D^sF(X)| = (D^sF(X)^*D^sF(X))^{1/2}$, regarded as an element of the tracial von Neumann algebra $M_{2n}(M\otimes M^{op})$, has geometric decay.  
\end{definition}

Given a finite tuple $F$ of elements in $\mathfrak{A}_n$ (the universal, unital complex $*$-algebra in $n$-indeterminates), $c(F)$ and $\deg(F)$ will have the same meaning as in Section 5.4. 

For the remainder of this section define $\psi_n=tr_n \otimes tr_2 \otimes (\varphi \otimes \varphi^{op})$.  $\psi_n$ is the canonical tracial state on $M_{2n}(M\otimes M^{op}) = M_n(\mathbb C) \otimes (M_2(M \otimes M^{op}))$ induced by $(M,\varphi)$.  For an element $T \in M_{2n}(M \otimes M^{op})$ recall that if $P= 1_{\{0\}}(|T|)$, then $2n \psi_n(P) = \text{Nullity}(T)$ (subsection 2.7).

Here is the fundamental microstates result which quantifies the local and global entropy loss and gains:

\begin{theorem}  Suppose $F$ is a $p$-tuple of elements in $\mathfrak{A}_n$, $L = 2np \cdot c(F)^2 \cdot (2n)^{c(F)}$, and $X$ is an $n$-tuple of operators in $M$ with operator norms strictly less than $R$ such that $F(X)=0$.   There exists a constant $D$ dependent only on $F$ such that for any $\frac{1}{5}>\rho, \delta, t > 0$ satisfying $t \rho^{-1/4} \delta - 5LR^{\deg(F)} > 1$    
\begin{eqnarray*}
\mathbb K_{\rho \delta}(X) & \leq & \mathbb K_{\rho}(X) + \\ & & [2n \cdot \psi(1_{[0,\rho^{1/2}]}(|D^sF(X|))+ t] \cdot | \log ((1-t^2)^{1/2} \delta | + \\ & & 24 D\rho^{1/D} \cdot ( \log 2D + |\log \rho|).
\end{eqnarray*}
\end{theorem}

\begin{proof}  Using the assumption $F(X)=0$ and Proposition 3.13, there exists an $m_1 \in \mathbb N$ and $0<\gamma_1$ such that if $\xi \in \Gamma_R(X;m_1,k,\gamma_1)$, then the following two conditions are satisfied:

\begin{itemize}
\item $\|F(\xi)\|_2 < \rho^2/12$.  
\item The real dimension of the range of the projection $1_{[0,\rho^{1/4}]}(DF(\xi)^*DF(\xi))$ as a real linear operator on $(M_k(\mathbb C))^n$ is no greater than $2nk^2 \cdot [\psi(1_{[0,\rho^{1/4}]}(|D^sF(X)|^2))+ t]$.  
\end{itemize}

Suppose that $m_1 < m \in \mathbb N$ and $0 < \gamma < \gamma_1$.  For each $k \in \mathbb N$ find a minimal $\rho$-cover $\langle \xi_i \rangle_{i \in I_k}$ w.r.t. the $\|\cdot \|_2$-norm for $\Gamma_R(X;m,k,\gamma)$ such that
\begin{eqnarray*}
\#I_k = K_{\rho}(\Gamma_R(X;m,k,\gamma)).
\end{eqnarray*}
\noindent Fix a $k$ and an $i \in I_k$.  Set $E = B_2(\xi_i, \rho) \cap \Gamma_R(X;m_1,k,\gamma_1)$.  I will now find a minimal $\rho \epsilon$-cover for $E$ with a suitable upper bound on its cardinality.  Find and fix an $\xi_0 \in E$.  Set $K= B_2(\xi_0, 2\rho)$.   By Proposition 5.26 there exists a universal constant $D$ dependent only on $L$ such that if $\xi_0 \in (M_k(\mathbb C)_R)^n$ and $K = B_2(\xi_0, 2\rho)$, then $\mathcal D(F,K)$ has an $(LR^{\deg(F)}\rho^{1/2})$-binding $\Theta_1$ focused on $DF(\xi_0)$ such that for any $1> \epsilon >0$,
\begin{eqnarray*}
K_{\epsilon}\left (\mathcal F(\Theta_1, 1) \right) & \leq & \left( \frac{D}{\epsilon} \right)^{\left(16D\rho^{1/D} \cdot k^2 \right )}.
\end{eqnarray*}
By the triangle inequality $B(\xi_i, \rho) \subset K$ so by definition $\mathcal D(F, B(\xi_i,\rho)) \subset \mathcal D(F,K)$.  $\Theta_1$ induces by restriction a $(LR^{\deg(F)}\rho^{1/2})$-binding $\Theta$ of $\mathcal D(F, B(\xi_i, \rho))$ focused on $DF(\xi_0)$ (Remark 5.5).  

I want to invoke Lemma 5.8 with the following substititutions: $f=F$, $V=(M_k(\mathbb C))^n$, $W=(M_k(\mathbb C))^p$, $x_0=\xi_0$, $B=B_2(\xi_i,\rho)$, $\Lambda = \mathcal D(F,B)$, $\rho_1 = LR^{\deg(F)}\rho^{1/2}$, $\alpha = \rho^{1/4}$, $\gamma = \rho^2/12$, $\epsilon =  \rho \delta$, $E$ and $\Theta$ as defined above, and $\beta = \alpha t \epsilon - 4 \rho \rho_1 -4 \gamma$.  The $\rho$ here will correspond with the $\rho$ in Lemma 5.8.  Denote by $Q^{\bot}$ the (real linear) projections $1_{[0,\alpha)}(|Df(\xi_0)|)$.   In order to use Lemma 5.8 I need to check that $\beta >0$:  
\begin{eqnarray*}
\beta & = & \alpha t \epsilon -4 \rho \rho_1 - 4 \gamma \\
          & = & \rho^{1/4} t \rho \delta - 4 \rho (LR^{\deg(F)} \rho^{1/2}) - 4 \gamma \\
         & > & t \rho^{5/4} \delta - 5 LR^{\deg(F)} \rho^{3/2} \\
         & = & \rho^{3/2} \cdot (t \rho^{-1/4} \delta - 5 LR^{\deg(F)}) \\
         & > & \rho^{3/2} \\
         & > & 0.\\
\end{eqnarray*}
Now $Q^{\bot}(E)$ is contained in a ball of radius $\rho$ in a copy of Euclidean space of dimension no greater than the real dimension of the range of the projection $1_{[0,\rho^{1/4}]}(DF(\xi_0)^*DF(\xi_0))$.  By the second condition stated above, this dimension is dominated by $2nk^2[\psi(1_{[0,\rho^{1/4}]}(|D^sF(X)|^2))+ t]$.  Applying Lemma 5.8, Rogers's asymptotic sphere covering estimate (\cite{rogers}), and the lower bound estimate on $\beta$ above
\begin{eqnarray*}
K_{\rho \delta}(E) & \leq &K_{(1-t^2)^{1/2} \rho \delta}(Q^{\bot}(E)) \cdot S_{\beta}(\mathcal F(\Theta, 2\rho)) \\
                            & \leq & (2nk^2)^3 \cdot \left ( \frac{\rho}{(1-t^2)^{1/2} \rho \delta} \right)^{[2nk^2 \psi(1_{[0,\rho^{1/4}]}(|D^sF(X)|^2))+ t]} \cdot S_{\rho^{3/2}}(\mathcal F(\Theta, 2\rho)). \\
\end{eqnarray*}
The third term on the RHS can be further dominated.  By Remark 5.3, the covering estimate on $\mathcal F(\Theta_0,1)$, and the properties of covering numbers and separating numbers (Proposition 2.1),
\begin{eqnarray*}
S_{\rho^{3/2}}(\mathcal F(\Theta, 2\rho)) & \leq & S_{\rho^{3/2}}(\mathcal F(\Theta_1, 2\rho)) \\
                                                                 & \leq & S_{\rho^{3/2}}(\mathcal F(\Theta_1, 1)) \\
                                                                 & \leq & K_{\frac{\rho^{3/2}}{2}}(\mathcal F(\Theta_1, 1)) \\
                                                                 &  \leq & \left( \frac{2D}{\rho^{3/2}} \right)^{\left(16D\rho^{1/D} \cdot k^2 \right )}.
\end{eqnarray*}
Putting this together with the inequalities which preceded it gives
\begin{eqnarray*}
K_{\rho \delta}(E) & \leq & (2nk^2)^3 \cdot \left ( \frac{\rho}{(1-t^2)^{1/2} \rho \delta} \right)^{[2nk^2 \psi(1_{[0,\rho^{1/4}]}(|D^sF(X)|^2))+ t]}  \cdot \left( \frac{2D}{\rho^{3/2}} \right)^{\left(16D\rho^{1/D} \cdot k^2 \right )}.
\end{eqnarray*}

Recall that $m_1 < m \in \mathbb N$, $0 < \gamma <\gamma_1$, and $\langle \xi_i \rangle_{I_k}$ is a minimal $\rho$-cover $\langle \xi_i \rangle_{I_k}$ for $\Gamma_R(X;m,k,\gamma)$.  By the preceding two paragraph there exists for each $i \in I_k$ a $\rho \delta$-cover for $E_i = B_2(\xi_i, \rho) \cap \Gamma_R(X,m,k,\gamma)$ with cardinality no greater than the last dominating expression of the preceding inequality.  Using Proposition 2.3,
\begin{eqnarray*}
\mathbb K_{\rho \delta}(X) & = & \mathbb K_{\rho \delta,R}(X) \\ 
                                         & \leq & \mathbb K_{\rho \delta, R}(X;m,\gamma) \\ & \leq & \limsup_{k \rightarrow \infty} k^{-2} \cdot \log \left ( \#I_k \cdot (2nk^2)^3 \left ( \frac{\rho}{(1-t^2)^{1/2} \rho \delta} \right)^{[2nk^2 \cdot \psi(1_{[0,\rho^{1/4}]}(|D^sF(X|^2))+ t]k^2} \right ) + \\
& & \limsup_{k \rightarrow \infty} k^{-2} \cdot \log\left (\frac{2D}{\rho^{3/2}} \right)^{\left(16D\rho^{1/D} \cdot k^2 \right)} \\
                                                             & \leq & \mathbb K_{\rho}(X;m,\gamma) + \\ 
                                                             & &  [2n \psi(1_{[0,\rho^{1/2}]}(|D^sF(X|))+ t] \cdot | \log ((1-t^2)^{1/2} \delta | + \\
                                                             & & 24D\rho^{1/D} \cdot ( \log 2D + |\log \rho|). \\
\end{eqnarray*}
\noindent As this is true for sufficiently large $m$ and small $\gamma$ the desired inequality follows.
\end{proof}

The preceding theorem extracts the basic covering estimate from the von Neumann algebra and microstate setting.  Having done this, proving the main result will now involve iterating the spectral splits on a geometric scale and aggregating the entropy.

\begin{theorem} Suppose $X$ is an $n$-tuple of operators in $(M,\varphi)$ and $F$ is a $p$-tuple of $*$-polynomials in $n$ indeterminates such that $F(X)=0$.  If $\alpha = \text{Nullity}(D^sF(X))$ and $D^sF(X)$ has geometric decay, then $X$ is $\alpha$-bounded.
\end{theorem}

\begin{proof} Set $L = 2np \cdot c(F)^2 \cdot (2n)^{c(F)}$ and $R >0$ strictly greater than the operator norm of any element of $X$ as in the statement of Theorem 6.8.  Fix once and for all an $\epsilon_0 \in (0,1/5)$ such that for all $k \in \mathbb N$ with $k >4$, 
\begin{eqnarray*}
\epsilon_0^{-(k/4-1)} > (5 L R^{\deg(F)} +1) k^2.
\end{eqnarray*}
Set $t_k = k^{-2}$.  The condition on $\epsilon_0$ implies $t_k \cdot \epsilon_0^{-k/4} \cdot \epsilon_0 - 5 L R^{\deg(F)} > 1$.  By Theorem 6.8 there exists a $D$ dependent only on $F$ such that for $k >4$, setting $\rho = \epsilon_0^k$, $\delta = \epsilon_0$, and $t = t_k$ in the context of Theorem 6.8 yields
\begin{eqnarray*} \mathbb K_{\epsilon_0^{k+1}}(X) & = & \mathbb K_{\epsilon_0^k \epsilon_0}(X) \\
                                                                                & \leq & \mathbb K_{\epsilon_0^k}(X) + [2n \cdot \psi(1_{[0,\epsilon_0^{k/2}]}(|D^sF(X|))+ k^{-2} ] \cdot | \log [(1-k^{-4})^{1/2} \epsilon_0]| + \\ & & 24D\epsilon_0^{k/D} \cdot ( \log 2D + |\log \epsilon_0^k|)\\
                                                                                & \leq & \mathbb K_{\epsilon_0^k}(X) + [2n \cdot \psi(1_{[0,\epsilon_0^{k/2}]}(|D^sF(X|))+ k^{-2} ] \cdot [k^{-4} + |\log \epsilon_0|] + \\ & & 24D\epsilon_0^{k/D} \cdot ( \log 2D + |\log \epsilon_0^k|).\\
\end{eqnarray*}
\noindent Iterating yields for any $k>4$,
\begin{eqnarray*}
\mathbb K_{\epsilon_0^{k+1}}(X) & \leq & \mathbb K_{\epsilon_0^5}(X) + \\ & &  \sum_{j=5}^k  [2n \cdot \psi(1_{[0,\epsilon_0^{j/2}]}(|D^sF(X|))+ j^{-2} ] \cdot [j^{-4} + |\log \epsilon_0|] + \\ & & \sum_{j=5}^k 24D\epsilon_0^{j/D} \cdot ( \log 2D+ |\log \epsilon_0^j|) \\
& = & \mathbb K_{\epsilon_0^5}(X) +  C_1 + C_2 \\
\end{eqnarray*}
where $C_1$ and $C_2$ denote the second and third term of the second expression above.  

To estimate the $C_1$ term set 
\begin{eqnarray*}
D_1 = \sum_{j=5}^{\infty} \left [2n \cdot \psi(1_{[0,\epsilon_0^{j/2}]}(|D^sF(X|))\cdot j^{-4} + j^{-6} +j^{-2} \cdot |\log \epsilon_0| \right]
\end{eqnarray*} 
and 
\begin{eqnarray*}
D_2 = 2n |\log \epsilon| \cdot \sum_{j=5}^{\infty} \psi(1_{(0,\epsilon_0^{j/2}]}(|D^sF(X)|)).
\end{eqnarray*}
Clearly $D_1 < \infty$ and $D_2 < \infty$ by the geometric decay assumption on $D^sF(X)$.
\begin{eqnarray*}
C_1 & = & \sum_
{j=5}^k  [2n \cdot \psi(1_{[0,\epsilon_0^{j/2}]}(|D^sF(X)|)+ j^{-2} ] \cdot [j^{-4} + |\log \epsilon_0|] \\ & \leq & \sum_{j=5}^k  2n \cdot \psi(1_{[0,\epsilon_0^{j/2}]}(|D^sF(X|))\cdot |\log \epsilon_0| +\\ &&  \sum_{j=5}^{\infty} \left [2n \cdot \psi(1_{[0,\epsilon_0^{j/2}]}(|D^sF(X|))\cdot j^{-4} + j^{-6} +j^{-2} \cdot |\log \epsilon_0| \right] \\ 
                                                                                                                                                 & = &  \left( \sum_{j=5}^k  [\alpha + 2n \cdot \psi(1_{(0,\epsilon_0^{j/2}]}(|D^sF(X|))] \cdot |\log \epsilon_0|] \right) + D_1 \\ 
                                                                                                                                                 & < & (k-4) \alpha |\log \epsilon_0| + 2n\cdot |\log \epsilon| \cdot \sum_{j=5}^{\infty} \psi(1_{(0,\epsilon_0^{j/2}]}(|D^sF(X|)) + D_1 \\
                                                                                                                                                 & = & \alpha |\log \epsilon_0^{k-4}| + D_1+ D_2 \\
\end{eqnarray*}

Turning to the $C_2$ term,
\begin{eqnarray*}
 C_2 & = & \sum_{j=5}^k 24D\epsilon_0^{j/D} \cdot ( \log 2D + |\log \epsilon_0^j|) \\ & \leq & 24 D \sum_{j=1}^{\infty} (\epsilon_0^{1/D})^j \cdot ( \log 2D + j |\log \epsilon_0|)\\
 & < & \infty \\
 \end{eqnarray*}
as the series on the right hand sides have geometric terms paired with polynomial ones which force convergence.   
  
Set $D = \mathbb K_{\epsilon_0^5}(X) +D_1 + D_2 + C_2$.  $D$ is independent of $k$. Putting the above inequalities together yields for any $k \in \mathbb N$, $k >4$,
 \begin{eqnarray*}
 \mathbb K_{\epsilon_0^k}(X) & \leq & \mathbb K_{\epsilon_0^5}(X) + C_1 + C_2 \\
 & < & \mathbb K_{\epsilon_0^5}(X) + \alpha \cdot |\log \epsilon_0^{k-4}| +D_1 + D_2 + C_2 \\
 & < & \alpha \cdot |\log \epsilon_0^k| + D.
 \end{eqnarray*}
This bound almost gives the estimate for any sufficiently small $\epsilon$.  To finish the argument suppose $\epsilon \in (0,\epsilon_0)$.  Find a $k \in \mathbb N$ such that $\epsilon_0^{k+1} \leq \epsilon < \epsilon_0^k$.  Using the monotonicity of $\mathbb K_{\cdot}()$ (Lemma 2.5),

\begin{eqnarray*} \mathbb K_{\epsilon}(X) & < & \mathbb K_{\epsilon_0^{k+1}}(X) \\ & \leq & \alpha \cdot |\log \epsilon_0^{k+1}|  +D \\ & \leq & \alpha \cdot  |\log \epsilon_0| + \alpha \cdot |\log \epsilon_0^k| +D\\ & < &  \alpha |\log \epsilon| + (\alpha \cdot |\log \epsilon_0| + D).\\
\end{eqnarray*}
By definition, $X$ is $\alpha$-bounded.
\end{proof}

\begin{remark} It may be possible to remove the condition $F(X)=0$ in Theorems 6.8 and 6.9.  Theorem 6.9 would then take the following, somewhat more general form.  Suppose $X$ is an $n$-tuple of operators $(M,\varphi)$ and $F$ is a $p$-tuple of $*$-polynomials in $n$ indeterminates such that $F(X)$ is a $\beta$-bounded $p$-tuple of operators.  If $\alpha = \text{Nullity}(D^sF(X))$ and $D^sF(X)$ has geometric decay, then $X$ is $(\alpha+\beta)$-bounded.  Proving this would perhaps involve making changes to Lemma 5.7 and Lemma 5.8, and then adjusting the bookkeeping in Theorem 6.8 and 6.9.  I won't pursue this more general line of reasoning here.
\end{remark}

\begin{corollary} Suppose $X$ is an $n$-tuple of self-adjoint elements in a tracial von Neumann algebra, $F$ is a $p$-tuple of noncommutative self-adjoint $*$-polynomials in $n$ indeterminates such that $F(X)=0$, and $\alpha = \text{Nullity}(D^{sa}F(X))$.  If $D^{sa}F(X)$ has geometric decay, then $X$ is $\alpha$-bounded.   
\end{corollary}

\begin{proof} Define $L = \{(X_1 - X_1^*)/2, \ldots, X_n - X_n^*)/2\}$, $G = F \cup L$.  By Proposition 3.17
\begin{eqnarray*} \alpha & = & \text{Nullity}(D^{sa}F(X)) \\
                                        & = & \text{Nullity}(D^sG(X)) \\
\end{eqnarray*}
and there exists a $c>0$ such that for any $t \in (0,1)$, $\mu((0,t])) \leq \nu((0,ct])$ where $\mu$ and $\nu$ are the spectral distributions of $|D^sG(X)|$ and $|D^{sa}F(X)|$, respectively. Since $D^{sa}F(X)$ has geometric decay, Lemma 6.2 and the inequality on the spectral distributions implies that $D^sG(X)$ has geometric decay.   Obviously $G(X)=0$.  By Theorem 6.9, $X$ is $\alpha$-bounded. 
\end{proof}
 
 \begin{corollary} Suppose $X$ is an $n$-tuple of unitaries in a tracial von Neumann algebra, $F$ is a $p$-tuple of noncommutative $*$-monomials in $n$ indeterminates such that each entry of $F(X)$ is the identity operator, and $\alpha = \text{Nullity}(D^uF(X))$.  If $D^uF(X)$ has geometric decay, then $X$ is $\alpha$-bounded.   
 \end{corollary}
 
 \begin{proof} The argument proceeds exactly as in the self-adjoint case.  Define $G=\{X_1^*X_1-I, \ldots, X_n^*X_n-I\}$ and $H = \langle f_i -I \rangle_{i=1}^p \cup G \subset (\mathfrak{A}_n)^{p+n}$.  By Proposition 3.26
 \begin{eqnarray*} \alpha & = & \text{Nullity}(D^{u}F(X)) \\
                                        & = & \text{Nullity}(D^sH(X)) \\
\end{eqnarray*}
and for any $t \in (0,1)$, $\mu((0,t])) \leq \nu((0,t])$ where $\mu$ and $\nu$ are the spectral distributions of $|D^sH(X)|$ and $|D^uF(X)|$, respectively. Since $D^uF(X)$ has geometric decay, Lemma 6.2 and the inequality on the spectral distributions implies that $D^sH(X)$ has geometric decay.  Obviously $H(X)=0$.  By Theorem 6.9, $X$ is $\alpha$-bounded. 
 \end{proof}
 
Theorem 6.9, Corollary 6.11, and Corollary 6.12 bound the $\alpha$-covering entropy in terms of the nullity/rank of the derivative provided that the derivative has geometric decay.   To apply the theorem one must first verify the geometric decay property of the derivative and secondly, make a computation of its nullity.  Note here that simply computing the nullity will give a free entropy dimension bound by Corollaries 4.9, 4.10, and 4.11.   

In the context of group von Neumann algebras the geometric decay condition is related to L{\"u}ck's Determinant conjecture while the computation of the nullity/rank is naturally linked to Atiyah's conjecture as well as Kaplanksy's and Linnell's  conjectures.  These will be discussed in the next section.   For now some simple computations can be made involving noncommutative quadratic varieties in the general tracial von Neumann algebra setting.  Despite this seemingly simple situation, they will yield new nonisomorphism results and generalizations of already established bounds.
 
 \subsection{Examples}

A few simple observations will be useful in the examples to come:

\begin{lemma} Suppose $T  = [T_1 \cdots T_n] \in M_{1\times n}(M)$ and $1 \leq i \leq n$.  If $T_i \in M$ is an injective operator, then $\text{Nullity}(T)=n-1$.  If in addition $T_i$ has geometric decay, then $T$ has geometric decay.
\end{lemma}
\begin{proof}  First note that if 
\begin{eqnarray*}
\overline{T} = \begin{bmatrix}
T_1 & \cdots & T_n \\
0 & \cdots & 0 \\
\vdots & & \vdots \\
0 & \cdots & 0 \\
\end{bmatrix} \\ \in M_{n\times n}(M),
\end{eqnarray*}
then $T^*T = \overline{T}^*\overline{T} \in M_{n \times n}(M)$.  Thus it suffices to establish the nullity and geometric decay claims for $\overline{T}$.   Because $|\overline{T}|$ and $|\overline{T}^*|$ have the same spectral distribution (see the proof of Lemma 2.10) it suffices to show the claims for $\overline{T}^*$.

Towards this end assume $T_i$ is injective. $|\overline{T}^*|$ is the $n \times n$ matrix whose $11$-entry is $(T_1 T_1^* + \cdots + T_n T_n^*)^{1/2}$ and whose other entries are all $0$.   Define $A \in M_{n \times n}(M)$ to be the matrix whose $11$-entry is $(T_iT_i^*)^{1/2}$ and whose other entries are all $0$.  It follows that $|\overline{T}^*| \geq A$.  By the assumed injectivity of $T_i$ it follows that $n-1 \leq \text{Nullity}(|\overline{T}^*|) \leq \text{Nullity}(A) = n-1$. 

Suppose that in addition $T_i$ has geometric decay.  If $\mu$ and $\nu$ are the spectral distributions associated to $|\overline{T}^*|$ and $|A|$, then by Weyl's Inequality for positive operators (Section 2.7), for any $t$, $\mu([0,t]) \leq \nu([0,t])$.  The nullities of $|\overline{T}^*|$ and $|A|$ being equal, $\mu(\{0\}) = \nu(\{0\})$, and it follows that for any $t$, $\mu((0,t]) \leq \nu((0,t])$.  If $\sigma$ is the spectral distribution of $|T_i^*| \in M$, then it is straightforward to see that $\nu = n^{-1} \cdot \sigma + (n-1)/n \cdot \delta_{\{0\}}$.  $\sigma$ is also the spectral distribution of $|T_i| \in M$ (again since the absolute value of an operator and its adjoint have the same spectral distribution).  It is obvious from this equation, Lemma 6.3, and the assumed geometric decay of $|T_i|$ that $\nu$ has geometric decay.  Thus, $\mu$ must have geometric decay.  So $|\overline{T}^*|$ has geometric decay.
\end{proof}

\begin{lemma} If $X$ is $1$-bounded and there exists a self-adjoint element $y$ in the $*$-algebra generated by $X$ such that $\chi^{sa}(y) > -\infty$, then $vN(X)$ is strongly $1$-bounded.  
\end{lemma}

\begin{proof} Denote by $Y$ the $(2n+1)$-tuple consisting of $y$ and the real and imaginary parts of the operators in $X$ (so $Y$ consists of self-adjoint elements).  Both $X$ and $Y$ generate the same $*$-algebra.  By Proposition 2.5 and the assumption that $X$ is $1$-bounded, it follows that $Y$ is $1$-bounded (as a general tuple or a tuple of self-adjoint elements).  $\chi^{sa}(y) > -\infty$ so $Y$ is strongly $1$-bounded, whence $vN(X)=vN(Y)$ is a strongly $1$-bounded von Neumann algebra.
\end{proof}

Taking the example of two commuting operators isn't particularly interesting, but as with the computations in Section 4, it serves as a good sanity check.

\begin{example} Suppose $X = \{x_1, x_2\}$ consists of commuting self-adjoint elements in $M$, $F = \{f\}$ where $f(X_1, X_2) = X_2 X_1 - X_1^* X_2^*$, and that the spectral distribution of $x_1$ has an $L^{\infty}$ density $g$ w.r.t. Lebesgue measure on the real line.  The setting is exactly that of Example 4.1 except here I assume a stronger condition on the spectral distribution of $x_1$ (it implies the absence of eigenvalues condition of Example 4.1).  It was observed in Example 4.1 that \begin{eqnarray*} D^{sa}F(X) & = & 
\begin{bmatrix}
(\partial_{1}^{sa} f)(X) & (\partial_{2}^{sa}f)(X) \\
\end{bmatrix} \\ & = & 
\begin{bmatrix}
x_2 \otimes I - I \otimes x_2 & I \otimes x_1 - x_1 \otimes I \\
\end{bmatrix} \\
& \in & M_{1\times 2}(M \otimes M^{op}). \\
\end{eqnarray*}
The density of the spectral distribution of $x_1 \otimes I - I \otimes x_1$ is $g*\tilde{g}$ where $\tilde{g}(t) = g(-t)$.  $g$ is compactly supported, whence $g, \tilde{g} \in L^1(\mathbb R) \cap L^{\infty}(\mathbb R)$, and thus $g*\tilde{g} \in L^{\infty}(\mathbb R)$.  By Remark 6.5 it follows that $x_1 \otimes I  - I \otimes x_1 \in M \otimes M^{op}$ has geometric decay.  The $L^{\infty}$ density also guarantees that $x_1 \otimes I - I \otimes I$ is injective.  By Lemma 6.13 $D^{sa}F(X)$ has nullity equal to $2-1=1$ and geometric decay.  By Corollary 6.11 $X$ is $1$-bounded.  $\chi^{sa}(x_1) > -\infty$ (by \cite{v1}) so $X$ is strongly $1$-bounded.   Of course, $X$ generates an abelian von Neumann algebra and the fact that it's strongly $1$-bounded was determined early on in \cite{v2}. 
\end{example}

\begin{example}   Suppose $X = \{x_1, \ldots, x_n\}$ consists of unitaries in $M$ and $F = \{f\}$ where $f = A X_1^{s_1} B X_1^{s_2} \in \mathfrak{A}_n$, $A$ and $B$ are $*$-monomials in $X_2,\ldots, X_n$, and either $s_1=s_2=1$ or $s_1 = 1$ and $s_2=*$.   Set $a = A(X)$ and $b = B(X)$; $a$ and $b$ are $*$-monomials in $x_2,\ldots, x_n$.  Assume that $bx_1$ has an $L^{\infty}$ density w.r.t. Lebesgue measure on the unit circle when $s_1=s_2=1$ or $b$ has an $L^{\infty}$ density w.r.t. Lebesgue measure on the unit circle when $s_1=1$ and $s_2=*$ and that in either case $f(X)=I$.   The setting is exactly that of Example 4.2 except that here I assume a stronger condition on the spectral distributions (it implies the absence of eigenvalues condition of Example 4.2).  In Example 4.2 the partial derivative of $f$ with respect to the first variable was computed:
\begin{eqnarray*} \partial_1^uf(X) & = & \begin{cases} x^*b^* \otimes bx + I_M \otimes I_{M^{op}} &\mbox{if } s_1 = s_2=1 \\
(x_1 \otimes x_1^*)(b^* \otimes b - I_M \otimes I_{M^{op}}) & \mbox{if } s_1=1, s_2=*\\
\end{cases}\\                      
\end{eqnarray*}
The $L^{\infty}$-density assumption again implies in either case that the operators have geometric decay.  To see this note that by Remark 6.6 it suffices to show that $x^*b^*\otimes bx$ or $b^* \otimes b$ have $L^{\infty}$ densities w.r.t. Lebesgue measure on the unit circle.  But the product of two independent unitaries, each with a spectral distribution which has an $L^{\infty}$ densities (which must also be in $L^1$) w.r.t. Lebesgue measure on the unit circle, again has a spectral distribution with an $L^{\infty}$ density w.r.t. Lebesgue measure on the unit circle.  Thus, $\partial_1^uf(X)$ has geometric decay.  $\partial^u_1f(X)$ is also injective by the $L^{\infty}$ density observation. By Lemma 6.13 this implies that $\text{Nullity} D^uf(X) = n-1$ and $D^uf(X)$ has geometric decay.  Invoking Corollary 6.12, $X$ is $(n-1)$-bounded.  

In particular, if $n=2$, then $X$ is $1$-bounded.   Using the $L^{\infty}$-density assumption on either $bx_1$ or $b$, it follows that there exists a selfadjoint element in the $*$-algebra generated by $X$ with finite free entropy.  Invoking Lemma 6.14, when $n=2$ $X$ is strongly $1$-bounded.  Thus, the von Neumann algebra generated by $X$ is strongly $1$-bounded.
\end{example}

\begin{remark} It follows that if $\Gamma= \langle a, b | a^mb^{s_1} a^n b^{s_2} \rangle$ with $m,n \in \mathbb Z -\{0\}$ and either $s_1 =s_2 =1$ or $s_1 = -s_2 =1$, then $L(\Gamma)$ is strongly $1$-bounded.  Indeed, since every nontrivial group element in $\Gamma$ gives rise to a Haar unitary in the left regular representation, this falls into the rubric of Example 6.2.  As noted before, when $s_1 = -s_2=1$, this is the Baumslag-Solitar group and it was shows in \cite{v2} and \cite{gs} that this von Neumann algebra was strongly $1$-bounded by using the normalizing relations in a direct way.  This example can in part be rederived from the general group von Neumann algebra results in the following section.
\end{remark}
\begin{lemma}  Suppose $m \leq n$ and $T \in M_{m \times n}(M)$ is upper triangular.  If $T_{ii} \in M$ is injective for $1 \leq i \leq m$, then $\text{Nullity}(T) = n-m$.  If in addition, each $T_{ii}$ has geometric decay, then $T$ has geometric decay. 
\end{lemma}
\begin{proof} For the first claim $\text{Rank}(T) = m$ by the injectivity of each $T_{ii}$ and the upper triangular form of $T$.  The nullity statement follows.  To establish the geometric decay property, by Lemma 6.3 this is equivalent to showing that $T$ is of determinant class.  But from the properties of the Fuglege-Kadison-L{\"u}ck determinant w.r.t. upper triangularity (Section 2.6) and the assumed geometric decay of the $T_{ii}$ it follows that $\det_{FKL}(T_{jj}) >0$ for all $j$ so that
\begin{eqnarray*}
\text{det}_{FKL}(T) & = & \Pi_{j=1}^m \text{det}_{FKL}(T_{ii}) \\
                               & > & 0.\\
\end{eqnarray*}
\end{proof}
\begin{example}  Recall Example 4.4 with the $n$-tuple $X$ of self-adjoint elements and the $n-1$ tuple of polynomial relations $F=\{f_1,\ldots, f_{n-1}\}$.  Assume further that for each $i$, $\partial^{sa}_if_i$ has geometric decay and one of the elements of $X$ has finite free entropy (as a singleton tuple).  This condition is satisfied when $f_i = X_i X_{i+1} - X^*_{i+1} X^*_i$ for $1 \leq i \leq n-1$ and each element in $X$ has an $L^{\infty}$-density w.r.t. Lebesgue measure.  From Lemma 6.16 $D^{sa}F(X) \in M_{n-1 \times n}(M \otimes M^{op})$ has nullity equal to $n-1$ and has geometric decay.  Thus, by Corollary 6.11 it follows that $X$ is $1$-bounded.  The condition of finite free entropy for one of the elements of $X$ shows that $X$ is strongly $1$-bounded.  This gives another way to see how tensor products of tracial von Neumann algebras are strongly $1$-bounded von Neumann algebras, a result first obtained by \cite{g}.  One can use other polynomials aside from commutators and do this in the unitary case, provided that the geometric decay conditions are satisfied (e.g., the noncommutative words in Example 6.2).  I'll say more about these staggered relations in the next section.
\end{example}

\section{Group von Neumann Algebras Applications}

In this section I want to apply the results in Section 6 to discrete group von Neumann algebras whose underlying groups satisfy additional, seemingly natural (but not so easy to verify) properties.  There are two group properties which I want to use.  The first, L{\"u}ck's determinant property, guarantees that a matrix of elements in the integer ring of a discrete group is of determinant class.  The second, Linnell's $L^2$-property, concerns $0$ as it occurs in the spectrum of such operators.  It allows after nondegeneracy checks, a computation of the rank of the associated derivative in the case of a single group relation.  

Throughout this section $\Gamma$ denotes a discrete, countable group.

\subsection{Determinant Conjecture}  Denote by $\mathbb Z \Gamma \subset L(\Gamma)$ the integral ring generated by the the group unitaries in the left regular representation, $\{u_g \in B(\ell^2(\Gamma)): g \in \Gamma\}$.  L{\"u}ck's determinant conjecture states that for any $x \in M_k(Z\Gamma) \subset M_k(L(\Gamma))$,

\begin{eqnarray*}  \text{det}_{FKL}(x) \geq 1.
\end{eqnarray*}

\noindent I'll say that $\Gamma$ has L{\"u}ck's property if the above inequality holds for any $x \in M_k(\mathbb Z \Gamma)$.  

In \cite{luck} L{\"u}ck showed that the conjecture is true for all residually finite groups (where he used $\det_{FKL}$ to compute the $L^2$-Betti numbers of residually finite groups in terms of the ordinary Betti numbers in \cite{luck}).  This was extended in \cite{schick} to residually amenable groups.  More recently it was shown in \cite{es} that all sofic groups have L{\"u}ck's determinant property (see also the operator algebraic approach in \cite{bs}).  

\begin{remark}  As far as I know, there are no known examples of non-sofic discrete groups.  Amenable, residually finite, and residually amenable discrete countable groups are all known to be sofic.
\end{remark}

\begin{remark} Sofic groups are closed under a number of natural group operations including inverse and direct limits, subgroups, free products, amenable extensions, and direct products (\cite{es}).  It is also easy to see that if $\Gamma$ is sofic, then so is $\Gamma^{op}$, the opposite group of $\Gamma$.  It follows that $\Gamma$ is sofic iff $\Gamma \times \Gamma^{op}$ is sofic.\end{remark}

By Corollary 6.4 if $\Gamma$ has L{\"u}ck's property, then every element of $\mathbb Z \Gamma$ has geometric decay.  Recall that $\mathfrak{A}_n$ denotes the universal unital, complex $*$-algebra on $n$ indeterminates. Applying the unitary calculus developed in Section 3 (Remark 3.23) yields the following:

\begin{lemma} Suppose $\Gamma$ is a countable, discrete sofic group, $F$ is a $p$-tuple of $*$-monomials in $\mathfrak{A}_n$, and $g_1,\ldots, g_n \in \Gamma$.  If $X = \{u_{g_1},\ldots, u_{g_n}\} \subset L(\Gamma)$, $F(X)=I$, and $\alpha = \text{Nullity}(D^uF(X))$, then $X$ is $\alpha$-bounded.
\end{lemma}

\begin{proof}  By Corollary 6.12 it suffices to show that $D^uF(X)$ has geometric decay.  By Definition 3.22 and Remark 3.24, 
\begin{eqnarray*} |D^uF(X)|^2 & = & D^uF(X)^*D^uF(X) \\
                                                & \in & M_n(\mathbb Z(\Gamma \times \Gamma^{op})) \\
                                                & \subset & M_n(L(\Gamma) \otimes L(\Gamma)^{op})\\
                                                & = & M_n(L(\Gamma \times \Gamma^{op})).\\
\end{eqnarray*}                                
\noindent By the remark above $\Gamma \otimes \Gamma^{op}$ is sofic since $\Gamma$ is.  Thus, $0 < 1 \leq \det_{FKL}(|D^uF(X)|^2)$.  $|D^uF(X)|^2$ has geometric decay and this implies that $|D^uF(X)|$ does as well (Lemma 6.3), whence by definition, $D^uF(X)$ does.
\end{proof}

\subsection{Linnell's $L^2$-Property}

Linnell posed the following in \cite{l}:

\begin{linnell} If $\Gamma$ is a torsion-free discrete group, then any nonzero $x \in \mathbb C\Gamma$ is injective on $\ell^2(\Gamma)$.
\end{linnell}

I will say that a countable, discrete group $\Gamma$ has Linnell's $L^2$-property if it satisfies the conclusion of the conjecture above.  

Linnell's conjecture is related to other open problems/conjectures.  An older and closely related statement is Kaplansky's conjecture:

\begin{kaplansky} If $\Gamma$ is a torsion-free discrete group, then for any nonzero $x,y \in \mathbb C\Gamma$, $xy \neq 0$.
\end{kaplansky}

Both conjectures are naturally connected to Atiyah's conjecture.  The interested reader can read more about this conjecture as well as progress on the two above conjectures in a variety of places, e.g., \cite{luckbook}.

Linnell showed in \cite{l} that a left orderable group always possess Linnell's $L^2$-property.  A group $\Gamma$ is said to be orderable if there exists a strict linear ordering $<$ on $\Gamma$ which preserves left and right multiplication, i.e., for any $a, b, c \in \Gamma$ if $a<b$, then $ac < bc$ and $ca < cb$.  One can relax the conclusion in the definition so that one only assumes $ac<bc$ in which case $\Gamma$ is right ordered, or $ca < cb$ in which case $\Gamma$ is said to be left-ordered.  A group is right-ordered iff it is left-ordered.  Orderability assumes one ordering can fulfill both rolls.  Clearly any ordered group is both left and right ordered.  Also note that any left or right ordered group is necessarily torsion-free.

Here are some examples of left orderable groups, which will by the discussion above, satisfy Linnell's $L^2$-property:

\begin{example} Residually torsion-free nilpotent groups are left orderable.  Notice that this subclass is contained in the class of residually amenable groups and thus they are all sofic.    Examples of residually torsion-free nilpotent groups include elementary, torsion-free amenable groups \cite{l0}, free groups, the braid groups, and more recently, Hydra groups (\cite{dr}) and some of their generalizations (\cite{bm}).
 \end{example}
 
\begin{example}  By a result independently obtained by Brodski{\u i} and Howie (\cite{b},\cite{h}), all torsion free, one-relator groups are locally indicable.  By \cite{bh} locally indicable groups are left orderable (\cite{bh}).  Thus, all torsion free one-relator groups are left orderable.
\end{example}

\begin{example} Left orderable groups are closed under subgroups, the opposite operation, free products, direct products, and extensions. 
\end{example}

\begin{remark} If $\Gamma$ satisfies Linnell's $L^2$-property, then $\Gamma$ must be torsion-free.  In particular if $g \in \Gamma$ is not the group identity and $u_g \in L(\Gamma)$ is the canonical unitary associated to $g$, then $u_g$ is Haar (i.e., has the same spectral distribution as Haar measure on the unit circle of $\mathbb C$).  It is easy to see from this observation and the entropy formula for a single selfadjoint in \cite{v1} that the real part of $u_g$ has finite free entropy.   From this and Lemma 6.14 it follows that if $X$ is a finite tuple of unitaries associated to group elements and $X$ is $1$-bounded, then $X$ is strongly $1$-bounded.

\end{remark}

\begin{proposition} Suppose $\Gamma$ is a group which satisfies Linnell's $L^2$-property and $X=\{u_{g_1},\ldots, u_{g_n}\}$ with $g_1,\ldots, g_n \in \Gamma$.  If there exists a nonempty, reduced word $f$ on $n$ letters such $f(g_1,\ldots, g_n)=e_{\Gamma}$ and the $i$th indeterminant appears in $f$, then there exists a $*$-monomial $w \in \mathfrak A_n$ such that $w(X)=I$ and $\partial^u_iw(X)$ is injective.  
\end{proposition}

\begin{proof}   Denote by $\ell$ the length function on $\mathbb F_n$, i.e., for $g \in \mathbb F_n$ $\ell(g)$ is the length of $g$ represented as a reduced word.  Set
\begin{eqnarray*}
m = \min\{\ell(w): w(g_1,\ldots,g_n)=e_{\Gamma}, w \neq e_{\mathbb F_n}, w \in \mathbb F_n\}.
\end{eqnarray*}
$m \in \mathbb N$ is well-defined by the existence of $f$.  Pick a $v \in \mathbb F_n$ such that $\ell(v)=m$, $v(g_1,\ldots,g_n)= e_{\Gamma}$, $v \neq e_{\mathbb F_n}$. There exists a $*$-monomial $w \in \mathfrak{A}_n$ such that $\ell_0(w)=m$, for some $i$, $X_i$ or $X_i^*$ appears in $w$, and $w(X)=w(u_{g_1},\ldots, u_{g_n}) =I$ where $\ell_0$ is the length function defined on $*$-monomials in $\mathfrak{A}_n$ (equivalently $\ell_0$ is the degree of the $*$-monomial).  This is obtained by taking $v$  (which is a sequence in $n$ indeterminates and labelled inverses), replacing the ith indeterminate with $X_i$, the inverses with adjoints, and concatenating the terms in the sequence to produce a product in $\mathfrak{A}_n$ (recall that $X_1,\ldots, X_n$ are the canonical generators of $\mathfrak{A}_n$).  By permuting the indices I can assume without loss of generality that $i=1$.  To complete the proof I have to show that $\partial^u_1w(X)$ is injectivity.

Set $B=\{X_2,\ldots, X_n\}$ and write $w = w_1X_1^{j_1} \cdots w_p X_1^{j_p} w_{p+1}$ where $p \in \mathbb N$ (since $X_1$ appears in $w$), $w_1,\ldots,w_{p+1}$ are (possibly empty) words in the $*$-semigroup generated by $B$, $j_1,\ldots, j_p \in \{*,1\}$, and for any $1 \leq k \leq p-1$, $j_k \neq j_{k+1}$ only if $w_k \neq I$.  This expression of $w$ is possible since $w$ is derived from $v$ and $v$'s length.  Written in this way, it follows that $1 \leq m = \ell_0(w) = p +\sum_{k=1}^{p+1} \ell_0(w_k)$.  From the unitary calculus formula in Lemma 3.20 and Remark 3.24 $(\partial_1^uw)(X)$ equals
\begin{eqnarray*} \sum_{k=1}^p (-1)^{\delta_{j_k, *}}((u_{g_1}^{\delta_{*, j_k}})^* w_{k+1}(X)  \cdots w_p(X)  u_{g_1}^{j_p} w_{p+1}(X))^* \otimes  ((u_{g_1}^{\delta_{*, j_k}})^* w_{k+1}(X)  \cdots w_p(X)  u_{g_1}^{j_p} w_{p+1}(X))^{op}.\\ 
\end{eqnarray*}
Consider the operator $T$ obtained from the right hand side of the elementary tensors above:
\begin{eqnarray*} T & = & \sum_{k=1}^p  (-1)^{\delta_{j_k, *}} \cdot ((u_{g_1}^{\delta_{j_k, *}})^* w_{k+1}(X)  \cdots w_p(X)  u_{g_1}^{j_p} w_{p+1}(X))^{op} \\
                                & \in & \mathbb Z(\Gamma^{op}) \\
                                & \subset & \mathbb C(\Gamma^{op}). \\
\end{eqnarray*}
It is straightforward to check that $T$ has the same noncommutative $*$-moments as $(\partial_1^uw)(X)$.
Moreover, if two operators $a, b$ in respective tracial von Neumann algebras have the same noncommutative $*$-moments and $a$ is injective, then so is $b$.  Applying this observation for $a=T$ and $b=\partial^u_1w(X)$, if I can show that $T$ is injective, then this will show the injectivity of $\partial^u_1w(X)$ and complete the proof.

Towards this end, notice that the terms in the expansion of $T$ above are pairwise orthogonal (as elements in $L(\Gamma^{op}) \subset \ell^2(\Gamma^{op})$).  Indeed, if this is not the case then for some $1 \leq r < s \leq p$, $j_r = j_s$ and
\begin{eqnarray*}
&((u_{g_1}^{\delta_{j_r, *}})^* w_{r+1}(X)  \cdots w_p(X)  u_{g_1}^{j_p} w_{p+1}(X)) = ((u_{g_1}^{\delta_{j_s, *}})^* w_{s+1}(X)  \cdots w_p(X)  u_{g_1}^{j_p} w_{p+1}(X)).&\\  
\end{eqnarray*}
Thus, 
\begin{eqnarray*}  ((u_{g_1}^{\delta_{j_r, *}})^* w_{r+1}(X)  \cdots w_s(X)  u_{g_1}^{j_s} & = &  (u_{g_1}^{\delta_{j_s, *}})^*.\\
\end{eqnarray*}
Since $j_r =j_s$, this implies 
\begin{eqnarray*}
w_{r+1}(X)  \cdots w_s(X)  u_{g_1}^{j_s} & = & I.\\
\end{eqnarray*}
There is a nontrivial $*$-monomial $g$ of length $1 \leq s-r+ \sum_{k=r+1}^s \ell_0(w_k) < m$ such that $g(X)=I$; $g$ yields a reduced word $w_0 \in \mathbb F_n$ for which $\pi(w_0)=e_{\Gamma}$ and $1\leq \ell(w_0) = s-r+ \sum_{k=r+1}^s \ell_0(w_k) < m$.  This is a contradiction.
It follows that $T$ is nonzero as it is a nonempty sum of orthogonal, nonzero vectors.  $\Gamma$ has Linnell's $L^2$-property iff $\Gamma^{op}$ has Linnell's $L^2$-property.  Thus, $\Gamma^{op}$ has Linnell's $L^2$-property.  Since $T \in \mathbb C(\Gamma^{op})$ is nonzero, it is injective and this finishes the proof.
\end{proof}

\begin{remark} The one place where I used the assumption that $\Gamma$ has Linnell's $L^2$-property was in concluding that nontriviality of $\partial_1^uf(X)$ implies that it's injective.  This seems excessive as it invokes a global group property to just one single operator (obtained through partial differentiation) which is naturally and  concretely expressed in terms of the group relation.  
\end{remark}

\begin{corollary}
Suppose $\Gamma$ is a sofic group which satisfies Linnell's $L^2$-property, has a finite generating tuple $\{g_1,\ldots, g_n\}$, and $X=\{u_{g_1},\ldots, u_{g_n}\} \subset L(\Gamma)$.  $\Gamma \not\simeq \mathbb F_n$ iff $X$ is $(n-1)$-bounded.
\end{corollary}

\begin{proof} Suppose $\Gamma \not\simeq \mathbb F_n$.  Denote by $a_1,\ldots, a_n$ the canonical set of generators for $\mathbb F_n$.  By universality there exists a $*$-homomorphism $\pi: \mathbb F_n \rightarrow \Gamma$ such that $\pi(a_i) = g_i$.  Since $\Gamma \not\simeq \mathbb F_n$ there exists a $b \in \ker \pi$ which is not the identity.   $b$ yields a non-empty reduced word $f$ in the $g_i$ such that $f(g_1,\ldots g_n) = e_{\Gamma}$.  Fix some index $1 \leq j \leq n$ such that the $j$th indeterminate appears in $f$.   By Proposition 7.5 there exists a nontrivial $*$-monomial $w \in \mathfrak{A}_n$ such that $w(X) =I$ and $\partial_i^uw(X)$ is injective.   By Lemma 6.13 $\alpha = \text{Nullity}(D^uw)(X) = n-1$.  By Lemma 7.3 $X$ is $n-1$ bounded.

Conversely, suppose $X$ is $(n-1)$-bounded.  Recall that $\delta_0()$ is a $*$-algebra invariant and $\delta_0(G)$ is well-defined for a finitely generated discrete group (Section 2.3).  By the computation of the free entropy dimension for freely independent self-adjoint variables in \cite{v2}, $\delta_0(\mathbb F_n) = n$ for $n \in \mathbb N$.  So if $\Gamma \simeq \mathbb F_n$, then $n-1=\delta_0(X)=\delta_0(\Gamma) = \delta_0(\mathbb F_n) = n$ which is preposterous.  $\Gamma \not\simeq \mathbb F_n$ as desired.
\end{proof}

\begin{corollary} Suppose $\Gamma$ is a sofic group with two generators, satisfies Linnell's $L^2$-property, and $\Gamma \neq \{0\}$. The following are equivalent:
\begin{enumerate}[(1)]
\item $\Gamma \not\simeq \mathbb F_2$. 
\item $L(\Gamma) \not\simeq L(\mathbb F_2)$.
\item $\delta_0(X) = 1$ for any finite set of generators $X$ for $L(\Gamma)$.
\item $L(\Gamma)$ is strongly $1$-bounded.
\end{enumerate}
\end{corollary}

\begin{proof} $ $

(1) $\Rightarrow$ (4): If (1), then by Corollary 7.7, $\Gamma$ is $1$-bounded.  By Remark 7.4 $L(\Gamma)$ is strongly $1$-bounded.

(4) $\Rightarrow$ (3): Suppose $X$ is any finite set of generators for $L(\Gamma)$.  By \cite{j3} and Proposition 2.5 $X$ is $1$-bounded.  It follows (Section 2) that $\delta_0(X) \leq 1$.  Since $\Gamma$ is sofic, by \cite{gs} $L(\Gamma)$ is embeddable into an ultraproduct of the hyperfinite $\mathrm{II}_1$-factor.  Moreover, $\Gamma$ being torsion free and not equal to the trivial identity group, $L(\Gamma)$ is diffuse.  So by \cite{j0} and Proposition 2.5, $\delta_0(X) \geq 1$.  Thus, $\delta_0(X)=1$.

(3) $\Rightarrow$ (2): By contradiction if $L(\Gamma) \simeq L(\mathbb F_2)$, then the assumed isomorphism identifies the canonical traces  (since $L(\mathbb F_2)$ has a unique tracial state).  Hence, there exists a 2-tuple $X$ of freely independent semicirculars (w.r.t. the canonical trace on $L(\Gamma)$) which generates $L(\Gamma)$.  By \cite{v1} $\delta_0(X) = \delta_0^{sa}(X)=2 \neq 1$ which violates (3). 

(2) $\Rightarrow$ (1): If $\Gamma \simeq \mathbb F_2$, then $L(\Gamma) \simeq L(\mathbb F_2)$.  Take the contrapositive. 
\end{proof}

 Recall that any nontrivial element $w$ (i.e., non-identity element) in a free group can be written as $w=v^m$ for some maximal $m \in \mathbb N$ and that subject to this maximality condition, $v$ is unique. If $m >1$ then $w$ is said to be a proper power.   Recall also that a positive element in a free group on $n$ generators $g_1,\ldots, g_n$, is an element of the form $g_{j_1} \cdots g_{j_d}$ where $d \in \mathbb N$, $1 \leq j_1, \ldots, j_d \leq n$.  
A positive $k$-relator group is simply a group with a presentation on finitely many generators and $k$ positive relators.
 
 \begin{corollary} If $\Gamma$ is a sofic, one-relator group on $n$ generators whose relator is a nontrivial, non-proper power, then $\Gamma$ is $(n-1)$-bounded.   If $\Gamma$ is a one-relator group on $n$ generators whose relator word is nontrivial, positive and not a proper power, then $\Gamma$ is $(n-1)$-bounded.
 \end{corollary}
 
 \begin{proof} For the first claim, by \cite{kms} the relator word is not a proper power iff $\Gamma$ is torsion free.  $\Gamma$ is a torsion-free one relator group so by Example 7.2 $\Gamma$ is left orderable and thus by \cite{l} satisfies Linnell's $L^2$-property.  $\Gamma$ is a sofic group which satisfies Linnell's $L^2$-property.   Since $\mathbb F_n$ is Hopfian, $\Gamma \not\simeq \mathbb F_n$.  By Corollary 7.7, $\Gamma$ is $(n-1)$-bounded.  
 
To establish the second claim by the first claim it is enough to show that the group is sofic.  By \cite{baum} a one-relator group whose relator is a positive word is residually solvable, whence, residually amenable, and thus sofic by \cite{es}.
 \end{proof}
 
 \begin{corollary} If $\Gamma$ is a sofic, one-relator group on 2 generators whose relator is nontrivial and not a proper power, then $L(\Gamma)$ is strongly 1-bounded.  In particular if $\Gamma$ is a positive one-relator group on 2 generators whose relator is nontrivial and not a proper power, then $L(\Gamma)$ is strongly $1$-bounded.
 \end{corollary}
 
\begin{proof}  For the first claim, by Corollary 7.9 $\Gamma$ is $1$-bounded. It is easy to show that $\Gamma \neq \{0\}$ and by Remark 7.4 this implies $L(\Gamma)$ is strongly $1$-bounded.   The second claim follows from the first provided soficity can be established.  As in Corollary 7.9 this follows from the residual solvability of positive relator groups established in \cite{baum} and \cite{es}.
 \end{proof}

\begin{remark}   By \cite{murasugi} it turns out that the center of a one relator group $\Gamma$ on $2$ generators has the following dichotomy: it is either trivial (e.g., the Baumslag-Solitar Groups BS(n,m) when $|n| \neq |m|$ were shown in \cite{ys} to be i.c.c.), or a copy of $\mathbb Z$.  The result also shows that when the number of generators is greater than $2$, then the center is always trivial.
\end{remark}

\begin{lemma} Suppose $\Gamma$ is a group with generators $a_1,\ldots, a_n$ and $f_1,\ldots, f_{n-1}$ are reduced, nonempty words where the $i$th indeterminate appears in $f_i$ and the first $i-1$ indeterminates do not appear in $f_i$.  If $\Gamma$ is sofic and has Linnell's $L^2$-property, then $L(\Gamma)$ is strongly $1$-bounded.    
\end{lemma}

\begin{proof}  Denote by $X$ the $n$-tuple of canonical group unitaries of $L(\Gamma)$ associated to $a_1,\ldots, a_n$.  By Proposition 7.5 and the fact that $\Gamma$ satisfies Linnell's $L^2$-property for each $i$ I can associate a $*$-monomial $w_i \in \mathfrak A_n$ such that $w_i(X)=I$ and $\partial_i^u w(X)$ is injective.  Moreover, from the proof of Proposition 7.5, it's easy to see that the indeterminates $X_1,\ldots, X_{i-1}$ do not appears in the expansion of $w_i$.  Thus,
\begin{eqnarray*}
D^uF(X) & = & \begin{bmatrix}
\partial^u_1w_1(X) & \cdots & \cdots & \cdots & \partial^u_n w_1(X) \\ 
0 & \ddots & & & \vdots \\
\vdots & \ddots & \ddots & & \vdots \\
0 & \cdots & 0 & \partial^u_{n-1}w_{n-1}(X) &\partial^u_n w_{n-1}(X) \\
\end{bmatrix} \\
& \in & M_{n-1,n}(L(\Gamma) \otimes L(\Gamma^{op})).\\
\end{eqnarray*}
By Lemma 6.16 $\text{Nullity}(D^uF(X)) = n-(n-1)=1$.  Because $\Gamma$ is sofic, by Lemma 7.3 $X$ is $1$-bounded.  By Remark 7.4 $X$ is strongly $1$-bounded, whence $L(\Gamma)$ is strongly $1$-bounded.
\end{proof}

\appendix
\section{St. Raymond's Estimates for the Schatten Norm Balls}

In this technical section I want to review St. Raymond's asymptotic estimates \cite{sr} for the unit balls of the $k \times k$ matrices w.r.t. the Schatten $p$-norms and quasi-norms.  I will then apply this to some computations on products of $L^2$ norm elements with proper rank.

In order to keep the notation somewhat consistent with St. Raymond's work, I will make use of the unnormalized as well as the normalized trace.  Thus, for any $0 < p < \infty$ and $x \in M_k(\mathbb C)$ define the $p$ (quasi) norm on $M_k(\mathbb C)$ by
\begin{eqnarray*}
\|x\|_{Tr,p,k} = (Tr(|x|^p))^{1/p}
\end{eqnarray*}
where $Tr$ is the unnormalized trace on $M_k(\mathbb C)$ so that $Tr$ of the identity in $M_k(\mathbb C)$ is $k$.  The normalized $p$ (quasi) norm on $M_k(\mathbb C)$ is denoted by
\begin{eqnarray*}
\|x\|_{tr,p,k} = (tr(|x|^p))^{1/p}
\end{eqnarray*}
where $tr$ is the normalized trace on $M_k(\mathbb C)$.  Notice that in this section I'm no longer using the $tr_k$ notation set forth in section 2; this is to reduce eye strain on multi-indices in the following estimates.  The dependence on the size of the matrices will be made clear as another trailing subscript.  For any $0< p, r < \infty$, $k \in \mathbb N$, denote by $B_{Tr,p,r,k}$ the $r$-ball of $M_k(\mathbb C)$ w.r.t. $\|\cdot \|_{Tr,p,k}$ centered at the origin and by $B_{tr,p,r,k}$ the $r$-ball of $M_k(\mathbb C)$ w.r.t. $\|\cdot \|_{tr,p,k}$ centered at the origin.  For any $x \in M_k(\mathbb C)$, $\|x\|_{Tr,p,k} = k^{1/p} \cdot \|x\|_{tr,p,k}$.  It follows that $B_{Tr,p,k, k^{1/p}r } = B_{tr,p,k,r}$.

$M_k(\mathbb C)$ becomes a real Hilbert space of dimension $2k^2$ w.r.t. the inner product $\langle x, y \rangle = \text{Re Tr}(y^*x)$.  Throughout this section denote by $\text{vol}$ Lebesgue measure w.r.t. this identification of $M_k(\mathbb C)$ with $\mathbb R^{2k^2}$ (this is the unnormalized scaling of Lebesgue measure, and not the normalized one used in the definition of free entropy in Subsection 2.3).  For any $k$ denote by $v_k$ the volume of the unit ball in $\mathbb R^k$, i.e.,
\begin{eqnarray*}
v_k = \frac{\pi^{k/2}}{\Gamma(\frac{k}{2}+1)}.
\end{eqnarray*}

Denote by $D^+_{k,\leq}$ the set of diagonal matrices $x$ where $0 \leq x_{11} \leq \cdots \leq x_{kk}$ and by $U_k$ the $k\times k$ unitaries.   For an element $w \in M_k(\mathbb C)$, denote by $s(w) \in \mathbb R^k$ the unique sequence consisting of $|w|$'s eigenvalues, listed in nondecreasing order, i.e., the singular values of $w$.   Consider as in \cite{sr} the map $\Psi_k : U_k \times U_k \times D^+_{k,\leq} \rightarrow M_k(\mathbb C)$ defined by $\Psi_k(u,v,d) = udv$.  For any $S \subset D^+_{k, \leq}$ denote by $\theta(S)$ the set of all matrices of the form $udv$ where $d \in S$, i.e., $\theta(S)$ is the set of matrices whose absolute value or singular value decomposition agree with some diagonal element of $S$.   St. Raymond showed that for any subset $E \subset \mathbb D^+_{k, \leq}$,
\begin{eqnarray*}
\text{vol}(\theta(E)) & = & \text{vol}(\{ udv : u,v\in U_k, d \in E\}) \\
                      & = & c_k \cdot \int_{ s(E) }  \varphi_k(\lambda_1,\ldots, \lambda_k) \cdot \lambda_1 \cdots \lambda_k \, d\lambda_1 \cdots d\lambda_k\\
\end{eqnarray*}
where $c_k = (2\pi)^{-k} \cdot (\Pi_{j=1}^k 2 j v_{2j})^2$ and $\varphi_k(\lambda_1, \ldots, \lambda_k) = \left (\Pi_{1 \leq i < j \leq k} (\lambda_i^2 - \lambda_j^2)^2 \right)$.  There are similar change of variables formulae in random matrix theory that one can also invoke here (e.g., \cite{m}).

He used this formula to then bound $\text{vol}(B_{Tr,k,p,1})$, $1 \leq p \leq \infty$ in terms of a constant $C_p>0$ dependent only on $p$.  Indeed in Corollary 8 of \cite{sr} he showed 
\begin{eqnarray*}
\text{vol}\left (B_{Tr,k,p,1} \right)^{1/2k^2} \sim \frac{C_p}{k^{(1/2+1/p)}}.
\end{eqnarray*}
\noindent as $k \rightarrow \infty$.  While his initial statement only dealt with values $p \in [1,\infty)$, the constant $C_p$ is expressed in terms of a function $\Delta(p)$ for which he proves several key properties for all $0 <p <\infty$.  While I haven't checked the details, I believe the above asymptotic equivalence is true for all $p \in (0, \infty]$.

In what follows below however, I will only need the asymptotic upper bound, i.e., for any $p \in (0,\infty]$ the existence of a constant $C_p > 0$ such that for all $k$,
\begin{eqnarray*}
\text{vol}\left (B_{Tr,k,p,1} \right)^{1/2k^2} \leq \frac{C_p}{k^{(1/2+1/p)}}.
\end{eqnarray*}
\noindent Here I will indicate the relevant parts of \cite{sr} that can be used to derive this.  Define as in \cite{sr}, for any $q \in (0, \infty)$,
\begin{eqnarray*}
\Delta_k(q) = \sup_{0 \leq \lambda_1 \leq \cdots \leq \lambda_k} k^{1/q} \cdot \frac{\left(\Pi_{1 \leq i < j \leq k} |\lambda_j - \lambda_i| \right)^{\frac{2}{k(k-1)}} } {\left(\sum_{i=1}^k \lambda_i^q \right)^{1/q} }
\end{eqnarray*}
By Lemma 5 of \cite{sr} $\langle \Delta_k(q) \rangle_{k=1}^{\infty}$ is a monotonically decreasing sequence with $\Delta(q) = \lim_{k \rightarrow \infty} \Delta_k(q) >0$.  If $\Omega_p = \{(\lambda_1, \ldots, \lambda_k): 0 < \lambda_1 < \cdots < \lambda_k, \sum_{j=1}^k \lambda_j^p <1\}$, then using the fact that $\Omega_p \subset \Omega_{\infty}$ for all $p \in (0,\infty)$
\begin{eqnarray*}
\int_{\Omega_p} \varphi_k(\lambda) \lambda_1 \cdots \lambda_k \, d\lambda & \leq & \int_{\Omega_p} \varphi_k(\lambda) \, d\lambda \\ & \leq & \text{vol}_{\mathbb R^k}(\Omega_p) \cdot k^{-2k(k-1)/p}(\Delta_k(p/2))^{k(k-1)} \\ & \leq & \frac{1}{k!} \cdot k^{-2k(k-1)/p}(\Delta_k(p/2))^{k(k-1)}. \\
\end{eqnarray*} 
\noindent Combining this with the change of variables volume formula above and the asymptotic expression in Theorem 2 of \cite{sr}, it follows that 
\begin{eqnarray*}
\text{vol}(B_{Tr,k,p,1})^{1/2k^2} & \leq & c_k^{1/2k^2} \cdot \left( \frac{1}{k!} \cdot k^{-2k(k-1)/p}(\Delta_k(p/2))^{k(k-1)}  \right)^{1/2k^2} \\ & \sim & \sqrt{\frac{\pi e^{3/2}}{k}} \cdot k^{-1/p} \cdot \sqrt{\Delta(p/2)} \\ 
& \leq & \frac{\sqrt{\pi e^{3/2}}}{k^{1/2+ 1/p}} \cdot \sqrt{\Delta(p/2)} \\ 
\\ 
& \leq & \frac{\sqrt{\pi e^{3/2}}}{k^{1/2+ 1/p}} \cdot 4^{1/p} \\ 
\end{eqnarray*}
as claimed.

Rephrasing this in terms of the normalized $L^p$-norms gives the following:

\begin{corollary} For any $p \in (0, \infty)$ there exists a constant $C_{p,2}>1$ such that
\begin{eqnarray*}
\text{vol}(B_{tr,k,p,1}) \leq (C_{p,2})^{2k^2} \cdot \text{vol}(B_{tr,k,2,1}).\\
\end{eqnarray*}
\end{corollary}
\begin{proof} From the above for $k$ sufficiently large,
\begin{eqnarray*}
 \text{vol}(B_{tr,k,p,1})^{1/2k^2} & = & \text{vol}(B_{Tr,k,p,k^{1/p}})^{1/2k^2} \\ 
  & = & k^{1/p} \cdot \text{vol}(B_{Tr,k,p,1})^{1/2k^2} \\ 
 & \leq & 2\cdot k^{1/p} \cdot \frac{\sqrt{\pi e^{3/2}}}{k^{1/2 + 1/p}} \cdot 4^{1/p} \\
 & = & 2 \cdot \frac{\sqrt{\pi e^{3/2}}}{k^{1/2}} \cdot 4^{1/p}. \\
\end{eqnarray*}
\noindent On the other hand, it follows from a direct computation in $\mathbb R^{2k^2}$ or Corollary 8 of \cite{sr}, that $\text{vol}(B_{tr,k,2,1})^{1/2k^2} \sim \frac{\sqrt{\pi e}}{k^{1/2}}$ as $k \rightarrow \infty$.  The existence of $C_{p,2}$ follows.
\end{proof}

Recall from Subsection 5.4 that $E(r,p,k,d) \subset M_k(\mathbb C)$ consists of all $k \times k$ complex matrices $x$ of rank no larger than $d$ such that $\|x\|_{tr, k, p} \leq r$.  $\ell^p(k)$ will denote the space of sequences of length $k$ with values in $\mathbb C$, endowed with the usual $\ell^p$-(quasi)norm.  Denote by $D(r,p,k,d,\epsilon)$ the set of all diagonal matrices $x \in M_k(\mathbb C)$ whose diagonals are nondecreasing sequences of nonnegative numbers such that $\| s(x) \cdot 1_{\{1,\ldots, k-d\}} \|_{\ell^2(k)} \leq \epsilon$ and $\|s(x)\|_{\ell^p(k)} \leq r$.  

In Lemma A.2 below $\mathcal N_{\epsilon}(\cdot)$ is taken w.r.t. the normalized $\|\cdot \|_{tr,2,k}$ norm on $M_k(\mathbb C)$ (or in the notation of the other parts of this paper, the $\|\cdot\|_2$ norm on $M_k(\mathbb C)$).  Recall that for any $S \subset D^+_{k, \leq}$, $\theta(S)$ denotes the set of all complex $k \times k$ matrices whose singular value decomposition agree with some diagonal element of $S$

\begin{lemma} If $p \in (0,1)$, then there exists a constant $Q_p>1$ dependent only on $p$ such that for any $\epsilon >0$, 
\begin{eqnarray*}
\mathcal N_{\epsilon}(E(r,p,k,d)) & \subset & \theta(D(Q_p(r+\epsilon)k^{1/p},p,k,d,\epsilon k^{1/2})).\\
\end{eqnarray*}
\end{lemma}
\begin{proof} Suppose $z \in \mathcal N_{\epsilon}(E(r,p,k,d))$.  There exist $x \in E(r,p,k,d)$ and $y \in M_k(\mathbb C)$ such that $\|y\|_{tr, k,2} < \epsilon$ and $z=x+y$.  By the Weilandt-Hoffman Inequality,
\begin{eqnarray*}
\| s(z) - s(x) \|_{\ell^2(k)} & \leq & \| s(z-x) \|_{\ell^2(k)} \\
                                                             & = & \| s(y) \|_{\ell^2(k)} \\
                                                             & < & \epsilon k^{1/2}. \\   
\end{eqnarray*}
Because $x$ has rank no larger than $d$, $s(x)_i =0$ for all $1 \leq i \leq k-d$.  Define $a \in \ell^2(k)$ by $a(i) = s(z)_i$  for $1 \leq i \leq k-d$ and $a(i) = 0$ for $k-d <i \leq k$.
Obviously for any $p$, $\|a\|_{\ell^p(k)} \leq \| s(z) \|_{\ell^p(k)}$.  From the above,
\begin{eqnarray*}
\|a\|_{\ell^2(k)} & = & \|a \cdot 1_{\{1,\ldots, k-d\}} \|_{\ell^2(k)} \\
                        & = & \| (s(z) - s(x)) \cdot 1_{\{1,\ldots, k-d\}}\|_{\ell^2(k)} \\
                        & \leq & \| s(z) - s(x) \|_{\ell^2(k)} \\
                        & < & \epsilon k^{1/2}.\\
\end{eqnarray*}
Thus, $\|s(z) \cdot 1_{\{1,\ldots, k-d\}} \|_{\ell^2(k)} = \|a\|_{\ell^2(k)} < \epsilon k^{1/2}$.  

Using the fact that $\|\cdot \|_{tr,p,k}$ are $Q_p$-quasi-norms (in fact $Q_p = 2^{(1/p-1)}$) with $Q_p$ dependent only on $p$,
\begin{eqnarray*}
\|z\|_{tr,p,k} & = & \|x+y\|_{tr,p,k} \\
                   & \leq & Q_p (\|x\|_{tr,p,k} + \|y\|_{tr,p,k}) \\
                   & \leq & Q_p (r + \|y\|_{tr,2,k}) \\
                   & \leq & Q_p (r+\epsilon) \\
\end{eqnarray*}
\noindent Thus, $\|s(z) \|_{\ell^p(k)} \leq Q_p (r+\epsilon)k^{1/p}$.  The first paragraph shows that the first $k-d$ singular values of $z$ have $\ell^2$ norm no greater than $\epsilon k^{1/2}$.  By definition  $z \in \theta(D(Q_p(r+\epsilon)k^{1/p},p,k,d,\epsilon k^{1/2}))$. 
\end{proof}

It remains to find a suitable upper bound on  the volume of $\theta(D(rk^{1/p},p,k,m,\epsilon k^{1/2}))$ where $p <1$.   I'll need an asymptotic computation and an inequality on the density of $\varphi_k$ over off-block elements.

\begin{lemma} $\lim_{k \rightarrow \infty} k^{-2} \cdot \log\left(\frac{v_{2k^2} \cdot k^{k^2}}{ c_k}\right) = -\frac{1}{2}$.
\end{lemma} 

\begin{proof}
By Theorem 4 of \cite{sr}, if $\tilde{P_k} = (\Pi_{j=1}^k 2j v_{2j})$, then an application of Stirling's formula shows as $k \rightarrow \infty$,
\begin{eqnarray*}
k^{-2} \cdot \log(\tilde{P_k}) & = & 2 \cdot \log(\tilde{P_k}^{1/2k^2}) \\
                                     & = & \frac{1}{2} \log(2 \pi e^{3/2}) -  \frac{\log(2k)}{2} + O\left(\frac{\log k}{k} \right). \\                                   
\end{eqnarray*}
Thus,
\begin{eqnarray*}
\lim_{k \rightarrow \infty} k^{-2} \log c_k - (\log(\pi e^{3/2}) - \log k) & = & \lim_{k \rightarrow \infty}  k^{-2} \log \left((2\pi)^{-k} \cdot \tilde{P_k}^2 \right) - \log(\pi e^{3/2}) +  \log(k)\\ & = & 
 \lim_{k \rightarrow \infty}  2 \cdot k^{-2} \log \left(\tilde{P_k}\right) - \log(\pi e^{3/2}) +  \log(k)\\
 & = & \lim_{k \rightarrow \infty} \log(2 \pi e^{3/2}) - \log(2k) - \log(\pi e^{3/2}) +  \log(k) \\
 & = & 0.\\
\end{eqnarray*}
Using Stirling's formula again
\begin{eqnarray*}
k^{-2} \cdot \log v_{2k^2} & = & k^{-2} \cdot \log \left(\frac{\pi^{k^2}}{\Gamma(k^2+1)}\right) \\ & \sim & \log(\pi) - k^{-2} \cdot \log \left[ \left(\frac{k^2}{e} \right)^{k^2} \cdot \sqrt{2\pi k} \right] \\ & = &  \log(\pi e) - 2 \log k - k^{-2} \cdot \log(\sqrt{2\pi k}),\\
\end{eqnarray*}
whence $\lim_{k \rightarrow \infty} k^{-2} \cdot \log v_{2k^2} - (\log \left( \pi e \right)  - 2 \log k) = 0$.
Putting this together,
\begin{eqnarray*}
\lim_{k \rightarrow \infty} k^{-2} \cdot \log \left(\frac{v_{2k^2} k^{k^2}}{ c_k} \right) &= & \lim_{k \rightarrow \infty} \left ( k^{-2} \cdot \log v_{2k^2} + \log k - k^{-2} \cdot \log c_k \right) \\ & = & \lim_{k \rightarrow \infty} \log(\pi e) - 2 \log k + \log k -(\log(\pi e^{3/2}) - \log k) \\ & = &- \frac{1}{2}.\\
\end{eqnarray*}
\end{proof}

\begin{lemma} For any $0 < \lambda_1 < \cdots < \lambda_k$ with $\sum_{l=1}^k \lambda_l^p < r^p k$ and $1 < d < k$
\begin{eqnarray*}
\left (\Pi_{1 \leq i < j \leq k} (\lambda_i^2 - \lambda_j^2)^2 \right)\lambda_1 \cdots \lambda_k & \leq &  \left (\frac{r^p \cdot k}{d} \right)^{d(k-d) \cdot \frac{4}{p}} \cdot \\ & & \left (\Pi_{1 \leq i < j \leq k-d} (\lambda_i^2 - \lambda_j^2)^2 \right)\lambda_1 \cdots \lambda_d \cdot \\ & &  \left (\Pi_{k-d +1 \leq i < j \leq k} (\lambda_i^2 - \lambda_j^2)^2 \right)\lambda_{d+1} \cdots \lambda_k.\\
\end{eqnarray*}
\end{lemma}

\begin{proof}  It suffices to show that 
\begin{eqnarray*}
\Pi_{1 \leq i \leq k-d < j \leq k} (\lambda_i^2 - \lambda_j^2)^2 & \leq & \left (\frac{r^p \cdot k}{d} \right)^{d(k-d) \cdot \frac{4}{p}}. \\
\end{eqnarray*}
But as in \cite{sr}, the am-gm inequality yields
\begin{eqnarray*}
\Pi_{1 \leq i \leq k-d < j \leq k} (\lambda_i^2 - \lambda_j^2)^2 
& \leq & \Pi_{i=1}^{k-d} (\Pi_{k-d < j \leq k} \lambda_j^p)^{\frac{4}{p}} \\ & \leq & \Pi_{i=1}^{k-d} \left (\frac{1}{d} \sum_{j = k-d+1}^k \lambda_j^p\right)^{d \cdot \frac{4}{p}}  \\
& \leq & \Pi_{i=1}^{k-d} \left (\frac{r^p \cdot k}{d} \right)^{d \cdot \frac{4}{p}} \\
& \leq &  \left (\frac{r^p \cdot k}{d} \right)^{d(k-d) \cdot \frac{4}{p}}. \\
\end{eqnarray*}
\end{proof}

\begin{proposition}  If $\delta \in (0,1/2)$, $\epsilon \in (0,1)$, then there exists a $D_p>0$ dependent on $p$ such that for $k$ sufficiently large
\begin{eqnarray*}
K_{\epsilon}(E(r,p,k,\delta k)) & \leq & \left ( \frac{D_p(r+1)^2}{\epsilon} \right)^{8 \delta^{1/2} k^2}.
\end{eqnarray*}
\end{proposition}

\begin{proof}  Straightforward manipulations of the definitions shows that for any $m \in \mathbb N$ and $r >0$, $\text{vol}(B_{tr,m,2,r})  = v_{2m^2} \cdot (r^2m)^{m^2}$.  Combining Proposition A.4 with St. Raymond's change of variables formula and this formula gives,
\begin{eqnarray*}
& & \text{vol}(\theta(D(rk^{1/p},p,k,d,\epsilon k^{1/2})) \\ & = & c_k \int_{s(D(r k^{1/p},p,k,d,\epsilon k^{1/2}))}  \left (\Pi_{1 \leq i < j \leq k} (\lambda_i^2 - \lambda_j^2)^2 \right)\lambda_1 \cdots \lambda_k \, d\lambda_1 \cdots d\lambda_k\\
                                                    & = & c_k \int_{\substack {0 \leq \lambda_1 < \cdots < \lambda_k\\ \sum_{j=1}^{k-d} \lambda_j^2 < \epsilon^2 k \\ \sum_{j=1}^k \lambda_j^p \leq r^p k}}  \left (\Pi_{1 \leq i < j \leq k} (\lambda_i^2 - \lambda_j^2)^2 \right)\lambda_1 \cdots \lambda_k \, d\lambda_1 \cdots d\lambda_k\\
                                                    & \leq & c_k \cdot  \left (\frac{r^p \cdot k}{d} \right)^{d(k-d) \cdot \frac{4}{p}} \cdot \\ 
                                                     & & \int_{\substack {0 \leq \lambda_1 < \cdots < \lambda_{k-d} \\ \sum_{j=1}^{k-d} \lambda_j^2 < \epsilon^2 k}}  \left (\Pi_{1 \leq i < j \leq k-d} (\lambda_i^2 - \lambda_j^2)^2 \right)\lambda_1 \cdots \lambda_{k-d} \, d\lambda_1 \cdots d\lambda_{k-d}  \cdot \\
                                                    & &  \int_{\substack {0 \leq \lambda_{k-d+1} < \cdots < \lambda_{k} \\ \sum_{j=k-d+1}^{k} \lambda_j^p < r^p k}}  \left (\Pi_{k-d < i < j \leq k} (\lambda_i^2 - \lambda_j^2)^2 \right)\lambda_{k-d+1} \cdots \lambda_{k} \, d\lambda_{k-d+1} \cdots d\lambda_{k} \\
                                                    & = & c_k \cdot \left (\frac{r^p \cdot k}{d} \right)^{d(k-d) \cdot \frac{4}{p}} \cdot \frac{\text{vol}(B_{tr, k-d, 2, \epsilon(k/(k-d))^{1/2} })}{c_{k-d}} \cdot \frac{\text{vol}(B_{tr,d,p,r\cdot (k/d)^{1/p}})}{c_d} \\
                                                     & \leq & c_k \cdot  \left (\frac{r^p \cdot k}{d} \right)^{d(k-d) \cdot \frac{4}{p}} \cdot \frac{v_{2(k-d)^2}}{c_{k-d}} \cdot (\epsilon^2 k)^{(k-d)^2}  \cdot \frac{(C_{2,p})^{d^2} \text{vol}(B_{tr,d,2,1})}{c_d} \cdot r^{2d^2} \cdot \left(\frac{k}{d}\right)^{\frac{2d^2}{p}} \\
                                                     & \leq & (C_{2,p})^{d^2} \left (\frac{r^p \cdot k}{d} \right)^{d(k-d) \cdot \frac{4}{p}} \left(\frac{c_k}{c_{k-d} c_d}\right) (v_{2(k-d)^2} v_{2d^2}) \epsilon^{2(k-d)^2} r^{2d^2} k^{(k-d)^2} d^{d^2} \left(\frac{k}{d}\right)^{\frac{2d^2}{p}}. \end{eqnarray*}
                                                     
It follows from Proposition 2.1, Proposition A.2, and the above computations that for fixed $\delta \in (0,1/2)$ and $1 \leq d \leq \delta k$,
\begin{eqnarray*}
& & K_{\epsilon}(E(r,p,k,d)) \\ & \leq & \text{vol}(\mathcal N_{\epsilon/2}(E(r,p,k,d))) \cdot  \text{vol}(B_{tr,k,2,\epsilon/2})^{-1}\\
& \leq & \text{vol}(\theta(D(Q_p(r+\epsilon/2)k^{1/p},p,k,d,(\epsilon k^{1/2})/2)) \cdot \text{vol}(B_{tr,k,2,\epsilon/2})^{-1}\\
                                                & \leq & (C_{2,p})^{d^2} \left (\frac{Q^p_p(r+1)^p \cdot k}{d} \right)^{4d(k-d)/p} \left(\frac{c_k}{c_{k-d} c_d}\right) (v_{2(k-d)^2} v_{2d^2}) \left(\frac{\epsilon}{2}\right)^{2(k-d)^2}  (Q_p(r+1))^{2d^2} \cdot \\ & & k^{(k-d)^2} \cdot d^{d^2} \cdot \left(\frac{k}{d}\right)^{\frac{2d^2}{p}} \cdot ((\epsilon/2)^{2k^2}v_{2k^2} k^{k^2})^{-1} \\
                                               & \leq &(Q_p(r+1))^{8dk} \cdot (C_{2,p})^{d^2} \cdot 2^{4dk} \cdot\epsilon^{-4kd} \cdot \left(\frac{c_k}{v_{2k^2} \cdot k^{k^2}}\right) \cdot \left(\frac{v_{2(k-d)^2} \cdot (k-d)^{(k-d)^2}}{c_{k-d}} \right) \cdot \left(\frac{v_{2d^2} \cdot d^{d^2}}{c_{d}} \right) \cdot \\
                                               & & \left (\frac{k}{d}\right)^{\frac{4dk}{p}}\\
                                                & \leq & \left(\frac{2Q_p^2(r+1)^2C_{2,p}}{\epsilon} \right)^{4 d k} \cdot \left(\frac{c_k}{v_{2k^2} \cdot k^{k^2}}\right) \cdot \left(\frac{v_{2(k-d)^2} \cdot (k-d)^{(k-d)^2}}{c_{k-d}} \right) \cdot \left(\frac{v_{2d^2} \cdot d^{d^2}}{c_{d}} \right) \cdot \left(\frac{k}{k-d} \right)^{(k-d)^2}\cdot \\ & & \left (\frac{k}{d}\right)^{\frac{4dk}{p}}. \\
\end{eqnarray*}
 
\noindent Now substitute $d = \delta k$ with $\delta \in (0,1/2)$ in the above, apply $k^{-2} \log$, a $\limsup_{k \rightarrow \infty}$, and Lemma A.3 to arrive at:
\begin{eqnarray*}
\limsup_{k \rightarrow \infty} k^{-2} \cdot \log \left( K_{\epsilon}(E(r,p,k,\delta k)) \right) & \leq & 4\delta \cdot \left (\log(2Q_p^2(r+1)^2C_{2,p}) + |\log \epsilon| \right) + \\ & & \frac{1}{2} - \frac{(1 - \delta)^2}{2} - \frac{\delta^2}{2} + (1-\delta)^2 |\log(1-\delta)| + 4 \delta | \log \delta| \\ 
& \leq &  4\delta \cdot \left (\log(2Q_p^2(r+1)^2C_{2,p}) + |\log \epsilon| \right ) + 3\delta + 4 \delta |\log \delta| \\
& = &  4\delta \cdot \left(\log(2Q_p^2(r+1)^2C_{2,p}e) + |\log \epsilon| \right) + 4 \delta |\log \delta| \\
& \leq &  4\delta^{1/2} \cdot \delta^{1/2} \left (\log(2Q_p^2(r+1)^2C_{2,p}e) + |\log \epsilon| \right) + 4 \delta^{1/2} \cdot \delta^{1/2} |\log \delta| \\
& \leq &  4\delta^{1/2} \cdot \left (\log(2Q_p^2(r+1)^2C_{2,p}e) + |\log \epsilon| \right) + 4 \delta^{1/2}  \\
& \leq &  4\delta^{1/2} \cdot \left (\log(2Q_p^2C_{2,p}e^2(r+1)^2) + |\log \epsilon| \right). \\
\end{eqnarray*}
\noindent Thus, for $\delta \in (0,1/2)$, $1> \epsilon >0$, and $k$ sufficiently large,
\begin{eqnarray*}
K_{\epsilon}(E(r,p,k,\delta k)) & \leq & \left ( \frac{4Q_p^2C_{2,p}e^{2}(r+1)^2}{\epsilon} \right)^{8 \delta^{1/2} k^2}.
\end{eqnarray*}
Set $D_p = 4Q_p^2C_{2,p}e^2$.
\end{proof}

\end{document}